\documentclass{amsart}
\usepackage{HMStyle}
\newcommand{\f}{f}

\begin{document}
\title[Deformations, local freeness, and base change]{Deformations, local freeness, and base change for \\ higher Du Bois singularities}
\author{Haoming Ning}
\date{October 5, 2026}
\thanks{Supported in part by NSF Grant DMS-2100389} 
\address{University of Washington, Department of Mathematics, Seattle, WA 98195, USA}
\email{hning99@uw.edu}

\begin{abstract}
We prove that strict higher Du Bois singularities are invariant under small deformations. Using this, we prove a base change theorem for the relative Du Bois complex with strict higher Du Bois fibers, answering a question of Kov\'acs--Taji. We exhibit failures of deformation invariance and base change for $1$-Du Bois fibers, showing that the strictness condition is essentially sharp. We also generalize the local-freeness theorem of Friedman--Laza beyond the local complete intersection setting for families over an arbitrary base, and show the unobstructedness of deformations of strict-$1$-Du Bois canonical Calabi-Yau varieties.
\end{abstract}

\maketitle

\setcounter{tocdepth}{1}
\tableofcontents

\section{Introduction}

Du Bois singularities \cite{DB81} arise from Deligne's mixed Hodge theory as the class of singularities for which the constant sheaf computes the cohomology of the structure sheaf, and they sit at the center of vanishing theorems and the moduli theory of higher-dimensional varieties. For each $p$, the Du Bois complex carries a graded piece $\underline{\Omega}^p_X$ together with a natural map $\Omega^p_X \to \underline{\Omega}^p_X$ from the sheaf of K\"ahler $p$-forms. Du Bois singularities are precisely those for which this map is a quasi-isomorphism when $p=0$. Refining $p=0$ to a range $p \le m$ yields the higher Du Bois conditions. The theory is well understood in the local complete intersection (lci) setting \cite{MOPW21,JKSY21,FL22,MP25} and is studied in general by \cite{SVV23,Kov25,NN25}. In general, the theory splits into several genuinely distinct conditions:
\[
\text{strict-$m$-Du Bois} \stackrel{\rm lci}{\Longleftrightarrow} \text{$m$-Du Bois} \Longrightarrow \text{weakly-$m$-Du Bois} \Longrightarrow \text{pre-$m$-Du Bois}.
\]
See Definition \ref{def_higher_DB} and \cite{SVV23,Kov25,NN25} for more details.

A natural way to study these singularities is in families. In this setting, three important questions arise.

\smallskip
{\bf Base change}. For a dominant morphism $\f:X\to B$ to a smooth curve and a fiber $X_b$, Kov\'acs--Taji asked whether the relative Du Bois complex \emph{base changes} to the absolute Du Bois complex of the fiber. In other words, when does
\[
\uOmega^{p,-}_{X/B}\otimes^{\tL}_X \scrO_{X_b}\qis \uOmega^p_{X_b}
\]
hold? $\uOmega^{p,-}_{X/B}$ denotes the relative Du Bois complex constructed in \cite{Kov96, Kov97, KT25}. It has numerous applications to families of singular varieties; see also \cite{Kov02, KT23, KT24}. An affirmative answer means, in particular, that the relative complex faithfully records the variation of Hodge-theoretic information in a family. Base change fails in general \cite{JK25}, making it important to identify geometric conditions on a special fiber under which it nevertheless holds.

\smallskip
{\bf Deformations}. Given a Cartier divisor $Z\subseteq X$, when does a higher Du Bois condition on $Z$ imply the corresponding condition on $X$ near $Z$? This is the local form of asking whether the singularity class is preserved under small deformations. From the perspective of moduli theory, such stability is considered a fundamental property of a good singularity class.

\smallskip
{\bf Local freeness}. For a flat proper family $\f: X \to B$, when are the higher direct image sheaves $R^i\f_*\Omega^p_{X/B}$ locally free and compatible with base change? For $p=0$, this is a property of families with Du Bois fibers \cite{DJ74, DB81}, with important consequences for compact moduli spaces of varieties of general type; see \cite[\S 2.5]{Kol23}. In higher degree, local freeness controls the variation of the Hodge-theoretic invariants of the fibers. It is shown to hold by \cite{FL22} in the lci setting over arbitrary base $B$, and the general case remains open.

\smallskip

It is generally recognized that a good notion of higher Du Bois singularities should satisfy all three of these properties. The main theorems of this article show that these conditions are indeed met for \emph{strict-$m$-Du Bois} singularities, which is the most direct, if na\"ive, generalization beyond the lci setting. We also establish sharpness, in the sense that weakening it to $m$-Du Bois makes base change and deformations fail for every $m\ge 1$; see Section \ref{section_example}, specifically Remark \ref{rem_bc_fails}, Proposition \ref{prop_bc_fail_summary}, \ref{prop_ex_deform} and Corollary \ref{cor_fail_bc_1DB}. In light of these failures, strict-$m$-Du Bois stands out as the most suitable among all the existing notions, despite being the most restrictive.

A principle of this article is that, for families over a smooth curve, these three questions are intimately connected. The bridge between them is an alternative construction of the relative Du Bois complex: $\uOmega^{p,+}_{X/B}$. We call it the \emph{right relative Du Bois complex} and refer to the original one, denoted $\uOmega^{p,-}_{X/B}$, as the \emph{left relative Du Bois complex}; see Theorem \ref{thm_2relDB}. In general, the right complex behaves very differently from the original left complex. The new version admits a more direct comparison with the absolute Du Bois complex of a fiber and is more likely to satisfy base change. For precise statements, see Propositions \ref{prop_left_implies_right} and \ref{prop_bc_fail_summary}.

\smallskip
\subsection{Base change}

Our first main result answers the question of Kov\'acs--Taji on base change of the left relative Du Bois complex for strict-$m$-Du Bois fibers.

\begin{theorem}[$=$ Theorems \ref{thm_basechange_strict}, \ref{thm_basechange_strict_all}]\label{thm_basechange_intro}
Let $\f:X\to B$ be a dominant morphism from a complex variety $X$ to a smooth curve $B$, $b\in B$.
\begin{enumerate}
\item Suppose $X_b$ has strict-$m$-Du Bois singularities for all $m\in \bbN$. Then $X_b$ satisfies base change for the relative Du Bois complex, that is, for all $p\ge 0$,
\[
\uOmega^{p, +}_{X/B}\otimes^{\tL}_X\scrO_{X_b} \qis \uOmega^{p, -}_{X/B}\otimes^{\tL}_X\scrO_{X_b} \qis \uOmega^p_{X_b}.
\]
\item Suppose $X_b$ has strict-$m$-Du Bois singularities for some $m\in \bbN$. Then $X_b$ satisfies right-$m$-base change (Definition \ref{def_bc}), that is, for all $0\le p\le m$,
\[
\uOmega^{p, +}_{X/B}\otimes^{\tL}_X\scrO_{X_b} \qis \uOmega^p_{X_b}.
\]
\end{enumerate}
\end{theorem}

The first statement gives base change for the original relative Du Bois complex of \cite{Kov96,KT25}. Base change was previously known for general members of a family \cite{JK25}. The theorem identifies a concrete class of singularities on the special fibers for which it continues to hold. We also recover \cite[Theorem 1.3]{JK25} as Corollary \ref{cor_basechange_general}. The second fixed-$m$ statement for the right complex is stronger for applications, as it retains the partial graded information without requiring strict-$m$-Du Bois in every degree.

The result is not confined to the lci setting. As a complement, we note that there are many examples of strict-$m$-Du Bois varieties with $m>0$ that are not lci, for example, determinantal varieties and cones over $\Gr(2, n)$; see \cite[Example 4.5]{CDO26} and the upcoming work of Kim--Raychaudhury--Perlman.

Studying base-change behavior also reveals that the notion of strict-$m$-Du Bois singularities is the most natural, in the sense of the following proposition.

\begin{proposition}[$=$ Proposition \ref{prop_right_bc_fails}]
Let $\f: X\to B$ be a dominant morphism from a variety $X$ to a smooth curve $B$, $b\in B$. Suppose $X$ is smooth, and suppose $X_b$ is strict-$(m-1)$-Du Bois. Then 
\begin{enumerate}
\item $X_b$ satisfies right-$m$-base change (Definition \ref{def_bc}) if and only if $X_b$ is strict-$m$-Du Bois;

\item $X_b$ satisfies left-$m$-base change (Definition \ref{def_bc}) only if $X_b$ is strict-$m$-Du Bois.
\end{enumerate}
\end{proposition}

\smallskip
\subsection{Deformations}

The key input into the base-change theorem is the following unconditional deformation result for strict-$m$-Du Bois singularities.

\begin{theorem}[$=$ Theorem \ref{thm_deform_strict}]\label{thm_deform_strict_intro}
Let $Z\subseteq X$ be a Cartier divisor in a complex variety $X$. Suppose $Z$ has strict-$m$-Du Bois singularities. Then $X$ has strict-$m$-Du Bois singularities near $Z$.
\end{theorem}

The above theorem relies on an extension of the higher Kov\'acs--Schwede injectivity theorem \cite{Kov25} for K\"ahler differentials, which is of independent interest. We note that this result is obtained by Chen--Dirks--Olano in their upcoming preprint \cite{CDO26}. We instead show this through an independent method extending that of \cite{Kov25}.

\begin{theorem}[$=$ Theorem \ref{thm_kahler_injectivity}]\label{thm_kahler_injectivity_intro}
Let $X$ be a variety with strict-$(m-1)$-Du Bois singularities. Then
\[
h^k(\bbD_X(\uOmega^p_X))\to h^k(\bbD_X(\Omega^p_X))
\]
is injective for all $k\in \bbZ$ and $p\le m$.
\end{theorem}

On the other hand, base change gives a direct comparison between the Du Bois complex of the total space and the fiber. This gives the following deformation statements for the various other flavors of higher Du Bois singularities, assuming base change.

\begin{theorem}[$=$ Propositions \ref{prop_deform_pre}, \ref{prop_deform_weak}, \ref{prop_deform_codim}]\label{thm_deform_combined}
Let $Z\subseteq X$ be a Cartier divisor in a complex variety $X$. Suppose $Z$ satisfies either left- or right-$m$-base change (cf. Definition \ref{def_bc_divisor}).
\begin{enumerate}
\item If $Z$ is pre-$m$-Du Bois, then $X$ is pre-$m$-Du Bois near $Z$.
\item If $Z$ is weakly-$m$-Du Bois, then $X$ is weakly-$m$-Du Bois near $Z$.
\item If $Z$ is $m$-Du Bois, then $X$ is $m$-Du Bois near $Z$.
\end{enumerate}
\end{theorem}

For sharpness, we show that deformations fail even for the next strongest notion, $m$-Du Bois.

\begin{proposition}[$=$ Proposition \ref{prop_ex_deform}]
There exists a variety $X$ and a Cartier divisor $Z\subseteq X$ such that $Z$ is $1$-Du Bois, but $X$ is not pre-$1$-Du Bois.
\end{proposition}

\smallskip
\subsection{Local freeness.} 

We extend the theorem of Friedman--Laza outside of the lci case.

\begin{theorem}[$=$ Theorem \ref{thm_local_free}]
Let $\f:X \to B$ be a flat proper family of finite type $\bbC$-schemes, and $b\in B$ a closed point. Suppose $X_b$ has strict-$m$-Du Bois singularities for some $m\in \bbN$. Then $R^i\f_*\Omega^p_{X/B}$ is locally free and compatible with base change for all $i\ge 0$ and $0\le p\le m$, in a neighborhood of $b$.
\end{theorem}

The new input is a flatness result of relative K\"ahler differentials, Lemma \ref{lem_rel_Kahler_flat}. Following the principle of \cite{SVV23, Kov25, NN25}, the zeroth cohomology sheaf of the Du Bois object often behaves much better than K\"ahler differentials. We show that this principle also holds for the relative Du Bois complex, see Section \ref{section_localfree} for more details.

The local-freeness theorem also gives an application to the deformation theory of singular Calabi--Yau varieties. Following Friedman--Laza \cite[Corollary 1.5]{FL22}, we obtain unobstructedness without the lci assumption.

\begin{corollary}[$=$ Corollary \ref{cor_unobstructed}]
Let $X$ be a canonical Calabi--Yau variety (\cite[Definition 4.4]{FL22}). Suppose $X$ has strict-$1$-Du Bois singularities. Then the deformation functor $\mathbf{Def}(X)$ is unobstructed.
\end{corollary}

\smallskip

We remark that the concurrent work of Chen--Dirks--Olano (\cite{CDO26}) proposes two new definitions of higher Du Bois singularities: strongly-$m$-Du Bois and cohomologically-$m$-Du Bois. They show deformations and constancy of Hodge numbers for these two notions. These provide new candidates for a robust theory of higher Du Bois singularities outside the lci setting. They also prove independent generalizations of the Kov\'acs--Schwede injectivity theorem, including Theorem \ref{thm_kahler_injectivity_intro}.

\smallskip
\subsection*{Paper outline}

Section \ref{section_preliminary} recalls the relative Du Bois complex, various definitions of higher Du Bois singularities, and records some basic facts and constructions. Section \ref{section_2relDB} constructs the alternative right relative Du Bois complex (Construction \ref{cons_relDB}) and studies its basic properties. Section \ref{section_basechange} studies in depth the questions of right-$m$-base change (Definition \ref{def_bc}) and contains a key technical input, Theorem \ref{thm_full_diagram}.

Section \ref{section_inj_thm} establishes a Kov\'acs--Schwede-type injectivity theorem for K\"ahler differentials similar to \cite{Kov25}. Section \ref{section_deform_strict} then proves the deformations of strict-$m$-Du Bois singularities. The questions of base change and deformations are connected by Section \ref{section_deformation}, which uses deformations to deduce right-$m$-base change, then derives the other deformation statements. As an application, we show the local-freeness theorem and constancy of Hodge numbers for strict-$m$-Du Bois families in Section \ref{section_localfree}. We conclude with Section \ref{section_example}, where we construct explicit examples showing failure of base change without the strict-$m$-Du Bois assumption, establishing sharpness.

\smallskip
\subsection*{Acknowledgements}

The author would like to thank S\'andor Kov\'acs for many helpful discussions and for detailed feedback that significantly improved this manuscript. The author would also like to thank Brad Dirks for sharing the preprint \cite{CDO26}, pointing out a mistake in an earlier version of this paper and for discussions that inspired Example \ref{ex_deform}.

\smallskip
\subsection*{AI Disclosure}

No result in this manuscript is obtained through the use of generative AI. After its writing, the author was granted access to OpenAI Models through the ChatGPT for Academic Researchers program, and used ChatGPT-6 Astra as an adversarial reader, leading to four particular simplifications in \ref{lem_sm_diagram_derived}, \ref{lem_logK_cover_local_decomp}, \ref{lem_logK_residue}, \ref{cons_deform_cover}, and other minor modifications. The author takes full responsibility for the mathematical contents of this manuscript.

Additionally, Example \ref{ex_deform} is inspired partially by (human discussions based on) an example (a hypersurface singularity exhibiting failure of deformations for pre-$1$-Du Bois) shared with me by Brad Dirks, which, to my knowledge, was obtained by an earlier model of ChatGPT.

\medskip

\section{Preliminaries}\label{section_preliminary}

Throughout this paper, varieties are understood to be reduced (possibly reducible) schemes of finite type over $\bbC$. A family $\f: X\to B$ means a flat morphism of finite type. Whenever a fiber of a morphism to a smooth curve is used as a Cartier divisor, we assume that every irreducible component of the total space dominates the curve.

For convenience, we recall the definitions of higher Du Bois singularities in the most general setting. We refer the reader to \cite{SVV23,Kov25,NN25} for detailed treatments.

\begin{definition}\label{def_higher_DB}
Let $X$ be a reduced scheme of finite type over $\bbC$ and let $m\in \bbN$. We say
\begin{enumerate}
\item $X$ has pre-$m$-Du Bois singularities if for each $p\le m$,
\[
h^0(\uOmega^p_X)\qis \uOmega^p_X;
\]

\item $X$ has weakly-$m$-Du Bois singularities if $X$ is seminormal, pre-$m$-Du Bois, and $h^0(\uOmega^p_X)$ is $S_2$ for each $1\le p\le m$;

\item $X$ has $m$-Du Bois singularities if $X$ is weakly-$m$-Du Bois, $\codim_X\Sing X\ge 2m+1$, and $h^0(\uOmega^p_X)$ is reflexive for each $1\le p\le m$;

\item $X$ has strict-$m$-Du Bois singularities if for each $p\le m$,
\[
\Omega^p_X\qis \uOmega^p_X.
\]
\end{enumerate}
\end{definition}

We also recall the properties of the relative Du Bois complex. See \cite{Kov96, Kov97, KT25} for detailed constructions of this filtered complex. The notation has been adapted to be consistent with the rest of the paper.

\begin{theorem}[\protect{\cite[\S 2]{KT25}}]\label{def_relDB}
Let $\f:X\to B$ be a dominant morphism of complex varieties over a smooth curve $B$, and $\dim X=n+1$. There exists $(\uOmega^{\kdot,-}_{X/B}, E^p)\in D^b_{\rm filt}(X)$ such that
\begin{enumerate}
\item the associated graded $\uOmega^{p,-}_{X/B}= \Gr^{p}_{\rm filt} \uOmega^{\kdot,-}_{X/B}[p] \in D^b_{\rm coh}(X)$;

\item $\uOmega^{p, -}_{X/B}= 0$ for $p>n$;

\item there is a natural triangle
\[
\uOmega^{p-1, -}_{X/B}\otimes f^*\omega_B \to \uOmega_X^p \to \uOmega_{X/B}^{p,-} \xrightarrow{+1}
\]
for each $p$;

\item if $f$ is smooth, then there is a filtered quasi-isomorphism
\[
\Omega_{X/B}^{\kdot,-} \qis \uOmega_{X/B}^\kdot.
\]
\end{enumerate}
\end{theorem}

In Section \ref{section_2relDB}, we will recall its explicit construction and utilize it to define a new relative Du Bois complex; see Theorem \ref{thm_2relDB}.

For many of our purposes, the following hyperresolution construction will be used. We fix the notation here for easy reference. For details on cubical hyperresolutions, see \cite[\S 10.6]{Kol13}.

\begin{construction}\label{eqn_hyperres}
Suppose $Z\subseteq X$ is a reduced closed subscheme. Choose cubical hyperresolutions $\eta_\kdot: \widetilde{Z}_\kdot \to Z$, $\pi_\kdot: X_\kdot \to X$ such that $\eta_\kdot$ factors through $\pi_\kdot$. Denoting $Z_\kdot = X_\kdot \times_X Z$, we may choose the resolution of $X$ such that each $Z_\alpha \subseteq X_\alpha$ has snc support (or is equal to $X_\alpha$), and 
each $\widetilde{Z}_\alpha\to Z_\alpha$ is a closed immersion with $\widetilde{Z}_{\alpha}$ either empty or a smooth divisor. We have a diagram
\begin{equation*}
\begin{tikzcd}
\widetilde{Z}_\kdot \ar[r] \ar[dr, "\eta_\kdot"] & Z_\kdot \ar[r] \ar[d] & X_\kdot \ar[d, "\pi_\kdot"] \\
& Z \ar[r] & X
\end{tikzcd}
\end{equation*}
\end{construction}

Here we collect a lemma regarding strict-$m$-Du Bois singularities and products.

\begin{lemma}\label{lem_product}
Let $X, Y$ be varieties, $m\in \bbN$. Let $p_X:X\times Y\to X$, $p_Y:X\times Y\to Y$ be the projections. Then 
\[
\uOmega^p_{X\times Y}\qis \bigoplus_{a+b=p} Lp_X^*\uOmega^a_X \otimes^{\tL} Lp_Y^*\uOmega^b_Y.
\]
Now assume $Y$ is smooth. Then $X\times Y$ is strict-$m$-Du Bois if and only if $X$ is strict-$m$-Du Bois.
\end{lemma}

\begin{proof}
We may choose hyperresolutions $\pi_{X_\kdot}: X_\kdot \to X, \pi_{Y_\kdot}: Y_\kdot \to Y$ such that $X_\kdot, Y_\kdot$ are cubical schemes over the same cubical diagram. In this case, $\pi_\kdot: X_\kdot\times Y_\kdot\to X\times Y$ gives a cubical hyperresolution of $X\times Y$. On each piece $\alpha$ of the hyperresolution, K\"unneth decomposition for the K\"ahler differentials reads
\[
\Omega^p_{X_\alpha \times Y_\alpha} \isom \bigoplus_{a+b=p} p_{X_\alpha}^*\Omega^a_{X_\alpha} \otimes p_{Y_\alpha}^* \Omega^b_{Y_\alpha}.
\]
Now, $X_\alpha$, $Y_\alpha$ are trivially $\Tor$-independent over $\bbC$, hence sheaves pulled back to $X_\alpha\times Y_\alpha$ from different factors are $\Tor$-independent. Note also that each projection from the product is flat. Therefore, K\"unneth formula, $\Tor$-independence and flat base change together give a sequence of quasi-isomorphisms
\begin{align*}
\uOmega^p_{X\times Y} \qis R(\pi_\kdot)_*\Omega^p_{X_\kdot \times Y_\kdot} \qis & \bigoplus_{a+b=p} R(\pi_\kdot)_* p_{X_\kdot}^*\Omega^a_{X_\kdot} \otimes^{\tL} R(\pi_\kdot)_*p_{Y_\kdot}^* \Omega^b_{Y_\kdot} \\
\qis & \bigoplus_{a+b=p} Lp_X^*R(\pi_{X_\kdot})_* \Omega^a_{X_\kdot} \otimes^{\tL} Lp_Y^* R(\pi_{Y_\kdot})_* \Omega^b_{Y_\kdot} \\
\qis & \bigoplus_{a+b=p} Lp_X^*\uOmega^a_X \otimes^{\tL} Lp_Y^*\uOmega^b_Y.
\end{align*}
This gives the first formula. By construction, the decomposition is compatible with the K\"unneth decomposition of $\Omega^p_{X\times Y}$ under the natural map $\Omega^p_{X\times Y}\to \uOmega^p_{X\times Y}$.

When $Y$ is smooth, $\uOmega^b_Y\qis \Omega^b_Y$ is a locally free sheaf. Thus, if $X$ is strict-$m$-Du Bois, every term in the decomposition of $\uOmega^p_{X\times Y}$ is identified with the K\"unneth decomposition of $\Omega^p_{X\times Y}$, and we see that $X\times Y$ is strict-$m$-Du Bois. Conversely, if $X\times Y$ is strict-$m$-Du Bois, examining the component where $a=p\le m, b=0$ gives $Lp^*_X\uOmega^p_X\qis p^*_X\Omega^p_X$. This shows that $\uOmega^p_X\qis \Omega^p_X$ since $p_X$ is faithfully flat and $Lp^*_X(\Cone (\Omega^p_X \to \uOmega^p_X))=0$. This concludes the proof.
\end{proof}

We will also need the following consequence of Nakayama's lemma.

\begin{lemma}\label{lem_Nakayama}
Let $Z\subseteq X$ be a Cartier divisor. Let $\scrF\in D^b_{\rm coh}(X)$, and suppose that
\[
h^i(\scrF \otimes^{\tL}_X \scrO_Z) = 0.
\]
Then $h^i(\scrF) = 0$ near $Z$. In particular, $h^i(\scrF)|_Z=0$.
\end{lemma}

\begin{proof}
Consider the short exact sequence
\[
0\to \scrO_X(-Z) \to \scrO_X \to \scrO_Z \to 0,
\]
which gives the triangle
\[
\scrF(-Z) \to \scrF \to \scrF \otimes^{\tL}_X \scrO_Z \xrightarrow{+1}.
\]
Taking cohomology and applying the hypothesis gives
\[
h^i(\scrF(-Z)) \twoheadrightarrow h^i(\scrF) \to h^i(\scrF\otimes^{\tL}_X \scrO_Z) = 0.
\]
Let $x\in Z \subseteq X$ and $t\in \scrO_{X,x}$ a local equation for $Z$. Since localization is exact, after localizing at $x$,
\[
h^i(\scrF(-Z)_x) \isom h^i(\scrF_x) \xrightarrow{\cdot t} h^i(\scrF_x) \to 0
\]
remains surjective. Using the hypothesis that $h^i(\scrF)$ is coherent, we see by Nakayama's lemma that $h^i(\scrF_x)\isom h^i(\scrF)_x=0$ for $x\in Z$. The claim now follows.
\end{proof}

\subsection{Torsion-free K\"ahler differentials}

The following lemmas will be useful for deformations.

\begin{lemma}\label{lem_tf_ses}
Let $\scrF_1, \scrF_2, \scrF_3$ be quasi-coherent sheaves on a scheme $X$. Suppose $\scrF_1, \scrF_3$ are torsion-free, and
\[
0\to \scrF_1\to \scrF_2\to \scrF_3\to 0
\]
is short exact. Then $\scrF_2$ is torsion-free.
\end{lemma}

\begin{proof}
Denote by $\Tors\scrF_2$ the torsion subsheaf of $\scrF_2$. Then, as subsheaves of $\scrF_2$, we have $\scrF_1\cap \Tors\scrF_2=0$ since $\scrF_1$ is torsion-free. Therefore, $\Tors \scrF_2$ descends to a natural subsheaf of the quotient
\[
\Tors \scrF_2 \hookrightarrow \scrF_2/\scrF_1 \isom \scrF_3.
\]
Thus we conclude $\Tors \scrF_2=0$ since $\scrF_3$ is also torsion-free, as desired.
\end{proof}

\begin{lemma}\label{lem_conormal}
Let $Z\subseteq X$ be a reduced effective Cartier divisor. Suppose $\Omega^p_Z$ is torsion-free for all $p\le m-1$. Then we have a short exact sequence
\[
0\to \Omega^{p-1}_Z(-Z) \to \Omega^p_X|_Z \to \Omega^p_Z \to 0
\]
for all $p\le m$.
\end{lemma}

\begin{proof}
Applying \cite[Tag 01CJ]{SP} to the conormal sequence:
\[
\scrO_Z(-Z) \xrightarrow{u} \Omega^1_X|_Z \to \Omega^1_Z \to 0,
\]
we obtain right exact sequences of its exterior powers
\[
\Omega^{p-1}_X|_Z(-Z) \to \Omega^p_X|_Z \to \Omega^p_Z \to 0
\]
for all $p\in \bbN$, where the first map is given locally by $\eta \otimes g\mapsto \eta\wedge u(g)$. We claim that it admits a factorization
\[\begin{tikzcd}
\Omega^{p-1}_X|_Z(-Z) \ar[d, "\gamma_{p-1}", twoheadrightarrow] \ar[dr, "\beta_{p-1}"] &
&
&
\\
\Omega^{p-1}_Z(-Z) \ar[r, dotted, "\alpha_{p-1}"] &
\Omega^p_X|_Z \ar[r] &
\Omega^p_Z \ar[r] &
0
\end{tikzcd}\]
for all $p\in \bbN$. It is sufficient to show that $\beta_{p-1}$ annihilates the kernel of $\gamma_{p-1}$. To this end, we consider the same sequence from \cite[Tag 01CJ]{SP} in one degree lower, tensored with $\scrO_Z(-Z)$:
\[
\Omega^{p-2}_X|_Z(-2Z) \to \Omega^{p-1}_X|_Z(-Z) \xrightarrow{\gamma_{p-1}} \Omega^{p-1}_Z(-Z) \to 0.
\]
Now consider the composition
\[
\Omega^{p-2}_X|_Z(-2Z)\to \Omega^{p-1}_X|_Z(-Z) \xrightarrow{\beta_{p-1}} \Omega^p_X|_Z,
\]
which is locally given by $\eta \otimes g \otimes g' \mapsto \eta \wedge u(g) \wedge u(g')$. Since $\scrO_Z(-Z)$ is a line bundle, we see that this map is zero. This proves the claim that $\beta_{p-1}$ descends to $\alpha_{p-1}$. Furthermore, since $\gamma_{p-1}$ is surjective, we see that the images of $\beta_{p-1}$ and $\alpha_{p-1}$ coincide, so the bottom row remains exact.

By hypothesis $\Omega^{p-1}_Z$ is torsion-free for all $p\le m$, hence so is $\Omega^{p-1}_Z(-Z)$, since $Z$ is Cartier. Now, $\alpha_{p-1}$ is generically injective on $Z$ for all $p\in \bbN$, because the conormal sequence is exact on the smooth locus of $Z$ and $X$. Therefore, $\alpha_{p-1}$ is injective for all $p\le m$.
\end{proof}

\begin{lemma}\label{lem_tf_omega}
Let $Z\subseteq X$ be a reduced effective Cartier divisor. Suppose $\Omega^p_Z$ is torsion-free for all $p\le m$. Then $\Omega^p_X|_Z$ is torsion-free for all $p\le m$. Furthermore, $\Omega^p_X$ is torsion-free and $S_2$ in a neighborhood of $Z$ for all $p\le m$.
\end{lemma}

\begin{proof}
We apply Lemma \ref{lem_tf_ses} to the short exact sequence
\[
0\to \Omega^{p-1}_Z(-Z) \to \Omega^p_X|_Z \to \Omega^p_Z \to 0
\]
from Lemma \ref{lem_conormal}. This gives $\Omega^p_X|_Z$ torsion-free for $p\le m$, as needed. Since $\Omega^p_X$ is generically locally free on $Z$, knowing that $\Omega^p_X|_Z$ is torsion-free implies that $Z$ contains no associated primes of $\Omega^p_X$. Therefore, any local equation of $Z$ is a non-zero-divisor on $\Omega^p_X$ and by depth counting $\Omega^p_X$ is $S_2$ and torsion-free on a neighborhood of $Z$ in $X$, giving the last statement.
\end{proof}

\subsection{Wedge product, residue, and conormal maps}

Here we show a specific construction in the smooth case that will be useful later on.

\begin{lemma}\label{lem_sm_diagram}
Let $X$ be a smooth variety, $Z$ an effective Cartier divisor and $\widetilde{Z}\subseteq Z$ a smooth divisor that is a subscheme of $Z$. For $p\ge 0$, there is a natural morphism $\tau: \Omega^{p-1}_{\widetilde{Z}}(-Z) \to \Omega^p_X|_Z$ on $Z$, such that
\[\begin{tikzcd}[cramped]
\Omega^{p-1}_{\widetilde{Z}}(-Z) \ar[r, "\tau"] \ar[d] &
\Omega^p_X|_Z \ar[d] \\
\Omega^{p-1}_{\widetilde{Z}}(-\widetilde{Z}) \ar[r] &
\Omega^p_X|_{\widetilde{Z}}
\end{tikzcd}\]
commutes with the conormal map $\Omega^{p-1}_{\widetilde{Z}}(-\widetilde{Z}) \to
\Omega^p_X|_{\widetilde{Z}}$ and vertical maps induced by $\scrO_X(-Z)\to \scrO_X(-\widetilde{Z})$, and $\scrO_Z\to \scrO_{\widetilde{Z}}$.
\end{lemma}

\begin{proof}
We first construct $\tau$. Since $\widetilde{Z}$ is smooth, recall the residue sequence (cf. \cite[2.3]{EV92}) on $X$, twisted by $\scrO_X(-Z)$:
\[
0 \to \Omega^p_X(-Z) \to \Omega^p_X(\log \widetilde{Z})(-Z) \to \Omega^{p-1}_{\widetilde{Z}}(-Z) \to 0
\]
On the other hand, we also have the following chain of inclusions of subsheaves:
\[
\Omega^p_X(-Z) \subseteq \Omega^p_X(\log \widetilde{Z})(-{Z}) \subseteq \Omega^p_X(\log \widetilde{Z})(-\widetilde{Z}) \subseteq \Omega^p_X \to \Omega^p_X|_Z
\]
This gives a map $\Omega^p_X(\log \widetilde{Z})(-{Z}) \to \Omega^p_X|_Z$, which kills the subsheaf $\Omega^p_X(-Z)$. Therefore, it factors through the cokernel, which by the residue sequence is precisely
\[
\tau: \Omega^{p-1}_{\widetilde{Z}}(-Z) = \coker \left( \Omega^p_X(-Z) \to \Omega^p_X(\log \widetilde{Z})(-Z)  \right) \to \Omega^p_X|_Z.
\]
It remains to show that the claimed diagram commutes. Since $\widetilde{Z}$ is smooth, recall the conormal sequence
\[
0 \to \Omega^{p-1}_{\widetilde{Z}}(-\widetilde{Z}) \xrightarrow{\sigma} \Omega^p_X|_{\widetilde{Z}} \to \Omega^p_{\widetilde{Z}} \to 0.
\]
As above, the composition from the residue sequence agrees with the canonical map $\Omega^p_X(\log \widetilde{Z})(-\widetilde{Z})\to \Omega^p_X|_{\widetilde{Z}}$. In particular, we see that the conormal map $\sigma$ sits in a larger diagram
\[\begin{tikzcd}
\Omega^p_X(\log \widetilde{Z})(-{Z}) \ar[r] \ar[d] &
\Omega^{p-1}_{\widetilde{Z}}(-Z) \ar[r, "\tau"] \ar[d] &
\Omega^p_X|_Z \ar[d] \\
\Omega^p_X(\log \widetilde{Z})(-\widetilde{Z}) \ar[r] &
\Omega^{p-1}_{\widetilde{Z}}(-\widetilde{Z}) \ar[r, "\sigma"] &
\Omega^p_X|_{\widetilde{Z}},
\end{tikzcd}\]
where the outer square is commutative. It follows that the right square is commutative since the first two horizontal maps are surjective. This construction works as long as $p\ge 0$, with the observation that the $p=0$ case is trivial, as desired.
\end{proof}

For a smooth variety $X$ that is a family over a smooth curve $B$ with $\f:X\to B$, recall the wedge map of differential forms
\[
\wedge_p:\Omega^p_X\otimes \f^*\omega_B \to \Omega^{p+1}_X.
\]
On a fiber $Z:= X_b$ for $b\in B$, there is a canonical identification
\[
\scrO_Z(-Z) \isom \f_b^*(\frakm_b/\frakm_b^2) \isom \f^*\omega_B|_Z.
\]
Throughout this paper, we will fix a local parameter at $b\in B$, with which we can further identify
\[
\scrO_Z(-Z)\isom \f^*\omega_B|_Z \isom \scrO_Z.
\]
It shall be understood that whenever we omit the term of $\scrO_Z(-Z)$ for a fiber $Z=X_b$, it is through this identification with respect to a fixed parameter. Now, restricting to any closed subscheme $\widetilde{Z}\subseteq X_b$, we have a map
\[
\wedge_{p-1}|_{\widetilde{Z}}:\Omega^{p-1}_X|_{\widetilde{Z}} \to \Omega^{p}_X|_{\widetilde{Z}}.
\]

\begin{lemma}\label{lem_sm_wedge}
Let $f:X\to B$ be a dominant morphism from a smooth complex variety $X$ to a smooth curve $B$. Let $Z=X_b$ be the fiber over a point $b\in B$, and $\widetilde{Z}\subseteq Z$ a smooth divisor of $X$ that is a subscheme of $Z$. Then the composition
\[
\Omega^{p-1}_X(-Z)|_Z \to \Omega^{p-1}_{\widetilde{Z}}(-Z) \xrightarrow{\tau} \Omega^p_X|_Z
\]
with $\tau$ as in Lemma \ref{lem_sm_diagram}, agrees with $\wedge_{p-1}|_{\widetilde{Z}}$ after restricting to $\widetilde{Z}$.
\end{lemma}

\begin{proof}
The first map is induced by the restriction of differential forms $\Omega^{p-1}_X|_Z \to \Omega^{p-1}_{\widetilde{Z}}$. Note also that since $Z=X_b$ is a fiber, we have $\scrO_X(-Z)|_Z\isom \scrO_Z$ and hence also $\scrO_X(-Z)|_{\widetilde{Z}}\isom \scrO_{\widetilde{Z}}$.

We need only check that the two maps agree locally on X and B. To this end, choose a local parameter $t$ at $b\in B$, a local equation $g$ for $\widetilde{Z}$, and lastly $h$ such that $gh=\f^* t$ is a local equation for $Z$. Now let $\xi$ be a local section of $\Omega^{p-1}_X$. Note that (up to sign convention)
\[
\wedge_{p-1}(\xi \otimes \f^*dt) = h\, dg \wedge \xi + g\, dh \wedge \xi.
\]
Hence after restricting to $\widetilde{Z}$ where $g=0$, we obtain
\[
\wedge_{p-1}(\xi \otimes \f^*dt)|_{\widetilde{Z}} = h\, dg\wedge \xi|_{\widetilde{Z}}.
\]
On the other hand, using Lemma \ref{lem_sm_diagram} and \cite[2.2]{EV92} for the local form for the residue class, we have
\[
\tau(\bar{\xi} \otimes gh) = \left( \frac{dg}{g} \wedge \xi \right) \otimes gh = h\, dg\wedge \xi.
\]
Therefore the two maps agree after further restriction to $\widetilde{Z}$, as needed.
\end{proof}

\subsection{Derived conormal maps.}

For some of our applications, we will also need a derived version of Lemma \ref{lem_sm_diagram}. We will work with this setup: let $Z\subseteq X$ be a Cartier divisor, both $X, Z$ possibly singular. Let $\pi_\alpha: X_\alpha \to X$ be a morphism with $X_\alpha$ smooth, and $Z_\alpha= X_\alpha \times_X Z$. Let $\widetilde{Z_\alpha}\subseteq Z_\alpha$ be a smooth divisor in $X_\alpha$ that is a closed subscheme of $Z_\alpha$.

\begin{lemma}\label{lem_sm_diagram_derived}
Notation as above. For $p\ge 0$, there is a natural morphism on $X_\alpha$:
\[
\tau_\alpha: \Omega^{p-1}_{\widetilde{Z}_\alpha}\otimes \pi_\alpha^*\scrO_X(-Z) \to \Omega^p_{X_\alpha} \otimes L\pi_\alpha^*\scrO_Z.
\]
Furthermore, together with the map induced by $L\pi^*\scrO_Z \to \scrO_{Z_\alpha}$, the composition
\[
\Omega^{p-1}_{\widetilde{Z}_\alpha}\otimes \pi_\alpha^*\scrO_X(-Z) \xrightarrow{\tau_\alpha} \Omega^p_{X_\alpha} \otimes L\pi_\alpha^*\scrO_Z \to \Omega^p_{X_\alpha}|_{Z_\alpha}
\]
agrees with the map $\tau$ constructed in Lemma \ref{lem_sm_diagram} when $Z_\alpha$ is a divisor. Consequently, if $\widetilde{Z}_\alpha = Z_\alpha$ is smooth, then $\tau_\alpha$ agrees with the usual conormal map $\Omega^{p-1}_{\widetilde{Z}_\alpha} (-\widetilde{Z}_\alpha) \to \Omega^p_{X_\alpha}|_{\widetilde{Z}_\alpha}$.
\end{lemma}

\begin{proof}
The purpose of this lemma is to address the case when $\pi_\alpha(X_\alpha)\subseteq Z$, which can occur in a cubical hyperresolution. When $\pi_\alpha$ is dominant, this would follow from Lemma \ref{lem_sm_diagram} as $\scrO_Z$ and $\pi_\alpha$ are $\Tor$-independent by \cite[Lemma 4.3]{JK25}, that is, $L\pi_\alpha^*\scrO_Z \qis \pi_\alpha^*\scrO_Z \isom \scrO_{Z_\alpha}$. Note that in general, $Z_\alpha$ need not be a divisor on $X_\alpha$. 

Consider the following diagram
\[\begin{tikzcd}
0 \ar[r] &
\pi_\alpha^*\scrO_X(-Z) \ar[r] \ar[d, equal] &
\pi_\alpha^*\scrO_X(-Z) \otimes \scrO_{X_\alpha} (\widetilde{Z}_\alpha) \ar[r] \ar[d] &
\pi_\alpha^*\scrO_X(-Z) \otimes \scrO_{X_\alpha} (\widetilde{Z}_\alpha)|_{\widetilde{Z}_\alpha} \ar[r] \ar[d, dotted, "\sigma_\alpha"] &
0 \\
&
\pi_\alpha^*\scrO_X(-Z) \ar[r] &
\scrO_{X_\alpha} \ar[r] &
L\pi_\alpha^* \scrO_Z. \ar[r, "+1"] &
\,
\end{tikzcd}\]
The first row comes from the restriction sequence associated to $\widetilde{Z}_\alpha$, and the second row is obtained by derived pullback of that associated to $Z$ on $X$, using that $Z_\alpha=X_\alpha\times_X Z$. Since $\widetilde{Z_\alpha}$ is contained scheme theoretically in $Z_\alpha$, we have a factorization 
\[
\pi^*\scrO_X(-Z)\to \scrO_{X_\alpha}(-\widetilde{Z}_\alpha) \to \scrO_{X_\alpha}.
\]
Tensoring the first map with $\scrO_{X_\alpha}(\widetilde{Z}_\alpha)$, we obtain the middle vertical map in the diagram above, and we see that the first square commutes. Now, with $[\pi^*_\alpha \scrO_{X}(-Z) \to \scrO_{X_\alpha}]$ as the specific representative of $L\pi^*_\alpha\scrO_Z$, this gives a canonical map $\sigma_\alpha$.

Now, tensor the conormal map $\Omega^{p-1}_{\widetilde{Z}_\alpha} (-\widetilde{Z}_\alpha) \to \Omega^p_{X_\alpha}|_{\widetilde{Z}_\alpha}$ with $\pi_\alpha^*\scrO_X(-Z) \otimes \scrO_{X_\alpha} (\widetilde{Z}_\alpha)|_{\widetilde{Z}_\alpha}$. After using the projection formula, we have the composition
\[
\tau_\alpha: \Omega^{p-1}_{\widetilde{Z_\alpha}} \otimes \pi_\alpha^*\scrO_X(-Z) \to \Omega^p_{X_\alpha} \otimes \pi_\alpha^*\scrO_X(-Z) \otimes \scrO_{X_\alpha} (\widetilde{Z}_\alpha)|_{\widetilde{Z}_\alpha} \xrightarrow{\Id \otimes \sigma_\alpha} \Omega^p_{X_\alpha} \otimes L\pi_\alpha^*\scrO_Z.
\]
This constructs the $\tau_\alpha$ that we required. If we further compose with $L\pi^*\scrO_Z \to \scrO_{Z_\alpha}$, this is simply taking the zeroth cohomology of $\tau_\alpha$. By construction this agrees with the map in Lemma \ref{lem_sm_diagram}, and the last claim on when $\widetilde{Z}_\alpha = Z_\alpha$ follows.
\end{proof}

\medskip

\section{A tale of two relative Du Bois complexes}\label{section_2relDB}

For the purpose of base change, it is in fact more natural to consider an alternative definition of the relative Du Bois complex. It is similar to the construction in \cite[Theorem-Definition 1.3]{Kov96}, with one key difference: the pieces $\uOmega^p_{X/B}$ are built upward from $p=0$ instead of downward from $p=\dim \f$. For clarity, we will use the same notation as in \emph{loc. cit.} and omit details that are exactly the same.

\begin{construction}\label{cons_relDB}
Let $\f:X\to B$ be a morphism from a complex variety $X$ to a smooth curve $B$. Fix a cubical hyperresolution $\pi_\kdot: X_\kdot \to X$. Then as in \cite[1.3]{Kov96}, let $K^\kdot_p\qis \uOmega^p_X$ be a specific representative of the Du Bois complex, together with a natural wedge map
\[
\wedge_p: K^\kdot_p \otimes \f^*\omega_B \to K^\kdot_{p+1},
\]
with the understanding that $\wedge_p$ is the zero map for $p< 0$ or $p>\dim X$. Denote also $\wedge'_p = \wedge_p\otimes \Id_{\f^*\omega_B}$ and observe that $\wedge_p\circ \wedge'_{p-1}=0$.

Our goal is to define, for each $p\in \bbZ$, a complex $M^\kdot_p$, with morphisms
\begin{align*}
w''_p:& K^\kdot_p\otimes \f^*\omega_B\to M^\kdot_p\otimes \f^*\omega_B, \\
w'_p:& M^\kdot_p \otimes \f^*\omega_B\to K^\kdot_{p+1}
\end{align*}
satisfying (viewed as morphisms in the homotopy category)
\[
\wedge_p=w'_p\circ w''_p \quad \text{and} \quad w''_p\circ \wedge'_{p-1} =0.
\]

For $p<0$, we define $M^\kdot_p=0$ and the morphisms $w'_p, w''_p$ each to be the zero morphism, and see that the required conditions are easily satisfied. 

For $p\ge 0$, suppose by induction this has been done for all $q<p$. Define $M^\kdot_{p} := \Cone(w'_{p-1})$, that is
\[
M^j_p = (M^{j+1}_{p-1}\otimes \f^*\omega_B) \oplus K^j_p
\]
with differentials given by the matrix
\[
d^j_{M^\kdot_p} = \begin{pmatrix}
-d^{j+1}_{M^\kdot_{p-1}}\otimes \Id_{\f^*\omega_B} & 0 \\
w'_{p-1} & d^j_{K^\kdot_p}
\end{pmatrix}.
\]
We can now define
\begin{align*}\label{eqn_relDB_maps}
w''_p := \begin{pmatrix}
0 \\
\Id_{K^\kdot_p}
\end{pmatrix}\otimes \Id_{\f^*\omega_B}:& K^\kdot_p\otimes \f^*\omega_B\to M^\kdot_p\otimes \f^*\omega_B, \\
w'_p := (0, \wedge_p):& M^\kdot_p \otimes \f^*\omega_B\to K^\kdot_{p+1}.
\end{align*}
In particular, 
\[
w_p:= w''_p\otimes \f^*\omega^{-1}_B:K^\kdot_p\to M^\kdot_p
\]
is simply the map to the cone. We summarize this information in the following diagram
\[\begin{tikzcd}
M^\kdot_{p-1}\otimes \f^*\omega_B \ar[r,"w'_{p-1}"] &
K^\kdot_p \ar[r, "w_p"] &
M^\kdot_p \ar[r, "+1"] &
\, \\
K^\kdot_{p-1}\otimes \f^*\omega_B \ar[u,"w''_{p-1}"] \ar[ur, "\wedge_{p-1}"'] \ar[urr, bend right=15, "0"']&
\, &
\, &
\,
\end{tikzcd}\]
The two desired properties are the two commuting triangles in the diagram, noting that
\[
w''_p\circ \wedge'_{p-1} = (w_p\circ \wedge_{p-1})\otimes \Id_{\f^*\omega_B} = 0.
\]
These follow from the constructions of $w''_p$ and $w'_p$.

By the same argument as in \cite{Kov96}, since the wedge map $\wedge_p$ is natural, the constructions of $M^\kdot_p, w'_p, w''_p$ are independent of the hyperresolution chosen. This concludes the constructions for all $p\in \bbZ$.
\end{construction}

This gives us properties very similar to those in \cite[Theorem-Definition 1.3]{Kov96}. We state them together for ease of comparison.

\begin{theorem}\label{thm_2relDB}
Let $\f:X\to B$ be a morphism from a complex variety $X$ to a smooth curve $B$. For each $p\in \bbZ$ there exists the \emph{right} relative Du Bois complex
\[
\uOmega^{p, +}_{X/B}\in D^b_{\rm coh}(X),
\]
and the \emph{left} relative Du Bois complex
\[
\uOmega^{p, -}_{X/B}\in D^b_{\rm coh}(X)
\]
with the following properties. We use $*\in \{+, -\}$ to denote properties that hold for either complex.
\begin{enumerate}
\item The natural map $\wedge_p$ factors through $\uOmega^{p, *}_{X/B}\otimes \f^*\omega_B$.
\item We have
\[
\uOmega^{p-1,*}_{X/B}\otimes \f^*\omega_B \to
\uOmega^p_X \to
\uOmega^{p,*}_{X/B} \xrightarrow{+1}
\]
a distinguished triangle in $D(X)$.

\item The map $\uOmega^p_X\to \uOmega^{p,*}_{X/B}$ is functorial for morphisms over $B$.

\item If $\f$ is smooth, then 
\[
\uOmega^{p,*}_{X/B}\qis \Omega^p_{X/B}.
\]

\item We have 
\begin{align*}
\uOmega^{p, +}_{X/B}=0& \quad p<0, \\
\uOmega^{p, -}_{X/B}=0& \quad p>\dim X - \dim \f(X).
\end{align*}

\item Let $\{\nu_i:U_i\to X\}$ be a finite open cover of $X$. Then 
\[
\uOmega^{p,*}_{X/B}\qis R(\nu_\kdot)_*\uOmega^{p, *}_{U_\kdot/B}.
\]
\end{enumerate}
\end{theorem}

\begin{proof}
Let $\uOmega^{p,+}_{X/B}\qis M^\kdot_p$ be as in Construction \ref{cons_relDB} and $\uOmega^{p,-}_{X/B}$ be the relative Du Bois complex as defined in \cite{Kov96}. Note that while the construction for $\uOmega^{p,-}_{X/B}$ is stated for $p\in \bbN$, it also works for $p<0$ by continuing the iterative downward construction.

Then, properties (1)-(4) and (6) follow exactly the same as in \emph{loc. cit.} using the properties demanded in Construction \ref{cons_relDB}, where we employ descending induction for $\uOmega^{p,-}_{X/B}$ from $p=\dim X+1$ and ascending induction for $\uOmega^{p,+}_{X/B}$ from $p=-1$. Finally, (5) is \cite[1.3.5]{Kov96} together with Construction \ref{cons_relDB}. We choose the terminologies ``left'' and ``right'', and notations $+, -$ since we can have $\uOmega^{p,+}_{X/B}\ne 0$ for $p\gg 0$ and $\uOmega^{p,-}_{X/B}\ne 0$ for $p\ll 0$.

We also note that the coherence statement: $\uOmega^{p,*}_{X/B}\in D^b_{\rm coh}(X)$ does not require any properness assumption. In either case, this follows using induction from the fact that $\uOmega^p_X\in D^b_{\rm coh}(X)$ has coherent cohomology sheaves; see \cite{DB81,GNPP88}. It is important to note that the functorial map in (3) need not be coherent for non-proper morphisms over $B$.

Finally, neither Construction \ref{cons_relDB} nor \cite[Theorem 1.3]{Kov96} rely on the fact that $\f: X \to B$ is dominant, with the understanding that the relevant wedge maps are zero for vertical $\f: X\to B$.
\end{proof}

\begin{remark}
Using a similar method to \cite{KT25}, one may also construct a filtered right relative Du Bois complex $\uOmega^{\kdot, +}_{X/B}$ whose associated graded pieces are $\uOmega^{p, +}_{X/B}$. Since the main applications of the right relative Du Bois complex do not require this strengthened construction, we will not include it here.
\end{remark}

From this point onward, we use the notation of Theorem \ref{thm_2relDB}. Namely, $f:X\to B$ is a dominant morphism from a complex variety $X$ to a smooth curve $B$. We use $\uOmega^{p, *}_{X/B}$, with $*\in \{+, -\}$ to denote either one of the relative Du Bois complexes.

The following standard lemma states that $\uOmega^{p, *}_{X/B}$ can be first constructed on a hyperresolution, with smooth total space. See also \cite[Lemma 3.3]{Kov97} for the statement for $\uOmega^{p, -}_{X/B}$ in a more general case.

\begin{lemma}\label{lem_relDB_hyperres}
Let $\pi_\kdot:X_\kdot\to X$ be a cubical hyperresolution of $X$. Then 
\[
R(\pi_\kdot)_*\uOmega^{p,*}_{X_\kdot/B} \qis \uOmega^{p, *}_{X/B}.
\]
\end{lemma}

\begin{proof}
By functoriality, we have a commutative diagram
\[\begin{tikzcd}
\left(R(\pi_\kdot)_*\uOmega^{p-1,*}_{X_\kdot/B}\right)\otimes \f^*\omega_B \ar[r] \ar[d] &
\left(R(\pi_\kdot)_*\uOmega^{p}_{X_\kdot}\right) \ar[r] \ar[d] &
\left(R(\pi_\kdot)_*\uOmega^{p,*}_{X_\kdot/B}\right) \ar[r, "+1"] \ar[d] &
\, \\
\uOmega^{p-1,*}_{X/B}\otimes \f^*\omega_B \ar[r] &
\uOmega^p_X \ar[r] &
\uOmega^{p,*}_{X/B} \ar[r,"+1"] &
.\,
\end{tikzcd}\]
The middle vertical map is a quasi-isomorphism by \cite[2.3]{DB81}. Hence the outer ones also are, using the appropriate induction for either $*\in \{+, -\}$ and the five lemma.
\end{proof}

We will also record some basic consequences of the constructions.

\begin{lemma}\label{lem_relDB_tail}
The following holds.
\begin{enumerate}
\item For $p> \dim X$,
\[
\uOmega^{p,+}_{X/B}\qis \uOmega^{\dim X, +}_{X/B}\otimes \f^*\omega_B^{\otimes p-\dim X} [p-\dim X].
\]

\item For $p< 0$,
\[
\uOmega^{p,-}_{X/B}\qis \uOmega^{-1, -}_{X/B}\otimes\f^*\omega_B^{\otimes p+1} [p+1].
\]
\end{enumerate}
\end{lemma}

\begin{proof}
Either $p< 0$ or $p> \dim X$ implies $\uOmega^p_X= 0$. The fundamental triangle in Theorem \ref{thm_2relDB} gives
\[
\uOmega^{p,*}_{X/B}\qis \uOmega^{p-1,*}_{X/B}\otimes \f^*\omega_B [1].
\]
The two statements now follow from counting.
\end{proof}

We see that the non-zero ``tail'' of the two relative Du Bois complexes just repeats after shifts by definition. One consequence is the following simple equivalent conditions under which the two agree.

\begin{lemma}\label{lem_2relDB_tfae}
The following are equivalent.
\begin{enumerate}
\item $\uOmega^{p,+}_{X/B}\qis \uOmega^{p,-}_{X/B}$ for all $p\in \bbZ$.

\item $\uOmega^{p,+}_{X/B}\qis \uOmega^{p,-}_{X/B}$ for some $p\in \bbZ$, compatible with the map $\uOmega^p_X\to \uOmega^{p, *}_{X/B}$.

\item $\uOmega^{p,+}_{X/B}=0$ for all $p> \dim X$.

\item $\uOmega^{p,+}_{X/B}=0$ for some $p> \dim X$.

\item $\uOmega^{p,-}_{X/B}=0$ for all $p< 0$.

\item $\uOmega^{p,-}_{X/B}=0$ for some $p< 0$.
\end{enumerate}
\end{lemma}

\begin{proof}
(1) implies (3)-(6). Either (4) or (6) implies (2), by Theorem \ref{thm_2relDB}(5), where we note that compatibility is automatic when both are $0$. Lemma \ref{lem_relDB_tail} establishes the equivalence between (3) and (4), and between (5) and (6). Thus it remains to show (2) implies (1). Suppose (2) holds, and fix a quasi-isomorphism
\[
\uOmega^{p,+}_{X/B} \xrightarrow{\qis} \uOmega^{p,-}_{X/B}.
\]
For $q\le p$, we consider the following diagrams of two triangles
\[\begin{tikzcd}
\uOmega^{q-1,+}_{X/B}\otimes \f^*\omega_B \ar[r] \ar[d, dotted] &
\uOmega^q_X \ar[r] \ar[d, equal] &
\uOmega^{q,+}_{X/B} \ar[r,"+1"] \ar[d, "\qis"] &
.\, \\
\uOmega^{q-1,-}_{X/B}\otimes \f^*\omega_B \ar[r] &
\uOmega^q_X \ar[r] &
\uOmega^{q,-}_{X/B} \ar[r,"+1"] &
.\,
\end{tikzcd}\]
By induction and the five lemma, we may extend to quasi-isomorphisms $\uOmega^{q-1,+}_{X/B}\qis \uOmega^{q-1,-}_{X/B}$ after twisting $f^*\omega^{-1}_B$. The same argument shows the quasi-isomorphism extends for $q\ge p$. This concludes the proof.
\end{proof}

We remark that there need not be natural comparison maps between $\uOmega^{p,+}_{X/B}$ and $\uOmega^{p,-}_{X/B}$, compatible with the map from $\uOmega^p_X$. It is directly supplied by condition (2), or automatic when both vanish.

\begin{lemma}\label{lem_0piece}
We have 
\[
\uOmega^{0,+}_{X/B}\qis \uOmega^0_X \quad \text{and} \quad \uOmega^{n,-}_{X/B}\qis \uOmega^{n+1}_X\otimes \f^*\omega_B^{-1},
\]
where $\dim X= n+1$.
\end{lemma}

\begin{proof}
These follow immediately from Theorem \ref{thm_2relDB}(2), (5).
\end{proof}

Denote 
\[
n_k = \max\{i\in \bbZ: h^i(\uOmega^k_X) \ne 0\}
\]
to be the maximum degree of non-vanishing cohomologies of $\uOmega^k_X$, for $0\le k\le \dim X$. The maximum is always obtained as $\uOmega^k_X\in D^b(X)$ unless $\uOmega^k_X=0$, for $k<0$ or $k>\dim X$, in which case by convention $n_k=-\infty$. Further denote, for suitable ranges of $p$
\[
m_p^+ = \max_{0\le k\le p} n_k \quad \text{and} \quad m_p^- = \max_{p\le k\le \dim X} n_k.
\]

\begin{lemma}\label{lem_relDB_coh_amp}
Let $p\in \bbZ$. Then using the notation directly above,
\begin{enumerate}
\item $\uOmega^{p,+}_{X/B}$ has cohomological amplitude in $[-p, m^+_p]$ for $p\ge 0$;

\item $\uOmega^{p,-}_{X/B}$ has cohomological amplitude in $[0, \max(\dim X-p, m^-_p)]$ for $p\le \dim X$.
\end{enumerate}
We say that $K\in D(X)$ has cohomological amplitude in $[a,b]$ if $h^i(K)=0$ for $i\notin [a,b]$.
\end{lemma}

\begin{proof}
This follows directly by counting the non-zero degrees of the representatives $M^\kdot_p$ in Construction \ref{cons_relDB}, for $\uOmega^{p,+}_{X/B}$, and in \cite[1.3]{Kov96}, for $\uOmega^{p,-}_{X/B}$.
\end{proof}

\begin{remark}
We see in the previous lemma that, while $\uOmega^{p,+}_{X/B}$ has the disadvantage of having non-zero negative cohomologies, it separates the contribution from the singularities of $X$ and that of the family. The former contributes, through $\uOmega^p_X$, to the non-negative cohomologies of $\uOmega^{p,+}_{X/B}$, while the latter contributes, through the coning construction, to the negative cohomologies. This trade-off actually makes $\uOmega^{p,+}_{X/B}$ easier to work with for the purpose of base change, as we will see next.
\end{remark}

\medskip

\section{Base change for the relative Du Bois complex}\label{section_basechange}

In this section, we study the base change behavior of the two relative Du Bois complexes. We will reserve the notation $f:X\to B$ as a dominant morphism from a complex variety $X$ to a smooth curve $B$.

\begin{definition}\label{def_bc}
Let $f:X\to B$ be a dominant morphism from a complex variety $X$ to a smooth curve $B$. Let $b\in B$, $m\in\bbZ$. We say the fiber $X_b$ satisfies
\begin{itemize}
\item \emph{left-$m$-base change} for the relative Du Bois complex if there is a natural quasi-isomorphism
\[
\uOmega^{p, -}_{X/B}\otimes^{\tL}_X\scrO_{X_b}\qis \uOmega^{p}_{X_b},
\]
for all $p\le m$.

\item \emph{right-$m$-base change} for the relative Du Bois complex if there is a natural quasi-isomorphism
\[
\uOmega^{p, +}_{X/B}\otimes^{\tL}_X\scrO_{X_b}\qis \uOmega^{p}_{X_b},
\]
for all $p\le m$.

\item \emph{$m$-base change} for the relative Du Bois complex if it satisfies both left- and right-$m$-base change.

\item \emph{base change} (resp. \emph{left-base change}, \emph{right-base change}) for the relative Du Bois complex if it satisfies $m$-base change (resp. left-$m$-base change, right-$m$-base change) for all $m\in \bbZ$.
\end{itemize}

We will frequently omit ``for the relative Du Bois complex'' if it is clear from the context.
\end{definition}

The definition easily extends to any Cartier divisors, since locally any Cartier divisor can be realized as a fiber of a morphism to a smooth curve.

\begin{definition}\label{def_bc_divisor}
Let $Z\subseteq X$ be a Cartier divisor, and let $x\in Z$. We say that $Z\subseteq X$ satisfies \emph{local left-$m$-base change} (resp. \emph{local right-$m$-base change}, \emph{local $m$-base change}, \emph{local base change}) for the relative Du Bois complex at $x\in Z$ if there exists a neighborhood $U\subseteq X$ containing $x$ and a dominant morphism $\f: U\to \bbA^1$ such that $\f^{-1}(0) = Z\cap U$ and the fiber $U_0$ satisfies {left-$m$-base change} (resp. {right-$m$-base change}, {$m$-base change}, {base change}) for the relative Du Bois complex.

We say that $Z\subseteq X$ satisfies \emph{left-$m$-base change} (resp. \emph{right-$m$-base change}, \emph{$m$-base change}, \emph{base change}) for the relative Du Bois complex if it does so locally at every $x\in Z$.
\end{definition}

We will see that left- and right-base change in all degrees is in fact equivalent as a consequence of the following claim.

\begin{proposition}\label{prop_leftright_agree}
Let $f:X\to B$ be a dominant morphism from a complex variety $X$ to a smooth curve $B$, $b\in B$. Suppose one of the following conditions holds:
\begin{enumerate}
\item $X_b$ satisfies left-$m$-base change for some $m<0$.
\item $X_b$ satisfies right-$m$-base change for some $m>\dim \f$.
\item More generally, if $\uOmega^{p,+}_{X/B}\otimes^{\tL} \scrO_{X_b}\qis \uOmega^p_{X_b}=0$ for some $p>\dim \f$.
\end{enumerate}
Then for all $p\in \bbZ$,
\[
\uOmega^{p, +}_{X/B}\otimes^{\tL}\scrO_{X_b}\qis \uOmega^{p, -}_{X/B}\otimes^{\tL}\scrO_{X_b}.
\]
Consequently, for all $m\in\bbZ$, $X_b$ satisfies left-$m$-base change if and only if it satisfies right-$m$-base change.
\end{proposition}

\begin{proof}
This follows by combining Lemma \ref{lem_2relDB_tfae}, \ref{lem_Nakayama}, and Theorem \ref{thm_2relDB}(5). Specifically, since by convention $\uOmega^p_{X_b}=0$ for $p<0$ or $p>\dim \f$, under any of the three assumptions we have
\[
\uOmega^{p, +}_{X/B}\otimes^{\tL}\scrO_{X_b}\qis \uOmega^{p, -}_{X/B}\otimes^{\tL}\scrO_{X_b} \qis 0
\]
for at least one value of $p\in \bbZ$. Now Lemma \ref{lem_2relDB_tfae} and \ref{lem_Nakayama} together imply the same is true in a neighborhood of $X_b$ for any $p\in\bbZ$. Under these conditions, the notions of left- and right-$m$-base change coincide for all $m\in\bbZ$, as needed.
\end{proof}

\begin{corollary}\label{cor_leftright_agree}
Let $f:X\to B$ be a dominant morphism from a complex variety $X$ to a smooth curve $B$, $b\in B$. Then $X_b$ satisfies left-base change if and only if it satisfies right-base change.
\end{corollary}

\begin{proof}
This follows directly from the definition and Proposition \ref{prop_leftright_agree}.
\end{proof}

We should view satisfying a certain level of base change as a condition on the singularities of the fiber. The next two lemmas are first examples of this.

\begin{lemma}\label{lem_basechange_0}
Suppose $X$ is Du Bois in a neighborhood of $X_b$. Then $X_b$ satisfies right-$0$-base change if and only if $X_b$ has (0-)Du Bois singularities.
\end{lemma}

\begin{proof}
Under the assumptions, using Lemma \ref{lem_0piece} we have
\[
\uOmega^{0,+}_{X/B}\otimes^{\tL}\scrO_{X_b} \qis \scrO_X\otimes^{\tL}\scrO_{X_b} \qis \scrO_{X_b}.
\]
Therefore right-$0$-base change holds if and only if $\uOmega^0_{X_b}\qis \scrO_{X_b}$, i.e. $X_b$ is Du Bois.
\end{proof}

\begin{lemma}\label{lem_conormal_triangle}
Suppose $X_b$ satisfies either left- or right-$m$-base change. Then there is a distinguished triangle
\[
\uOmega^{p-1}_{X_b} \to \uOmega^p_X\otimes^{\tL}_X \scrO_{X_b} \to \uOmega^p_{X_b} \xrightarrow{+1}
\]
for all $p \le m$.
\end{lemma}

We remark that this triangle is known to be exact if $X_b$ is general; see \cite[Lemma 5.5]{Kov25}.

\begin{proof}
Simply apply $\otimes^{\tL}_X\scrO_{X_b}$ to the fundamental triangle in \ref{thm_2relDB}(2) to $X_b$, we see the $f^*\omega_B$ term becomes trivial as it factors through $b\in B$. The desired triangle thus follows.
\end{proof}

\subsection{A key diagram}

The advantage of using the right relative Du Bois complex is the following diagram (Theorem \ref{thm_full_diagram}), allowing for direct comparison between $\uOmega^p_{X_b}$ and $\uOmega^{p,+}_{X/B}\otimes^{\tL}\scrO_Z$:

\[\begin{tikzcd}[cramped]
\uOmega^{p-1,+}_{X/B}\otimes^{\tL} \scrO_Z \ar[r] \ar[d] &
\uOmega^{p}_{X}\otimes^{\tL} \scrO_Z \ar[r] \ar[d, "{\rm Id}"] &
\uOmega^{p,+}_{X/B}\otimes^{\tL} \scrO_Z \ar[r, "+1"] \ar[d] &
\, \\
\uOmega^{p-1}_{X_b} \ar[r] &
\uOmega^{p}_{X}\otimes^{\tL} \scrO_Z \ar[r] &
\uOmega^{p}_{X_b}. &
\,
\end{tikzcd}\]

We will arrive at this through several steps. First, the right square is a type of functoriality statement for the right relative Du Bois complex over different bases (cf. \cite[Theorem 4.2]{KT23}).

\begin{proposition}\label{prop_factorization}
Let $\f:X\to B$ be a dominant morphism from a variety $X$ to a smooth curve $B$, $b\in B$. There exists a natural map $\alpha_p: \uOmega^{p,+}_{X/B} \to \uOmega^p_{X_b}$ on $X$ such that the following diagram commutes:
\[\begin{tikzcd}
\uOmega^{p}_X \ar[r, "\beta_p"] \ar[dr, "\psi_p"'] &
\uOmega^{p,+}_{X/B} \ar[d, dotted, "\alpha_p"]  \\
&
\uOmega^{p}_{X_b} 
\end{tikzcd}\]
\end{proposition}

\begin{proof}
The two existing natural maps are $\uOmega_X^p \to \uOmega_{X_b}^p$ and $\uOmega_X^p\to \uOmega^{p,+}_{X/B}$, and the latter fits in the triangle
\[
\uOmega_{X/B}^{p-1,+} \otimes \f^*\omega_B \xrightarrow{\phi_{p-1}} \uOmega_X^{p} \xrightarrow{\beta_p} \uOmega_{X/B}^{p,+} \xrightarrow{+1}
\]
By considering the following diagram on $X$,
\[\begin{tikzcd}
\uOmega_{X/B}^{p-1,+} \otimes \f^*\omega_B \ar[r] \ar[d] &
\uOmega_X^p \ar[r] \ar[d] &
\uOmega_{X/B}^{p,+} \ar[d, dotted, "\alpha_p"] \ar[r, "+1"] &
\, \\
0 \ar[r] & 
\uOmega_{X_b}^p \ar[r, "{\rm Id}"] &
\uOmega_{X_b}^p \ar[r, "+1"] &
\,
\end{tikzcd}
\]
we see that the desired factorization exists if and only if the first square commutes. In other words, if the composition
\[
\uOmega_{X/B}^{p-1, +} \otimes \f^*\omega_B \xrightarrow{\phi_{p-1}} \uOmega_X^{p} \xrightarrow{\psi_p} \uOmega_{X_b}^{p}
\]
is $0$. In the smooth case, the first map is given by wedging differential forms from the base, which then restricts to $0$ on the fiber $X_b$. In general, using the notation in Construction \ref{cons_relDB}, let $K_p^\kdot, M_p^\kdot$ be specific representatives for $\uOmega^p_X$ and $\uOmega^{p,+}_{X/B}$ after choosing a hyperresolution. Then
\[
\wedge_{p-1}: K^\kdot_{p-1} \otimes \f^*\omega_B \to K^\kdot_{p}
\]
satisfies $\psi_p\circ \wedge_{p-1}=0$. Furthermore, we have 
\[
M^j_{p-1} = (M^{j+1}_{p-2}\otimes \f^*\omega_B) \oplus K^j_{p-1}
\]
and under that identification 
\[
\phi_{p-1} = (0, \wedge_{p-1}): M^\kdot_{p-1} \otimes \f^*\omega_B \to K^\kdot_{p}.
\]
Therefore we can conclude that $\psi_p\circ \phi_{p-1}=(0, \psi_p\circ \wedge_{p-1})=0$, as desired.
\end{proof}

\begin{remark}
It is not immediate that the left relative Du Bois complex satisfies the same property unless the two agree. The key difference is that the map $\uOmega^{p-1,-}_{X/B}\otimes \f^*\omega_B \to \uOmega^p_X$ becomes instead
\[
(\Id_{K^\kdot_p}, 0):M^\kdot_{p-1}\otimes \f^*\omega_B \to K^\kdot_{p},
\]
where $K_p^\kdot, M_p^\kdot$ are the representatives for $\uOmega^p_X, \uOmega^{p,-}_{X/B}$ as in \cite[1.3]{Kov96}, and
\[
M^j_{p-1}\otimes \f^*\omega_B = K^j_{p}\oplus M^{j-1}_{p}.
\]
It is unclear that the composition $\psi_p\circ (\Id_{K^\kdot_p}, 0)$ to $\uOmega^p_{X_b}$ becomes zero on the level of complexes.

This difference is due to the nature of the inductive constructions. Even though $\wedge_{p-1}$ factors through both constructions (cf. Theorem \ref{thm_2relDB}(2)):
\[
\wedge_{p-1}: \uOmega^{p-1}_X \otimes \f^*\omega_B \to \uOmega^{p-1, *}_{X/B} \otimes \f^*\omega_B \to \uOmega^{p}_X,
\]
the left relative complex inherits the wedge map on the first part, while the right relative complex inherits it on the second part. This is crucial for the factorization in Proposition \ref{prop_factorization}.
\end{remark}

Next, we need to build the left square of Theorem \ref{thm_full_diagram}. This will be done from a hyperresolution.

\begin{lemma}\label{lem_conormal_map}
Suppose $Z\subseteq X$ is an effective Cartier divisor. There exists a natural map
\[
\tau_{p-1}: \uOmega^{p-1}_Z(-Z) \to \uOmega^p_X\otimes_X^{\tL} \scrO_Z
\]
which agrees with the conormal map when $Z\subseteq X$ is smooth, or more generally, when $Z$ is a general member of a basepoint-free linear system (cf. \cite[Lemma 5.5]{Kov25}).
\end{lemma}

\begin{proof}
Choose a hyperresolution as in Construction \ref{eqn_hyperres}. Applying Lemma \ref{lem_sm_diagram_derived} on each $\widetilde{Z}_\alpha \subseteq Z_\alpha \subseteq X_\alpha$ gives
\[
\tau_\kdot: \Omega^{p-1}_{\widetilde{Z}_\kdot}(-Z_\kdot) \to \Omega^p_{X_\kdot}\otimes L\pi_{\kdot}^*\scrO_{Z}
\]
Since $Z_\kdot$ is obtained via pullback and $Z$ is Cartier,
\[
L\pi_\kdot^* \scrO_X(-Z) \qis \pi_\kdot^* \scrO_X(-Z) \qis \scrO_{X_\kdot}(-Z_{\kdot}).
\]
Note that as in Lemma \ref{lem_sm_diagram_derived}, for components such that $Z_\alpha$ is not a divisor, $\scrO_{X_\alpha}(-Z_{\alpha})$ simply denotes the pullback $\pi_\alpha^*\scrO_X(-Z)$. Then, the projection formula gives
\begin{align*}
R(\pi_\kdot)_* \Omega_{\widetilde{Z}_\kdot}^{p-1}(-Z_\kdot) & \qis R(\pi_\kdot)_*( \Omega_{\widetilde{Z}_\kdot}^{p-1} )(-Z) \qis \uOmega^{p-1}_Z(-Z), \\
R(\pi_\kdot)_* \left( \Omega_{X_\kdot}^p\otimes L\pi_{\kdot}^*\scrO_{Z} \right) & \qis R(\pi_\kdot)_*(\Omega_{X_\kdot}^p)\otimes^{\tL} \scrO_{Z} \qis \uOmega^p_X\otimes^{\tL} \scrO_{Z},
\end{align*}
The desired map is thus $\tau_{p-1} = R(\pi_\kdot)_*(\tau_\kdot)$. Lastly, this map gives the conormal map when $Z$ is general because we may choose the hyperresolution such that $\widetilde{Z}_\kdot \isom Z_\kdot$ and use the last statements in Lemma \ref{lem_sm_diagram}, \ref{lem_sm_diagram_derived}. By commutativity in \ref{lem_sm_diagram}, $\tau_{p-1}$ is independent of the hyperresolution chosen.
\end{proof}

\begin{lemma}\label{lem_left_triangle}
The following diagram
\[\begin{tikzcd}
\uOmega^{p-1,+}_{X/B}\otimes^{\tL} \scrO_{X_b} \ar[r, "\phi_{p-1}"] \ar[d, "\alpha_{p-1}"] &
\uOmega^{p}_{X}\otimes^{\tL} \scrO_{X_b} \\
\uOmega^{p-1}_{X_b} \ar[ur, "\tau_{p-1}"'] &
\end{tikzcd}\]
with $\tau_{p-1}$ as in Lemma \ref{lem_conormal_map} and $\alpha_{p-1}, \phi_{p-1}$ as in Proposition \ref{prop_factorization}, commutes for all $p$.
\end{lemma}

\begin{proof}
By abuse of notation, we use the same letter to denote the maps after pulling back to $X_b$, post-composing with the counit $\uOmega^p_{X_b}\otimes^{\tL}_X \scrO_{X_b}\to \uOmega^p_{X_b}$ if necessary. We consider the following enlarged diagram:
\[\begin{tikzcd}
\uOmega^{p-1}_X\otimes^{\tL} \scrO_{X_b} \ar[r] \ar[dr] \ar[rr, bend left=20, "\wedge_{p-1}"]&
\uOmega^{p-1,+}_{X/B}\otimes^{\tL} \scrO_{X_b} \ar[r, "\phi_{p-1}"] \ar[d, "\alpha_{p-1}"]&
\uOmega^{p}_{X}\otimes^{\tL} \scrO_{X_b} \\
&
\uOmega^{p-1}_{X_b}. \ar[ur, "\tau_{p-1}"'] &
\end{tikzcd}\]
Note that before pulling back, the top row is the wedge map $\wedge_{p-1}:\uOmega^{p-1}_X\otimes \f^*\omega_B\to \uOmega^p_X$, as in Theorem \ref{thm_2relDB}(2). The left triangle commutes by Proposition \ref{prop_factorization}.

Using Lemma \ref{lem_relDB_hyperres} and the same argument as in Lemma \ref{lem_conormal_map}, it suffices to show that the right triangle commutes on a hyperresolution. We may therefore assume that $X$ is smooth and either ${X_b}=X$ or $Z:=X_b$ is a divisor with snc support. In either case, the outer triangle commutes. The divisor case is by Lemma \ref{lem_sm_wedge} and the other case trivially holds. To obtain that the right triangle commutes in $D(X)$, we need only show that the map
\[
\Hom(\uOmega^{p-1,+}_{X/B}\otimes^{\tL} \scrO_Z, \Omega^p_X|_Z) \to \Hom(\Omega^{p-1}_X|_Z, \Omega^p_X|_Z)
\]
is injective, so that $\tau_{p-1}\circ\alpha_{p-1} - \phi_{p-1}=0$. It is then sufficient to show the stronger condition
\[
\Hom(\Cone(\Omega^{p-1}_X|_Z\to \uOmega^{p-1,+}_{X/B}\otimes^{\tL} \scrO_Z), \Omega^p_X|_Z) 
\isom \Hom(\uOmega^{p-2,+}_{X/B}[1]\otimes^{\tL} \scrO_Z, \Omega^p_X|_Z) 
= 0.
\]
By Lemma \ref{lem_relDB_coh_amp}, $\uOmega^{p-2,+}_{X/B}$ is supported in cohomological degree $[-p+2, 0]$. This implies that $\uOmega^{p-2,+}_{X/B}[1]\otimes^{\tL} \scrO_Z$ is supported in degree $\le -1$, hence it admits no non-trivial map in $D(X)$ to $\Omega^p_X|_Z$, which is concentrated in degree $0$. Combined with the above, this shows $\tau_{p-1}\circ \alpha_{p-1}=\phi_{p-1}$, as desired.
\end{proof}

Both the construction involving the wedge map and the cohomological amplitude of $\uOmega^{p,+}_{X/B}$ prove to be more useful for base change, as compared to that of $\uOmega^{p,-}_{X/B}$. The following example shows that the same diagram does not commute without additional hypotheses for $\uOmega^{p,-}_{X/B}$.

\begin{example}\label{ex_relDB_no_commute}
Let $X$ be smooth and $Z:=X_b$ be a singular snc fiber. For the top degree $n=\dim \f$, we have $\uOmega^{n,-}_{X/B}\otimes^{\tL} \scrO_Z \qis \Omega^{n+1}_X|_Z$ by Lemma \ref{lem_0piece}. The diagram in Lemma \ref{lem_left_triangle} reads
\[\begin{tikzcd}[cramped]
\Omega^{n+1}_X|_Z \ar[r, "\isom"] \ar[d] &
\Omega^{n+1}_X|_Z \\
\uOmega^{n}_{Z}. \ar[ur] &
\end{tikzcd}\]
If it commutes, then $\Omega^{n+1}_X|_Z$ is a summand of $\uOmega^n_Z$. Now, it is well known that $\uOmega^n_Z$ is a single torsion-free sheaf, for example by \cite[Example 7.23]{PS08}. However, the two sheaves are generically of rank one, so the complementary summand in $\uOmega^n_Z$ must be both rank zero and torsion-free, hence trivial. This contradicts the fact that $\uOmega^n_Z\not\isom \Omega^{n+1}_X|_Z$ -- the former is not even locally free.
\end{example}

\begin{theorem}\label{thm_full_diagram}
Let $\f:X\to B$ be a dominant morphism from a variety $X$ to a smooth curve $B$, $b\in B$. The following diagram:
\[\begin{tikzcd}
\uOmega^{p-1, +}_{X/B}\otimes^{\tL} \scrO_{X_b} \ar[r, "\phi_{p-1}"] \ar[d, "\alpha_{p-1}"] &
\uOmega^{p}_{X}\otimes^{\tL} \scrO_{X_b} \ar[r, "\beta_p"] \ar[d, "{\rm Id}"] &
\uOmega^{p, +}_{X/B}\otimes^{\tL} \scrO_{X_b} \ar[r, "+1"] \ar[d, "\alpha_{p}"] &
\, \\
\uOmega^{p-1}_{X_b} \ar[r, "\tau_{p-1}"] &
\uOmega^{p}_{X}\otimes^{\tL} \scrO_{X_b} \ar[r, "\psi_p"] &
\uOmega^{p}_{X_b} &
\,
\end{tikzcd}\]
with $\tau_{p-1}$ as in Lemma \ref{lem_conormal_map} and $\alpha_p, \beta_p, \phi_p, \psi_p$ as in Proposition \ref{prop_factorization}, commutes.
\end{theorem}

\begin{proof}
The top row comes from Theorem \ref{thm_2relDB}(2). This now follows from combining Proposition \ref{prop_factorization} and Lemma \ref{lem_left_triangle}, for the commutativity of the right and left square, respectively.
\end{proof}

An immediate consequence of this construction is the following.

\begin{corollary}\label{cor_bc_iff_exact}
Let $\f:X\to B$ be a dominant morphism from a variety $X$ to a smooth curve $B$, $b\in B$ and $m\in \bbN$. Then $X_b$ satisfies right-$m$-base change if and only if
\[
\uOmega^{p-1}_{X_b} \to \uOmega^p_X\otimes^{\tL}\scrO_{X_b} \to \uOmega^p_{X_b} \xrightarrow{+1}
\]
is a distinguished triangle for all $0\le p\le m$.
\end{corollary}

\begin{proof}
The only if direction is provided by Lemma \ref{lem_conormal_triangle}. Now suppose the given triangle is exact. We consider the same diagram as in Theorem \ref{thm_full_diagram}:
\[\begin{tikzcd}
\uOmega^{p-1, +}_{X/B}\otimes^{\tL} \scrO_{X_b} \ar[r] \ar[d, "\alpha_{p-1}"] &
\uOmega^{p}_{X}\otimes^{\tL} \scrO_{X_b} \ar[r] \ar[d, "{\rm Id}"] &
\uOmega^{p, +}_{X/B}\otimes^{\tL} \scrO_{X_b} \ar[r, "+1"] \ar[d, "\alpha_{p}"] &
\, \\
\uOmega^{p-1}_{X_b} \ar[r] &
\uOmega^{p}_{X}\otimes^{\tL} \scrO_{X_b} \ar[r] &
\uOmega^{p}_{X_b}. \ar[r, "+1"]&
\,
\end{tikzcd}\]
We have $\uOmega^{-1, +}_{X/B}=0$ by Theorem \ref{thm_2relDB}(5). Then $\alpha_{-1}$ is always a quasi-isomorphism. By induction, so is each $\alpha_p$ for $p\le m$. This is precisely saying $X_b$ satisfies right-$m$-base change.
\end{proof}

Through this, we see that right-$m$-base change is a weaker condition than left-$m$-base change. This should be interpreted as an advantage for the right relative Du Bois complex, as frequently the applications of interest require only the existence of \emph{some} object on the total space that base changes.

\begin{proposition}\label{prop_left_implies_right}
Let $\f:X\to B$ be a dominant morphism from a variety $X$ to a smooth curve $B$, $b\in B$ and $m\in\bbZ$. If $X_b$ satisfies left-$m$-base change, then $X_b$ satisfies right-$m$-base change.
\end{proposition}

\begin{proof}
This follows from combining Corollary \ref{cor_bc_iff_exact} and Lemma \ref{lem_conormal_triangle}. The case when $m<0$ is automatic by Theorem \ref{thm_2relDB}(5).
\end{proof}

Theorem \ref{thm_full_diagram} also provides an alternative proof for base change if $X_b$ is assumed to be general, in which case we recover the (left-)base change result as in \cite[Theorem 1.3]{JK25}.

\begin{corollary}\label{cor_basechange_general}
Let $\f:X\to B$ be a dominant morphism from a variety $X$ to a smooth curve $B$, $b\in B$. Suppose $X_b$ is a general fiber. Then $X_b$ satisfies base change for the relative Du Bois complex, i.e.
\[
\uOmega^{p, +}_{X/B}\otimes^{\tL}_X\scrO_{X_b} \qis \uOmega^{p, -}_{X/B}\otimes^{\tL}_X\scrO_{X_b} \qis \uOmega^p_{X_b},
\]
for all $p$.
\end{corollary}

\begin{proof}
Since $X_b$ is general, by \cite[Lemma 5.5]{Kov25}, the triangle
\[
\uOmega^{p-1}_{X_b} \to \uOmega^p_X\otimes^{\tL}\scrO_{X_b} \to \uOmega^p_{X_b} \xrightarrow{+1}
\]
is exact for all $p$. The right-$m$-base change claim follows directly from Corollary \ref{cor_bc_iff_exact} for all $m\in \bbZ$. Therefore, $X_b$ satisfies base change by Corollary \ref{cor_leftright_agree}.
\end{proof}

\medskip

\section{Injectivity theorems}\label{section_inj_thm}

An important ingredient for deformations is the deep Kov\'acs--Schwede-type injectivity theorems, cf. \cite{KS16a,KS16b,PSV24,Kov25}. We recall the most general statement here for convenience. Let $\omega^\kdot_X$ be a dualizing complex on $X$ and denote the shifted Grothendieck duality functor
\[
\bbD_X(-) = R\HOM(-, \omega_X^\kdot)[-\dim X].
\]

\begin{theorem}[\protect{\cite[Theorem 1.1]{Kov25}}]\label{thm_kov_injectivity}
Let $X$ be a variety with pre-$(m-1)$-Du Bois singularities. Then
\[
h^k(\bbD_X(\uOmega^p_X))\to h^k(\bbD_X(h^0(\uOmega^p_X)))
\]
is injective for all $k\in \bbZ$ and $p\le m$.
\end{theorem}

For the purpose of deformations for strict-$m$-Du Bois singularities, we will require an analogous theorem for K\"ahler differentials instead of $h^0$ of the Du Bois complex.

\begin{theorem}\label{thm_kahler_injectivity}
Let $X$ be a variety with strict-$(m-1)$-Du Bois singularities. Then
\[
h^k(\bbD_X(\uOmega^p_X))\to h^k(\bbD_X(\Omega^p_X))
\]
is injective for all $k\in \bbZ$ and $p\le m$.
\end{theorem}

Under the stronger hypotheses of strict-$(m-1)$-Du Bois, this statement also implies the injectivity in Theorem \ref{thm_kov_injectivity} by considering the following natural factorization
\[
\bbD_X(\uOmega^p_X)\to \bbD_X(h^0(\uOmega^p_X))\to \bbD_X(\Omega^p_X).
\]
It easily follows that if the composition is injective on cohomology sheaves, then the first map is also.

\subsection{K\"ahler differentials under cyclic covers.}

First, we study the behavior of K\"ahler differentials under cyclic covers of not necessarily smooth varieties.

\begin{notation}\label{notation_cover}

Let $X$ be a complex variety, $\scrL$ a semi-ample line bundle on $X$. Let $s\in H^0(X,\scrL^N)$ be a section for some $N\gg 0$, so that the zero locus $H=(s=0)$ is an effective Cartier divisor. Now form
\[
\eta: Y=\Spec_X \bigoplus_{i=0}^{N-1} \scrL^{-i} \to X,
\]
the cyclic cover corresponding to the section $s$.
\end{notation}

When $X, Y$ and $H_{\red}$ are assumed to be nonsingular and $N\gg 0$, the classical result of \cite[3.16]{EV92} gives the decomposition
\[
\eta_*(\Omega^p_Y)\isom \Omega^p_X \oplus \bigoplus_{i=1}^{N-1} \Omega^p_X(\log H)\otimes \scrL^{-i},
\]
where $\Omega^p_X(\log H)$ are the sheaves of K\"ahler differentials with logarithmic poles along $H$. Without smoothness assumptions, we will require a separate object.

\begin{definition}\label{def_logK}
Let $X$ be a complex variety, $H\subseteq X$ an effective Cartier divisor. Denote
\[
\Omega^1_{X}(\log_K H) := \Omega^1_{(X, M_H)/(\Spec \bbC, \bbC^*)}
\]
to be sheaf of logarithmic $1$-forms for the rank-one log structure associated to $H\subseteq X$ (i.e. determined locally by the chart $\bbN \to \scrO_X$, $1\mapsto s$ with $s$ a local equation for $H$) over the trivial log structure on $\Spec \bbC$. See \cite[\S 1]{Kat89} and \cite[Chapter IV, \S 1.1]{Ogu18}. We use this notation to preserve the analogy with log K\"ahler differentials, while distinguishing it from the (reflexive) log differentials in the singular case with the additional $K$ for Kato.

By convention, $\Omega^0_{X}(\log_K H):= \scrO_X$. For $p \ge 1$, define
\[
\Omega^p_{X}(\log_K H) := \bigwedge^p \Omega^1_{X}(\log_K H).
\]
The only fact we will need is its local presentation. Suppose $X = \Spec A$, and $H=(s=0)$ for $s\in A$. Let $\theta$ be a formal symbol representing $d\log s$. Then
\[
\Omega^1_{X}(\log_K H) = \Omega^1_{A}(\log_K s) \isom \frac{\Omega^1_{A/\bbC} \oplus A\theta}{(ds-s\theta)}.
\]
In other words, we have a short exact sequence
\begin{equation}\label{eqn_logK_ses}
0\to (ds -s\theta)\to \Omega^1_{A}\oplus A\theta \to \Omega^1_A(\log_K s)\to 0.
\end{equation}
\end{definition}

\begin{lemma}\label{lem_logK_wedge}
Using Definition \ref{def_logK}, there is a short exact sequence
\[
0\to J^p \to \Omega^p_{A}\oplus \left(A\theta \wedge \Omega^{p-1}_A\right) \to \Omega^p_{A}(\log_K s) \to 0
\]
for each $p\ge 1$, where $J^p$ is the submodule generated by 
\begin{align*}
ds\wedge \beta - s\theta \wedge \beta, \quad &\beta \in \Omega^{p-1}_A, \text{ and} \\
ds\wedge \theta\wedge \gamma, \quad &\gamma \in \Omega^{p-2}_A.
\end{align*}
\end{lemma}

\begin{proof}
Fix $p\ge 1$. Note first that
\[
\bigwedge^p \left( \Omega^1_A \oplus A\theta \right) \isom \Omega^p_A \oplus \left( A\theta \wedge \Omega^{p-1}_A\right).
\]
Using (\ref{eqn_logK_ses}) and \cite[Tag 01CJ]{SP}, we have a right exact sequence
\[
A(ds-s\theta) \wedge \left(\Omega^{p-1}_A \oplus A\theta \wedge \Omega^{p-2}_A \right) \to \Omega^p_A \oplus A\theta \wedge \Omega^{p-1}_A \to \Omega^p_A(\log_K s) \to 0.
\]
Simply by reading off the image of the first map, we obtain the generators of the kernel $J^p$ as claimed.
\end{proof}

Locally, let $Y=\Spec B$ where $B= A[t]/(t^N-s)$. $B$ has a natural $\bbZ/N\bbZ$ grading from the cyclic cover, with $\deg A=0$ and $\deg t=1$. The universal derivation is homogeneous with $\deg(dt)=1$, hence every $\Omega^p_B$ admits a natural grading with each graded piece $(\Omega^p_B)_i$ a natural $A$-module.

\begin{lemma}\label{lem_logK_cover_local_decomp}
There is a short exact sequence of graded $B$-modules
\[
0\to I^p \to \left(\Omega^p_A\otimes_A B\right) \oplus \left(\Omega^{p-1}_A\otimes_A B\right)\wedge B\, dt \to \Omega^p_B \to 0
\]
for each $p\ge 1$, where $I^p$ is the submodule generated by
\begin{align*}
b\,ds \wedge \beta - Nbt^{N-1}dt \wedge \beta, \quad &b\in B, \beta \in \Omega^{p-1}_A, \text{ and} \\
b\,ds\wedge dt \wedge \gamma, \quad &b\in B, \gamma \in \Omega^{p-2}_A.
\end{align*}
The degree-$i$ part gives a short exact sequence of $A$-modules
\[
0\to (I^p)_i \to t^i \Omega^p_A\oplus \left( At^{i-1}dt \wedge \Omega^{p-1}_A \right) \to (\Omega^p_B)_i \to 0
\]
for $1\le i\le N-1$, where $(I^p)_i$ is generated by
\begin{align*}
t^i\,ds \wedge \beta - Nst^{i-1}dt \wedge \beta, \quad &\beta \in \Omega^{p-1}_A, \text{ and} \\
t^{i-1}\,ds\wedge dt \wedge \gamma, \quad &\gamma \in \Omega^{p-2}_A.
\end{align*}
For $i=0$, we have $(\Omega^p_B)_0\isom \Omega^p_A$.
\end{lemma}

\begin{proof}
The conormal sequence for the hypersurface $B$ reads
\[
B(t^N-s)\xrightarrow{d} \Omega^1_{A[t]}|_B \to \Omega^1_B \to 0.
\]
Note that $\Omega^1_{A[t]}|_B \isom (\Omega^1_A\otimes_A B)\oplus B\, dt$. Consequently,
\[
\bigwedge^p \Omega^1_{A[t]}|_B \isom \Omega^p_{A[t]}|_B \isom (\Omega^p_A\otimes B) \oplus (\Omega^{p-1}_A\otimes B)\wedge B\, dt
\]
Taking exterior power of the conormal sequence using \cite[Tag 01CJ]{SP}, we have a right exact sequence
\[
B(t^N-s) \otimes \left((\Omega^{p-1}_A\otimes B) \oplus (\Omega^{p-2}_A\otimes B)\wedge B\, dt \right) \to (\Omega^p_A\otimes B) \oplus (\Omega^{p-1}_A\otimes B)\wedge B\, dt \to \Omega^p_B \to 0
\]
The first map is induced by the exterior differentiation from the conormal sequence $d:B(t^N-s)\to \Omega^1_{A[t]}|_B$. Therefore, it is given by wedging elements of $\Omega^{p-1}_{A[t]}|_B$ with $d(bt^N -bs)$ for $b\in B$. Similar to before, denoting the image of the first map as $I^p$, we simply read off the generators from the first term as claimed.

Now, $B$ is free over $A$ with basis $1, t,\dots, t^{N-1}$. The presentation for the degree-$i$ part of the sequence directly follows once we note that $\deg(dt)=1$ and $t^{N+i-1}=st^{i-1}$ for $i\ge 1$. 

For degree $i=0$, the middle term becomes
\[
\Omega^p_A \oplus \left( A t^{N-1}dt \wedge \Omega^{p-1}_A\right)
\]
with relations from $(I^p)_0$ given by generators
\[
ds \wedge \beta = Nt^{N-1}dt \wedge \beta, \quad \text{and} \quad 0 = t^{N-1}\, ds \wedge dt \wedge \gamma,
\]
with $\beta \in \Omega^{p-1}_A, \gamma \in \Omega^{p-2}_A$. The right hand side of both types of relations are contained in the summand $A t^{N-1}dt \wedge \Omega^{p-1}_A$. Therefore, we see that
\[
\left( \Omega^p_A \oplus A t^{N-1}dt \wedge \Omega^{p-1}_A\right) / (I^p)_0 \isom \Omega^p_A.
\]
This shows that $\Omega^p_A\isom (\Omega^p_B)_0$, and concludes the proof.
\end{proof}

This local computation gives the decomposition akin to Esnault--Viehweg without smoothness assumptions. Note that $s \in H^0(X, \scrL^N)$ is also not required to be general.

\begin{proposition}\label{prop_logK_decomp}
Using Notation \ref{notation_cover} and Definition \ref{def_logK}, there are natural isomorphisms of $\bbZ/N\bbZ$-graded $\scrO_X$-modules
\[
\eta_*\Omega^p_Y \isom \Omega^p_X \oplus \bigoplus_{i=1}^{N-1}\Omega^p_X(\log_K H)\otimes \scrL^{-i}
\]
for $p\ge 0$.
\end{proposition}

\begin{proof}
For $p=0$, this is the standard cyclic cover decomposition $\eta_*\scrO_Y = \bigoplus_{i=0}^{N-1} \scrL^{-i}$.

For $p\ge 1$, the $i=0$ case is directly given by Lemma \ref{lem_logK_cover_local_decomp}. When $i\ge 1$, the local presentations given in Lemma \ref{lem_logK_wedge} and \ref{lem_logK_cover_local_decomp} correspond exactly with $Ndt/t$ in place of the symbol $\theta$. Specifically, we have a diagram
\[\begin{tikzcd}
0 \ar[r] &
J^p \ar[r] \ar[d] &
\Omega^p_{A}\oplus \left(A\theta \wedge \Omega^{p-1}_A\right) \ar[r] \ar[d, "{(\cdot t^i, \cdot N t^{i-1}dt)}"] &
\Omega^p_{A}(\log_K s) \ar[r] \ar[d, dotted, "\isom"] &
0 \\
0 \ar[r] &
(I^p)_i \ar[r] &
t^i \Omega^p_A\oplus \left( At^{i-1}dt \wedge \Omega^{p-1}_A \right) \ar[r] &
(\Omega^p_B)_i \ar[r] &
0.
\end{tikzcd}\]
The middle isomorphism is given by sending $\theta$ to $Nt^{i-1}dt$. Under this identification, the generators of $J^p$ correspond to the generators of $(I^p)_i$, as given in Lemma \ref{lem_logK_wedge} and \ref{lem_logK_cover_local_decomp}. Note that here we're using the fact that $N$ is invertible as we are over $\bbC$. Therefore, the isomorphism descends to the quotient. 

These local isomorphisms simply identify $Ndt/t$ with $\theta$, thus the effects of any transition functions on them are the same for either obtaining the eigensheaves of $\eta_*\Omega^p_Y$ or $\Omega^p_X(\log_K H)$. Thus after gluing, we obtain the desired isomorphism
\[
(\eta_*\Omega^p_Y)_i \isom \Omega^p_X(\log_K H)\otimes \scrL^{-i}.
\]
This completes the proof.
\end{proof}

\subsection{A residue sequence.} To utilize this decomposition for the injectivity theorem, we also need a residue-type sequence (cf. \cite[2.3b]{EV92}). We proceed with more local calculations.

\begin{lemma}\label{lem_logK_residue}
Using Notation \ref{notation_cover} and Definition \ref{def_logK}, there is a right exact sequence
\[
\Omega^{p}_X \to \Omega^p_X(\log_K H) \to \Omega^{p-1}_H \to 0
\]
for each $p\ge 1$.
\end{lemma}

\begin{proof}
As before, this is readily reduced to a local calculation. So using the same notation as in Definition \ref{def_logK} and Lemma \ref{lem_logK_wedge}, we define
\[
\res: \Omega^p_A\oplus A\theta\wedge \Omega^{p-1}_A \to \Omega^{p-1}_{A/(s)}
\]
by $\res(\alpha + \theta\wedge \beta) = \bar{\beta}$. It is surjective since any $(p-1)$-form on $A/(s)$ lifts to $A$. Using the presentation in Lemma \ref{lem_logK_wedge}, we check that
\[
\res(ds\wedge \beta - s\theta\wedge \beta) = 0, \quad \text{and} \quad \res(ds\wedge \theta \wedge \gamma) = 0.
\]
Therefore, the map descends to the quotient $(\Omega^p_A\oplus A\theta\wedge \Omega^{p-1}_A)/J^p$, giving
\[
\res: \Omega^p_A(\log_K s) \to \Omega^{p-1}_{A/(s)}.
\]
Combining with the map $\Omega^p_A\to \Omega^p_A(\log_K s)$ induced by the sum, we have a sequence
\[
\Omega^{p}_A \to  \Omega^p_A(\log_K s) \to \Omega^{p-1}_{A/(s)} \to 0.
\]
By construction, the composition is zero, so it remains to show that the kernel of the residue map is contained in $\Omega^{p}_A$. Before that, we need to also consider the conormal sequence associated to $A\to A/(s)$. Its exterior power yields a right exact sequence through \cite[Tag 01CJ]{SP}:
\[
\Omega^{p-2}_A \otimes (s) \to \Omega^{p-1}_A\to \Omega^{p-1}_{A/(s)} \to 0.
\]
The first map is given by $\gamma \otimes a \mapsto \gamma \wedge d(a)$ for $\gamma \in \Omega^{p-2}_A$ and $a\in (s)$. We see that any $\beta \in \Omega^{p-1}_A$ with $\bar{\beta}=0$ satisfies $\beta = s\beta' + \gamma \wedge ds$ for some $\beta'\in \Omega^{p-1}_A$ and $\gamma \in \Omega^{p-2}_A$.

Now, let $\alpha+ \theta\wedge \beta \in \ker \res$ be a representative with $\alpha\in \Omega^p_A$, $\beta\in \Omega^{p-1}_A$, so that $\bar{\beta}=0\in \Omega^{p-1}_{A/(s)}$. We therefore have
\[
\alpha+ \theta\wedge \beta = \alpha+ s\theta\wedge \beta' + \theta\wedge \gamma \wedge ds = \alpha+ ds\wedge \beta' \in \Omega^p_A
\]
using the relations $s\theta\wedge \beta'= ds\wedge \beta'$ and $ds \wedge \theta \wedge \gamma =0$ given as in Lemma \ref{lem_logK_wedge}. This proves the desired exactness.
\end{proof}

\begin{lemma}\label{lem_logK_residue_exact}
Using Notation \ref{notation_cover} and Definition \ref{def_logK}, suppose further that $H$ is general. Then there is a short exact sequence
\[
0\to \Omega^{p}_X \to \Omega^p_X(\log_K H) \to \Omega^{p-1}_H \to 0
\]
for each $p\ge 1$.
\end{lemma}

\begin{proof}
The question is local and we may assume $X$ is affine. Away from $H$, $s$ is invertible hence $\theta = ds/s$. We see that both generating relations of $J^p$ as in Lemma \ref{lem_logK_wedge} become zero. Consequently, $\Omega^{p}_X\to \Omega^p_X(\log_K H)$ is injective on $X\setminus H$. In other words, its kernel is a torsion subsheaf of $\Omega^p_X$ supported on $H$, and this holds for an arbitrary $s$.

By prime avoidance, we may choose $H$ general such that $s$ is a non-zero-divisor on $\Omega^p_X$. In other words, multiplication by $s$ is injective on $\Omega^p_X$. Since $\ker(\Omega^p_X\to \Omega^p_X(\log_K H))$ is annihilated by a power of $s$, it must be trivial. This implies the desired left exactness on all of $X$.
\end{proof}

\begin{corollary}\label{cor_logK_strictDB}
Using Notation \ref{notation_cover} and Definition \ref{def_logK}, suppose that $H$ is general. Suppose $X$ is strict-$m$-Du Bois. Then
\[
\Omega^p_X(\log_K H)\qis \uOmega^p_X(\log H)
\]
for $p\le m$. See \cite[3.I]{Kov25} for definition of the logarithmic Du Bois complex $\uOmega^p_X(\log H)$.
\end{corollary}

\begin{proof}
The case $p=0$ follows by convention for both the K\"ahler and Du Bois objects, since we have $\Omega^p_X(\log_K H)=\scrO_X$ and $\uOmega^p_X(\log H)\qis \uOmega^p_X$. For $H$ general, \cite[5.7.1]{Kov25} gives a triangle
\[
\uOmega^p_X \to \uOmega^p_X(\log H)\to \uOmega^{p-1}_H \xrightarrow{+1}
\]
for all $p\ge 1$. Combining this with Lemma \ref{lem_logK_residue_exact}, we have a diagram
\[\begin{tikzcd}
0 \ar[r] &
\Omega^p_X \ar[r] \ar[d] &
\Omega^p_X(\log_K H) \ar[r] \ar[d] &
\Omega^{p-1}_H \ar[r] \ar[d] &
0 \\
&
\uOmega^p_X \ar[r] &
\uOmega^p_X(\log H) \ar[r] &
\uOmega^{p-1}_H. \ar[r, "+1"] &
\,
\end{tikzcd}\]
Now, since $H$ is general, \cite[Corollary 3.3]{SVV23} gives that $H$ has strict-$m$-Du Bois singularities as well. The desired quasi-isomorphism now follows from the five lemma.
\end{proof}

\subsection{Cofiltrations and cohomologies.}

\begin{notation}\label{def_logK_cofil}
We will use the same notations of \cite[\S 3.B, 3.E, 3.I]{Kov25} for the filtration and cofiltrations associated to the de Rham complex, the Du Bois complex and the logarithmic Du Bois complex. Specifically, we will denote
\begin{align*}
& \nf^p_X = F^p\Omega^\kdot_X, && \nf^X_p = \Omega^\kdot_X/F^{p+1}, \\
& \uf^p_X = F^p\uOmega^\kdot_X, && \uf^X_p = \Cone(\uf^{p+1}_X\to \uOmega^\kdot_X), \\
& \uf^p_X(\log H) = F^p\uOmega^\kdot_X(\log H), \text{ and} && \uf^X_p(\log H) = \Cone(\uf^{p+1}_X(\log H)\to \uOmega^\kdot_X(\log H)).
\end{align*}
\end{notation}

\begin{lemma}\label{lem_logK_decomp_cofil}
Using Notation \ref{notation_cover} and \ref{def_logK_cofil}, there exist filtered objects $\nf^X_p (\log_K H) \otimes \scrL^{-i}$ with 
\[
\Gr^k_{\rm filt}(\nf^X_p (\log_K H) \otimes \scrL^{-i}) \qis \Omega^k_X(\log_K H)\otimes \scrL^{-i}[-k]
\]
for $k\le p$ such that
\[
\eta_* \nf^Y_p \isom \nf^X_p \oplus \bigoplus_{i=1}^{N-1} \nf^X_p (\log_K H) \otimes \scrL^{-i}
\]
for each $p\ge 0$.
\end{lemma}

\begin{proof}
This follows from Proposition \ref{prop_logK_decomp} and the same argument as in \cite[Remark 6.5, Corollary 6.6]{Kov25}. Note very well that each $\nf^X_p (\log_K H) \otimes \scrL^{-i}$ is a complex of filtered connections on $\nf^X_p(\log_K H)$ with respect to $\scrL^{-i}$ (cf \cite[Definition 2.27]{Kov25}). Instead of tensoring, it is a formal object obtained from the cyclic cover, akin to the logarithmic connections of Esnault--Viehweg.
\end{proof}

Now we can derive a key surjection statement for the K\"ahler differentials as a replacement for \cite[Corollary 7.6]{Kov25}.

\begin{proposition}\label{prop_cofil_surject}
Using Notation \ref{notation_cover}, suppose $H$ is general, and further assume $X$ is a connected proper variety. Then there is a surjective map
\[
\bbH^k(X, \nf^X_p(\log_K H)\otimes \scrL^{-i}) \to \bbH^k(X, \uf^X_p(\log H)\otimes \scrL^{-i})
\]
for all $k \in \bbZ$ and all $0<i<N$.
\end{proposition}

\begin{proof}
The proof proceeds in the same way as in \emph{loc. cit.} using Lemma \ref{lem_logK_decomp_cofil}. There is a surjective map $\bbH^k(Y, \nf^Y_p)\to \bbH^k(Y, \uf^Y_p)$. Since the map $\Omega^p_Y\to \uOmega^p_Y$ is functorial and $\bbZ/N\bbZ$-equivariant, after the exact pushforward $\eta_*$ it decomposes into surjections on each component in Lemma \ref{lem_logK_decomp_cofil}. The decomposition for the Du Bois side is already supplied by \cite[Corollary 6.6]{Kov25}.
\end{proof}

Now we have all the necessary ingredients to replicate the injectivity theorem for K\"ahler differentials.

\begin{proof}[Proof of Theorem \ref{thm_kahler_injectivity}]
This follows essentially verbatim from \cite[\S 8]{Kov25} with each $h^0(\uOmega^p_X)$ object replaced with $\Omega^p_X$. We will only point out where each new ingredient from the K\"ahler picture enters. First, with the Du Bois defect replaced with
\[
\uOmega^{p,+}_X=\Cone(\Omega^p_X \to \uOmega^p_X),
\]
Theorem 8.3 and the coherent replacement Corollary 8.6 in \emph{loc. cit.} remain otherwise unchanged. (Note well that the notation $\uOmega^{p,+}_X$ for Du Bois defect in \cite[\S 6]{Kol13} and \cite{Kov25} indicates positive cohomologies of $\uOmega^p_X$, and is otherwise unrelated to the $+$ notation for the right relative Du Bois complex).

In the key Theorem 8.7, we now use the surjection of hypercohomologies from cofiltration of Kato's log K\"ahler differentials (Proposition \ref{prop_cofil_surject}). For statements (v)-(vi), we assume $X$ is strict-$(m-1)$-Du Bois, $H$ is general, and use Corollary \ref{cor_logK_strictDB} to obtain the isomorphisms on lower degree cofiltration complexes. The final duality arguments to obtain injectivity on the cohomology sheaves of $\bbD_X(\uOmega^p_X)\to \bbD_X(\Omega^p_X)$ are formal, and hence unchanged. Note that the compactification argument remains exactly the same, and no properness of $X$ is required.
\end{proof}

\medskip

\section{Deformation of strict-$m$-Du Bois singularities}\label{section_deform_strict}

In this section we show that strict-$m$-Du Bois singularities are invariant under deformations. Our proof uses an auxiliary construction to provide a left inverse for $\Omega^p_X\to \uOmega^p_X$, and uses the injectivity theorem from the previous section to show strict-$m$-Du Bois for the ambient space. We first introduce a key local construction.

\begin{construction}\label{cons_deform_cover}
Let $X$ be affine and $Z=V(t)\subseteq X$ a reduced Cartier divisor defined by a nonzerodivisor $t$ on $X$. Fix $m\ge 1$ and define
\[
Y = \Spec_X(\scrO_X[y_1,\dots,y_{m+1}]/(y_1\cdots y_{m+1}-t)).
\]
Denote $Z_i = V(y_i)\subseteq Y$. Observe that the relation becomes simply $t=0$ if any coordinate $y_i$ vanishes, so $Z_i\isom Z\times \bbA^m$. Furthermore, for $1\le r\le m+1$ we also have
\[
Z\times \bbA^{m+1-r} \isom V(y_1,\dots, y_r)\subseteq Y.
\]
In particular, $y_1,\dots, y_{m+1}$ is an $\scrO_Y$-regular sequence, and the embedding of $Z\into Y$ as the zero section of $y_1=\dots=y_{m+1}=0$ is a regular immersion of codimension $m+1$. We will consider $Z\subseteq Z_i\subseteq Y$ as a subscheme via this immersion. We will also denote $U\subseteq Y$ the open subset where $y_i\ne 0$ for $i\ge 2$. We may invert to obtain $y_1=t/(y_2\cdots y_{m+1})$, and see that $U\isom X\times (\bbA^1\setminus \{0\})^m$.
\end{construction}

The main benefit of this construction is that the pullback of differentials admits a simple description. Using the presentation for $\Omega^1_Y$ as a quotient of $X\times \bbA^{m+1}$, the relation reduces to $d(y_1\cdots y_{m+1}-t)=dt$ on $Z$ because $y_i=0$ and $m\ge 1$. Therefore, we have
\[
\Omega^1_Y|_Z = \frac{\Omega^1_X|_Z}{\scrO_Z\, dt} \oplus \bigoplus_{j\in \{1,\dots, m+1\}}\scrO_Z\, dy_j
\]
Since $Z$ is a Cartier divisor, the first term simplifies to $\Omega^1_Z$ using the conormal sequence. We now have
\[
\Omega^1_Y|_Z = \Omega^1_Z \oplus \bigoplus_{j\in \{1,\dots, m+1\}}\scrO_Z\, dy_j, \quad \text{ and } \quad
\Omega^1_{Z_i}|_Z = \Omega^1_Z \oplus \bigoplus_{j\in \{ 1,\dots, m+1\}\setminus \{i\}}\scrO_Z\, dy_j.
\]
The analogous description for $Z_i$ follows from the fact that $Z_i\isom Z\times \bbA^m$. For simplicity of notation, when $J\subseteq \{1,\dots, m+1\}$, denote $dy_J = \bigwedge_{j\in J} dy_j$ in increasing order. Taking exterior power, we obtain
\begin{equation}\label{eqn_deform_kahler}
\Omega^m_Y|_Z = \bigoplus_{\substack{J\subseteq \{1,\dots, m+1\} \\ |J|\le m}}\Omega^{m-|J|}_Z\wedge dy_J, \quad \text{ and } \quad
\Omega^m_{Z_i}|_Z = \bigoplus_{\substack{J\subseteq \{1,\dots, m+1\}\setminus\{i\} \\ |J|\le m}}\Omega^{m-|J|}_Z\wedge dy_J.
\end{equation}
Note that the $\Omega^m_Z$ term is given by the $J=\emptyset$ summand. Under this identification, the usual restriction of differentials $\Omega^m_Y\to \Omega^m_{Z_i}$ induces the identity on the $J$-component if $i\notin J$ and zero otherwise.

\begin{lemma}\label{lem_deform_tor_ind}
Using the notation in Construction \ref{cons_deform_cover}, as $\scrO_{Z_i}$-modules, we have
\begin{enumerate}
\item $\Omega^p_{Z_i}$ and $\scrO_Z$ are $\Tor$-independent for all $p\ge 0$.
\end{enumerate}
If in addition $Z$ is strict-$m$-Du Bois, then as $\scrO_Y$-modules, we have
\begin{enumerate}
\setcounter{enumi}{1}
\item $\Omega^p_Y$ and $\scrO_Z\,$ are $\Tor$-independent for $p\le m$, and
\item $\Omega^p_Y$ and $\scrO_{Z_i}$ are $\Tor$-independent for $p\le m$ and all $i$.
\end{enumerate}
\end{lemma}

\begin{proof}
The first claim essentially follows since $Z_i\isom Z\times\bbA^m$ and $Z$ embeds as the zero section. More precisely, let $\pi_i:Z_i\to Z$ be the flat projection map and $\sigma_i:Z\to Z_i$ be the zero section. For any coherent sheaf $\scrF$ on $Z$, we have
\[
\scrF \qis L\sigma_i^*L\pi_i^*\scrF\qis L\sigma_i^* \pi_i^*\scrF.
\]
Note that the differentials on $Z_i\isom Z\times \bbA^m$ decompose into sums of sheaves pulled back from $Z$,
\[
\Omega^p_{Z_i}= \bigoplus_q (\pi_i^* \Omega^q_Z)^{\oplus {m \choose p-q}}.
\]
In particular, there is a quasi-isomorphism 
\[
L\sigma_i^* \Omega^p_{Z_i}\qis \bigoplus_q  (\Omega^q_Z)^{\oplus {m \choose p-q}} \isom \sigma_i^*\Omega^p_{Z_i},
\] 
proving (1). On the other hand, each $Z_i$ is strict-$m$-Du Bois by Lemma \ref{lem_product}. Fix $p\le m$, by Lemma \ref{lem_conormal} we have a short exact sequence
\[
0\to \Omega^{p-1}_{Z_i}(-Z_i)\to \Omega^p_Y|_{Z_i}\to \Omega^p_{Z_i} \to 0
\]
Furthermore, Lemma \ref{lem_tf_omega} implies that $\Omega_Y^p$ is torsion-free, hence $\Tor$-independent with $\scrO_{Z_i}$. This is (3). After derived restriction to $Z$, the sequence becomes
\[
\Omega^{p-1}_{Z_i}(-Z_i) \otimes^{\tL} \scrO_Z \to \Omega^p_Y \otimes^{\tL} \scrO_Z \to \Omega^p_{Z_i}\otimes^{\tL}\scrO_Z \xrightarrow{+1}.
\]
Now (1) shows that the outer two terms have only nontrivial cohomology in degree $0$, hence the same holds for $\Omega^p_Y\otimes^{\tL}\scrO_Z$. In other words, $\Omega^p_Y$ and $\scrO_Z$ are $\Tor$-independent, which shows the only remaining claim (2).
\end{proof}

\begin{lemma}\label{lem_deform_cons_left_inverse}
Using the notations in Construction \ref{cons_deform_cover}, suppose $Z$ is strict-$m$-Du Bois. Then the natural map
\[
\Omega^m_Y \otimes^{\tL} \scrO_Z \to \uOmega^m_Y\otimes^{\tL} \scrO_Z
\]
admits a left inverse.
\end{lemma}

\begin{proof}
Our goal is to construct the inverse component-wise using (\ref{eqn_deform_kahler}) from $\Omega_{Z_i}^m|_Z$. First note that we have a commutative diagram for each $i$ by the functoriality of the Du Bois complex
\[\begin{tikzcd}
\Omega^m_Y|_{Z_i} \ar[r, "\epsilon"] \ar[d, "\res"] &
\uOmega^m_Y\otimes^{\tL} \scrO_{Z_i} \ar[d] \ar[dl, "\Psi_i", dotted]  \\
\Omega^m_{Z_i} \ar[r, "\qis"'] &
\uOmega^m_{Z_i}.
\end{tikzcd}\]
Here we used the $\Tor$-independence in Lemma \ref{lem_deform_tor_ind}(3) to identify $\Omega^m_Y\otimes^{\tL}\scrO_{Z_i}\qis \Omega^m_Y|_{Z_i}$. Since $Z_i$ are strict-$m$-Du Bois by Lemma \ref{lem_product}, we may obtain the map $\Psi_i$ by composing with the inverse of $\Omega^{m}_{Z_i}\to \uOmega^m_{Z_i}$. By the commutativity of the square, we trivially have $\Psi_i\circ \epsilon_Y = \res:\Omega^m_Y|_{Z_i}\to \Omega^m_{Z_i}$, as maps in $D(Z_i)$.

Applying further derived restriction to $Z$ and using Lemma \ref{lem_deform_tor_ind}(1-2), we have the same diagram
\[\begin{tikzcd}
\Omega^m_Y|_Z \ar[r, "\epsilon"] \ar[d] &
\uOmega^m_Y\otimes^{\tL} \scrO_Z \ar[d] \ar[dl, "\Psi_i"] \\
\Omega^m_{Z_i}|_Z \ar[r, "\qis"']&
\uOmega^m_{Z_i}\otimes^{\tL} \scrO_Z,
\end{tikzcd}\]
and all triangles commute. By abuse we will denote the maps with the same notation, understood to be on $Z$ through pullback. By (\ref{eqn_deform_kahler}) the components of $\Omega^m_Y|_Z$ and $\Omega^m_{Z_i}|_Z$ are indexed by appropriate subsets $J\subseteq \{1,\dots, m+1\}$ of size $|J|\le m$ (with $J=\emptyset$ giving $\Omega^m_Z$). For any $J$, choose the index $i(J)=\min(\{1,\dots, m+1\}\setminus J)$, so that $\Omega^m_{Z_{i(J)}}|_Z$ contains a nontrivial $J$-component. $\Psi_{i(J)}$ yields a map after projection from the summand:
\[
\uOmega^m_Y\otimes^{\tL} \scrO_Z \xrightarrow{\Psi_{i(J)}} \Omega^m_{Z_{i(J)}}|_Z \to \Omega^{m-|J|}_Z\wedge dy_J.
\]
Using the description (\ref{eqn_deform_kahler}) of $\Omega^m_Y|_Z$, now as a product, this assembles a map
\[
\Psi: \uOmega^m_Y\otimes^{\tL} \scrO_Z \to \Omega^m_Y|_Z.
\]
This provides a left inverse to $\epsilon: \Omega^m_Y|_Z \to \uOmega^m_Y\otimes^{\tL}\scrO_Z$ in $D(Z)$. Indeed, each composition $\Psi_{i(J)}\circ \epsilon$ agrees with the restriction of differentials, which by the discussion in (\ref{eqn_deform_kahler}) is the identity on the $J$-component since $i(J)\notin J$. Therefore the composition $\Psi\circ \epsilon$ is the identity in every $J$-component, as desired.

Note well that the component-wise inverse construction does not require any compatibility claim for inverses on different $Z_i$, as we made the choice of $i(J)$ to avoid gluing.
\end{proof}

\begin{remark}
The main advantage of Construction \ref{cons_deform_cover} is that along $Z$, any $m$-form of the restrictions must avoid at least one coordinate $dy_i$. This allows for a direct component-wise left inverse in Lemma \ref{lem_deform_cons_left_inverse} of $\Omega^m_Y|_Z \to \uOmega^m_Y\otimes^{\tL}\scrO_Z$, which is otherwise difficult to construct for the ambient $\Omega^m_X|_Z \to \uOmega^m_X\otimes^{\tL}\scrO_Z$. This technique falls in line with the principles that left inverses lead to deformation statements, which has been used by many authors, cf. \cite{Elk81,Kov00,KS16a,KS16b}.
\end{remark}

We will also need a more general form of Elkik-type lemmas for left-inverse criteria on a closed subscheme defined by a regular sequence, instead of a Cartier divisor. Note that the statement below does not directly follow through induction by viewing $Z$ as an intersection of Cartier divisors, since the left inverse does not automatically lift.

\begin{lemma}\label{lem_left_inverse_suffices}
Let $Z\to Y$ be a regular closed immersion (see \cite[Tag 063J]{SP}) of codimension $c$. Suppose $Y$ is strict-$(m-1)$-Du Bois on a neighborhood of $Z$. If
\[
\Omega^m_Y \otimes^{\tL} \scrO_Z \to \uOmega^m_Y \otimes^{\tL} \scrO_Z
\]
admits a left inverse in $D(Z)$, then $\Omega^m_Y\to \uOmega^m_Y$ is a quasi-isomorphism in a neighborhood of $Z$.
\end{lemma}

\begin{proof}
Replace $Y$ by a neighborhood of $Z$ that is strict-$(m-1)$-Du Bois so that we may apply the K\"ahler injectivity Theorem \ref{thm_kahler_injectivity}. Let $C$ be the cone of the natural map $\Omega^m_Y \to \uOmega^m_Y$ so that we have a triangle
\[
\bbD_Y(C) \xrightarrow{a} \bbD_Y(\uOmega^m_Y)\to \bbD_Y(\Omega^m_Y) \xrightarrow{+1}.
\]
Theorem \ref{thm_kahler_injectivity} implies that $a$ is zero on cohomologies. There is an analogous sequence
\[
\bbD_Y(C\otimes^{\tL} \scrO_Z) \xrightarrow{a_Z} \bbD_Y(\uOmega^m_Y\otimes^{\tL} \scrO_Z)\to \bbD_Y(\Omega^m_Y\otimes^{\tL} \scrO_Z) \xrightarrow{+1},
\]
obtained after derived restriction to $Z$. Note that we omit exact pushforward through the closed immersion $Z\to Y$ to view the objects on $Y$. The left inverse given by the hypothesis implies that $a_Z$ is split injective on cohomologies. Observe that derived tensor-Hom adjunction reads
\[
\bbD_Y(C \otimes^{\tL} \scrO_Z)\simeq R\HOM_Y(\scrO_Z, \bbD_Y(C)).
\]
Naturality and the same identity for $\uOmega^m_Y$ gives that $a_Z = R\HOM_Y(\scrO_Z, a)$.

Now suppose that $C\ne 0$ in any neighborhood of $Z$. Since it has bounded cohomologies, so does $\bbD_Y(C)$, thus let $k\in \bbZ$ be the largest such that $h^k(\bbD_Y(C))\ne 0$. Furthermore, since $a$ is zero on cohomologies, in fact it factors through the truncated complex in one degree lower, namely
\[
a: \bbD_Y(C) \to \tau_{\le k-1} \bbD_Y(\uOmega^m_Y)\to \bbD_Y(\uOmega^m_Y),
\]
where we used the same notation as in \cite[Tag 08Q2]{SP} for truncations. Since $Z$ is defined by a regular immersion of codimension $c$, $\scrO_Z$ admits a free resolution of length $c$. Applying $R\HOM_Y(\scrO_Z, -)$ to the above factorization of $a$, we see that
\[
a_Z: \bbD_Y(C \otimes^{\tL} \scrO_Z) \to \bbD_Y(\uOmega^m_Y \otimes^{\tL} \scrO_Z)
\]
factors through a middle term with vanishing cohomologies beyond $i> k+c-1$. In particular, $h^i(a_Z) = 0$ for all $i\ge k+c$. Combining this with the fact that $a_Z$ is split injective, we see that for all $i\ge k+c$,
\[
h^{i}(\bbD_Y(C\otimes^{\tL}\scrO_Z)) = 0.
\]
With this in mind, we now restrict our attention to the local setting. Choose $z\in Z$ where $C$ is nonzero, work over $R:=\scrO_{Y,z}$. Let $I_Z\subseteq \scrO_{Y,z}$ be the ideal defining $Z\subseteq Y$. The vanishing of $h^{k+c}(\bbD_Y(C\otimes^{\tL}\scrO_Z))$ now gives, in the top degree
\[
0 = h^{k+c}(R\Hom_R(R/I_Z, \bbD_R(C_z)))\isom \Ext^c_R(R/I_Z, h^k(\bbD_R(C_z))).
\]
Note that here we also used $h^i(\bbD_R(C_z))=0$ for $i>k$. Since $Z$ is given by a (Koszul-)regular immersion, see \cite[Tag 062F]{SP}, the top $\Ext$ term computes precisely
\[
\Ext^c_R(R/I_Z, h^k(\bbD_R(C_z))) = \frac{h^k(\bbD_R(C_z))}{I_Z\cdot h^k(\bbD_R(C_z))} = 0.
\]
Nakayama's lemma implies $h^k(\bbD_Y(C))=0$ in a neighborhood of $Z$, giving the desired contradiction.
\end{proof}

Now we can arrive at the desired deformation statement for strict-$m$-Du Bois singularities.

\begin{theorem}\label{thm_deform_strict}
Let $Z\subseteq X$ be a Cartier divisor. Suppose $Z$ has strict-$m$-Du Bois singularities. Then $X$ has strict-$m$-Du Bois singularities near $Z$.
\end{theorem}

\begin{proof}
The question is local so we may assume $X$ is affine and $Z=V(t)$ on $X$, so that we may use Construction \ref{cons_deform_cover}. Note that by Lemma \ref{lem_product} each $Z_i$ is strict-$m$-Du Bois.

Using induction on $m$, with $m=0$ provided by \cite[Theorem 4.1]{KS16a}, we may assume $Y$ is strict-$(m-1)$-Du Bois in a neighborhood of each $Z_i$. Then Lemma \ref{lem_deform_tor_ind}(2) and \ref{lem_deform_cons_left_inverse} give a left inverse of
\[
\Omega^m_Y \otimes^{\tL} \scrO_Z \to \uOmega^m_Y \otimes^{\tL} \scrO_Z
\]
and Lemma \ref{lem_left_inverse_suffices} shows that $Y$ is strict-$m$-Du Bois near $Z$ since $Z\to Y$ is a regular immersion of codimension $m+1$. Let $V\subseteq Y$ be the open strict-$m$-Du Bois locus, containing $Z$. As in Construction \ref{cons_deform_cover}, the neighborhood $U\subseteq Y$ defined by $y_i\ne 0$ for $i\ge 2$ is a product $U\isom X\times (\bbA^1\setminus \{0\})^m$. Let $\pi: U \to X$ be the flat projection. 

One more application of Lemma \ref{lem_product} reveals that $\pi(U\cap V)$ is contained in the strict-$m$-Du Bois locus of $X$. It then suffices to show that $\pi(U\cap V)$ contains $Z$. For any closed point $z\in Z$, consider $\{z\}\times \bbA^m\subseteq Z\times \bbA^m =Z_1$. Since $V$ contains the zero section $Z\subseteq Y$, we see that $V\cap (\{z\}\times \bbA^m)$ contains a neighborhood of the origin. In particular, $U\cap V \cap (\{z\}\times \bbA^m) \ne \emptyset$, which implies that $z\in \pi(U\cap V)$. This shows that the strict-$m$-Du Bois locus of $X$ contains $Z$, and concludes the proof.
\end{proof}

\begin{corollary}\label{cor_deform_strict}
Let $\f:X\to B$ be a proper family over a smooth curve $B$, $b\in B$. Suppose $X_b$ has strict-$m$-Du Bois singularities. Then $X_{b'}$ has strict-$m$-Du Bois singularities for $b'$ in a neighborhood of $b\in B$.
\end{corollary}

\begin{proof}
This is a standard consequence of Theorem \ref{thm_deform_strict} and a Bertini-type theorem, e.g. \cite[Corollary 3.3]{SVV23}. By Theorem \ref{thm_deform_strict}, $X$ has strict-$m$-Du Bois singularities in a neighborhood of $X_b$. Since being strict-$m$-Du Bois is an open condition, the set $S\subseteq X$ denoting the non-strict-$m$-Du Bois locus is closed. By properness $\f(S)$ is also closed and excludes $b$ by assumption. Replacing $B$ with $B\setminus \f(S)$, $X$ is strict-$m$-Du Bois. Now apply Bertini to a base-point-free linear system inducing $\f$, we see that a general fiber $X_{b'}$ is strict-$m$-Du Bois, cf. \cite[Corollary 4.2]{KS16a}.
\end{proof}

\medskip

\section{From deformations to base change, and back}\label{section_deformation}

In this section, we will demonstrate and exploit the close connection between base change and deformation. We obtain base change as a consequence of the deformation invariance of strict-$m$-Du Bois singularities. Then as a consequence, we will also obtain deformation results for various other definitions of higher Du Bois singularities.

\begin{lemma}\label{lem_deform_triangle}
Let $Z\subseteq X$ be a Cartier divisor. Suppose $Z$ has strict-$m$-Du Bois singularities. Then the following triangle
\[
\uOmega^{p-1}_Z(-Z) \to \uOmega^{p}_X|^{\tL}_Z \to \uOmega^p_Z \xrightarrow{+1}
\]
is exact for all $p\le m$.
\end{lemma}

\begin{proof}
First, note that by convention the triangle is trivially exact for $p<0$ or $p>\dim X$. 

By Theorem \ref{thm_deform_strict}, in a neighborhood of $Z$, we have $\Omega^p_X\qis \uOmega^p_X$. In particular, the former is a torsion-free sheaf. Again using \cite[Lemma 2.5]{Kov25}, we in fact have $\Omega^p_X|_Z \qis \uOmega^p_X|^{\tL}_Z$. Now simply consider the following diagram
\[\begin{tikzcd}
0 \ar[r] &
\Omega_Z^{p-1}(-Z) \ar[r] \ar[d] & 
\Omega_X^p|_Z \ar[r] \ar[d] & 
\Omega_Z^p \ar[r] \ar[d] &
0 \\
&
\uOmega_Z^{p-1}(-Z) \ar[r]& 
\uOmega_X^p|^{\tL}_Z \ar[r] & 
\uOmega_Z^p. & 
\,
\end{tikzcd}\]
All three vertical maps are quasi-isomorphisms, and the top row is exact by Lemma \ref{lem_conormal}. This shows that the bottom row is quasi-isomorphic to a short exact sequence of sheaves, hence a distinguished triangle.
\end{proof}

\begin{theorem}\label{thm_basechange_strict}
Let $\f:X\to B$ be a dominant morphism from a variety $X$ to a smooth curve $B$, $b\in B$. Suppose $X_b$ has strict-$m$-Du Bois singularities. Then $X_b$ satisfies right-$m$-base change for the relative Du Bois complex, i.e. for all $p\le m$,
\[
\uOmega^{p,+}_{X/B}\otimes^{\tL}_X\scrO_{X_b} \qis \uOmega^p_{X_b}.
\]
\end{theorem}

\begin{proof}
By Corollary \ref{cor_bc_iff_exact}, right-$m$-base change is equivalent to the triangle
\[
\uOmega^{p-1}_{X_b} \to \uOmega^p_X\otimes^{\tL}\scrO_{X_b} \to \uOmega^p_{X_b} \xrightarrow{+1}
\]
being exact for $p\le m$. In our setting, $\scrO_{X_b}(-X_b)\isom \scrO_{X_b}$ and this triangle agrees with the one in Lemma \ref{lem_deform_triangle}. More precisely, in both scenarios, the first map comes from Lemma \ref{lem_conormal_map} and the second map is the restriction. Therefore, since $X_b$ is strict-$m$-Du Bois, we conclude that $X_b$ satisfies right-$m$-base change.
\end{proof}

If we assume that the fiber $X_b$ is strict-$m$-Du Bois for all $m$, then we can also obtain the left-base change statement, thereby answering the question of Kov\'acs--Taji on base change.

\begin{theorem}\label{thm_basechange_strict_all}
Let $\f:X\to B$ be a dominant morphism from a variety $X$ to a smooth curve $B$, $b\in B$. Suppose $X_b$ has strict-$m$-Du Bois singularities for all $m\in \bbN$. Then $X_b$ satisfies base change for the relative Du Bois complex, i.e. for all $p\in\bbZ$,
\[
\uOmega^{p}_{X/B}\otimes^{\tL}_X\scrO_{X_b} \qis \uOmega^p_{X_b},
\]
where $\uOmega^{p}_{X/B}$ can be taken to be either $\uOmega^{p,+}_{X/B}$ or $\uOmega^{p,-}_{X/B}$.
\end{theorem}

\begin{proof}
This follows from Theorem \ref{thm_basechange_strict} and Corollary \ref{cor_leftright_agree}.
\end{proof}

\subsection*{Deformation results assuming base change}

For the next part, we will derive deformation results as a consequence of base change. In light of Theorem \ref{thm_basechange_strict}, we only know base change for fibers with strict-$m$-Du Bois singularities. However, it is still useful to record the connection. If anything, it adds to the importance of studying base change behavior in more general situations.

We first recall a useful $\Tor$-independence statement. 

\begin{lemma}\label{lem_h0_tor_ind}
Let $Z\subseteq X$ be a Cartier divisor. Then $h^0(\uOmega^p_X)$ and $\scrO_Z$ are $\Tor$-independent. That is
\[
h^0(\uOmega^p_X)|_Z \qis h^0(\uOmega^p_X)|^{\tL}_Z.
\]
\end{lemma}

\begin{proof}
This follows by combining the fact that $h^0(\uOmega^p_X)$ is torsion-free with \cite[Lemma 2.5, 2.6]{Kov25}.
\end{proof}

\begin{proposition}\label{prop_deform_pre}
Let $Z\subseteq X$ be a Cartier divisor in a complex variety $X$. Suppose $Z$ satisfies either left- or right-$m$-base change (Definition \ref{def_bc_divisor}) and is pre-$m$-Du Bois. Then $X$ is pre-$m$-Du Bois near $Z$.
\end{proposition}

\begin{proof}
The question is local on $X$ and $Z$, so as in Definition \ref{def_bc_divisor}, after possibly shrinking $X$, we let $\f:X\to \bbA^1$ be a dominant morphism such that $X_0=Z$ satisfies either left- or right-$m$-base change.

Fix $p\le m$. The input from base change is through Lemma \ref{lem_conormal_triangle}, giving a commutative diagram
\[\begin{tikzcd}
0 \ar[r] &
h^0(\uOmega_Z^{p-1}) \ar[r] \ar[d] & 
h^0(\Omega_X^p|^{\tL}_Z) \ar[r] \ar[d] & 
h^0(\Omega_Z^p) \ar[r] \ar[d] &
0 \\
&
\uOmega_Z^{p-1} \ar[r]& 
\uOmega_X^p|^{\tL}_Z \ar[r] & 
\uOmega_Z^p, & 
\,
\end{tikzcd}\]
where the top row is short exact since $Z$ is pre-$m$-Du Bois. By the five lemma, the map $h^0(\Omega_X^p|^{\tL}_Z) \to \uOmega^p_X|^{\tL}_Z$ is a quasi-isomorphism, so $h^i(\uOmega^p_X|^{\tL}_Z)=0$ for $i>0$. Therefore, we conclude that $X$ is pre-$m$-Du Bois in a neighborhood of $Z$ by using Lemma \ref{lem_Nakayama}.
\end{proof}

\begin{proposition}\label{prop_deform_weak}
Let $Z\subseteq X$ be a Cartier divisor in a complex variety $X$. Suppose $Z$ satisfies either left- or right-$m$-base change (Definition \ref{def_bc_divisor}) and is weakly-$m$-Du Bois. Then $X$ is weakly-$m$-Du Bois near $Z$.
\end{proposition}

\begin{proof}
As before, after possibly shrinking $X$, we let $\f:X\to \bbA^1$ be a dominant morphism such that $X_0=Z$ satisfies either left- or right-$m$-base change.

By Proposition \ref{prop_deform_pre} $X$ is pre-$m$-Du Bois near $Z$. Fix $p\le m$. We need to show that $X$ is seminormal and $h^0(\uOmega^p_X)$ is $S_2$. One particular consequence of $X$ having pre-$m$-Du Bois singularities is the quasi-isomorphisms
\[
h^0(\uOmega^p_X|^{\tL}_Z) \qis h^0(\uOmega^p_X)|^{\tL}_Z \qis h^0(\uOmega^p_X)|_Z,
\]
in conjunction with Lemma \ref{lem_h0_tor_ind}. Since $Z$ is pre-$m$-Du Bois, the zeroth cohomology of Lemma \ref{lem_conormal_triangle} is short exact:
\begin{equation*}
0\to h^0(\uOmega^{p-1}_Z) \to h^0(\uOmega^p_X)|_Z \to h^0(\uOmega^p_Z) \to 0.
\end{equation*}
By depth counting (see \cite[Tag 00LX]{SP}), we have that $h^0(\uOmega^p_X)|_Z$ is $S_1$ on $Z$. 

Now, recall that $h^0(\uOmega^p_X)$ is torsion-free on $X$, so in particular $Z$ contains no associated primes of $h^0(\uOmega^p_X)$. As $Z$ is a Cartier divisor, again by depth counting, we have that $h^0(\uOmega^p_X)$ is $S_2$ on $X$ near $Z$ for $p\le m$. Finally, seminormality lifts from Cartier divisors by \cite{Hei08}. This concludes that $X$ is weakly-$m$-Du Bois near $Z$.
\end{proof}

\begin{proposition}\label{prop_deform_codim}
Let $Z\subseteq X$ be a Cartier divisor in a pure-dimensional complex variety $X$. Suppose $Z$ satisfies either left- or right-$m$-base change (Definition \ref{def_bc_divisor}) and is $m$-Du Bois. Then $X$ is $m$-Du Bois near $Z$.
\end{proposition}

\begin{proof}
As before, after possibly shrinking $X$, we let $\f:X\to \bbA^1$ be a dominant morphism such that $X_0=Z$ satisfies either left- or right-$m$-base change.

The assumption of $Z$ being $m$-Du Bois includes $Z$ being reduced, hence $S_1$. Since $Z$ is a Cartier divisor in $X$, $X$ is $S_2$ and $\Sing X\cap Z \subseteq \Sing Z$. Therefore in a neighborhood of $Z$, $\codim_X \Sing X\ge 2$, and we see that $X$ is normal. We will replace $X$ with a neighborhood of $Z$ such that this holds.

For normal varieties, the only remaining difference between weakly-$m$-Du Bois and $m$-Du Bois is the codimension bound on $\Sing X$. So keeping Proposition \ref{prop_deform_weak} in mind, we only need to show that $\codim_X \Sing X\ge 2m+1$ near $Z$.

Let $W\subseteq \Sing X$ be a component meeting $Z$. Then since $Z$ is Cartier, $W\cap Z\subseteq \Sing Z$ and $\dim W\cap Z \ge \dim W -1$. On the other hand, by assumption $\codim_Z \Sing Z \ge 2m+1$. Combining this all together, we obtain
\[
\codim_X W +1 \ge \codim_X W\cap Z \ge \codim_X\Sing Z = \codim_Z \Sing Z + 1 \ge 2m+2.
\]
This shows the desired codimension bound near $Z$.
\end{proof}

\begin{remark}
Note that the above deformation statements do not require the higher Kov\'acs--Schwede injectivity theorem. This is because our setup assumes base change, which provides a direct comparison between $\uOmega^p_X|^{\tL}_Z$ and $\uOmega^p_Z$. The Hodge-theoretic input of the injectivity theorem is now indirectly reflected through the behavior of $\uOmega^p_{X/B}$ and the associated filtrations. The Elkik-type argument is also implicitly present through the use of Lemma \ref{lem_Nakayama}. In any case, we see that base change is intimately connected to deformations through Theorem \ref{thm_basechange_strict}.
\end{remark}

\medskip

\section{Local-freeness theorem}\label{section_localfree}

In this section, we will show constancy of higher Hodge numbers in strict-$m$-Du Bois families. Building on the work of Friedman--Laza, a critical piece of information is the flatness of relative K\"ahler differentials over the base $B$; see \cite[Theorem 2.5]{FL22}. In fact, their argument shows the following.

\begin{theorem}[\cite{FL22}]\label{thm_FL}
Let $\f:X \to B$ be a flat proper family of finite type $\bbC$-schemes, and $b\in B$. Suppose $X_b$ has strict-$m$-Du Bois singularities and $\Omega^p_{X/B}$ is flat over $B$ for $p\le m$. Then $R^i\f_*\Omega^p_{X/B}$ is locally free and compatible with base change for all $i\ge 0$ and $0\le p\le m$, in a neighborhood of $b$.
\end{theorem}

\begin{proof}
Assuming flatness as a replacement for \cite[Theorem 2.5]{FL22}, the proof of \cite[Theorem 1.2]{FL22} does not require the lci assumption; see \cite[\S 4]{FL22}.
\end{proof}

\begin{remark}
Even when the base $B$ is a smooth curve, flatness is not automatic as $\Omega^p_{X/B}$ can have vertical torsion for $p\ge 1$.
\end{remark}

Using torsion-freeness of the K\"ahler differentials for strict higher Du Bois singularities, we can obtain flatness directly without using the lci hypothesis in \cite{FL22}.

\begin{lemma}\label{lem_rel_Kahler_flat}
Let $\f:X\to B$ be a flat morphism, locally of finite presentation, and $b\in B$ a closed point. Suppose the fiber $X_b$ is strict-$m$-Du Bois. Then $\Omega^p_{X/B}$ is flat over $B$ at every point of $X_b$, for $p\le m$. 

If $\f$ is in addition proper, then $\Omega^p_{X/B}$ is flat over $B$ in a neighborhood of $b\in B$, for $p\le m$.
\end{lemma}

\begin{proof}
Note that $X_b$ is reduced by the strict-$m$-Du Bois assumption. Further, $\Omega^p_{X_b}\isom h^0(\uOmega^p_{X_b})$ is torsion-free for $p\le m$. Therefore every associated point of $\Omega^p_{X/B}|_{X_b} \isom \Omega^p_{X_b}$ is a generic point of an irreducible component of $X_b$. This implies that $\Omega^p_{X/B}$ is flat at every point of $X_b$ over $B$, using the criterion in \cite[Tag 05UA]{SP}. More specifically, by generic smoothness on $X_b$ reduced and flatness of $\f$, we have that $\f$ is smooth at the generic points of $X_b$. This in particular implies that $\Omega^p_{X/B}$ is locally free hence flat over $B$ at generic points of $X_b$, as desired.

The last claim follows since the locus in $X$ where $\Omega^p_{X/B}$ is flat over $B$ is open. Properness of $\f$ gives that the image of its complement is closed. This completes the proof.
\end{proof}

\begin{theorem}\label{thm_local_free}
Let $\f:X \to B$ be a flat proper family of finite type $\bbC$-schemes, and $b\in B$ a closed point. Suppose $X_b$ has strict-$m$-Du Bois singularities for some $m\in \bbN$. Then $R^i\f_*\Omega^p_{X/B}$ is locally free and compatible with base change for all $i\ge 0$ and $0\le p\le m$, in a neighborhood of $b$.
\end{theorem}

\begin{proof}
This follows directly from Theorem \ref{thm_FL} and Lemma \ref{lem_rel_Kahler_flat}.
\end{proof}

As standard corollaries (cf. \cite[Corollary 1.4]{FL22}), we collect several direct consequences of this result for easy applications.

\begin{corollary}\label{thm_constant_hodge}
Let $\f:X \to B$ be a flat proper family of finite type $\bbC$-schemes, and $b\in B$ a closed point. Suppose $X_b$ has strict-$m$-Du Bois singularities for some $m\in \bbN$. Then for every (not necessarily closed) point $b'$ in a neighborhood of $b\in B$, we have
\[
\dim_{\bbC}H^q(X_{b},\Omega^p_{X_{b}}) = \dim_{k(b')}H^q(X_{b'}, \Omega^p_{X_{b'}/k(b')})
\]
for all $q\ge 0$ and $0\le p\le m$.\hfill\qedsymbol
\end{corollary}

\begin{corollary}\label{cor_constant_hodge_sm}
Let $\f:X \to B$ be a flat proper family of finite type $\bbC$-schemes, and $b\in B$ a closed point. Assume $B$ is irreducible. Suppose $X_{b}$ has strict-$m$-Du Bois singularities for some $m\in \bbN$. Then, for every $b'\in B$ such that $X_{b'}$ is smooth, we have
\[
\dim_{\bbC}H^q(X_{b},\Omega^p_{X_{b}}) = \dim_{k(b')}H^q(X_{b'}, \Omega^p_{X_{b'}/k(b')})
\]
for all $q\ge 0$ and $0\le p\le m$.
\end{corollary}

\begin{proof}
A smooth fiber $X_{b'}$ is strict-$m$-Du Bois for all $m$. Since $B$ is irreducible, the neighborhood around $b, b'\in B$ given by Corollary \ref{thm_constant_hodge} intersects nontrivially, hence the claim.
\end{proof}

\begin{corollary}\label{cor_constant_hodge_all}
Let $\f:X \to B$ be a flat proper family of finite type $\bbC$-schemes. Assume $B$ is connected. Suppose for some fixed $m\in \bbN$, $X_{b}$ has strict-$m$-Du Bois singularities for all closed $b\in B$. Then,
\[
\dim_{k(b)}H^q(X_{b},\Omega^p_{X_{b}/k(b)}) = \dim_{k(b')}H^q(X_{b'}, \Omega^p_{X_{b'}/k(b')})
\]
for all $b, b'\in B$, $q\ge 0$ and $0\le p\le m$.
\end{corollary}

\begin{proof}
Corollary \ref{thm_constant_hodge} applied to every closed fiber implies that the level sets of $h^{p,q}$ are open in $B$. By connectedness, there is only one constant value.
\end{proof}

We conclude with an application to Calabi--Yau deformations using the local-freeness theorem. As in \cite[Definition 4.4]{FL22}, a canonical Calabi--Yau variety is a proper, Gorenstein variety with canonical singularities, and such that $\omega_X\isom \scrO_X$.

\begin{corollary}\label{cor_unobstructed}
Let $X$ be a canonical Calabi--Yau variety. Suppose $X$ has strict-$1$-Du Bois singularities. Then the deformation functor $\mathbf{Def}(X)$ is unobstructed.
\end{corollary}

\begin{proof}
The proof of \cite[Corollary 1.5]{FL22} applies verbatim, replacing \cite[Theorems 2.5 and 1.2]{FL22} with Lemma \ref{lem_rel_Kahler_flat} and Theorem \ref{thm_local_free}, respectively.
\end{proof}

\smallskip
\subsection{Comparison with relative Du Bois complex.}

Let $f:X\to B$ be a dominant morphism to a smooth curve $B$. Here we also collect several lemmas regarding properties of $\uOmega^{p, *}_{X/B}$ and comparison with relative K\"ahler differentials $\Omega^p_{X/B}$.

\begin{lemma}\label{lem_relDB_tf_left}
$h^0(\uOmega^{p,-}_{X/B})$ is torsion-free for all $p\in \bbZ$.
\end{lemma}

\begin{proof}
When $p\ge \dim X$, by definition $\uOmega^{p,-}_{X/B}=0$. By Lemma \ref{lem_relDB_coh_amp}, $h^{-1}(\uOmega^{p,-}_{X/B}) = 0$ for all $p\le \dim X$. Taking cohomology sheaves of the fundamental triangle in Theorem \ref{thm_2relDB}(2), we have
\[
h^{-1}(\uOmega^{p,-}_{X/B}) = 0 \to h^0(\uOmega^{p-1,-}_{X/B}\otimes \f^*\omega_B) \to h^0(\uOmega^p_X).
\]
Therefore $h^0(\uOmega^{p-1,-}_{X/B}\otimes \f^*\omega_B)$ is torsion-free as a subsheaf of the torsion-free sheaf $h^0(\uOmega^p_X)$. We conclude that $h^0(\uOmega^{p,-}_{X/B})$ is torsion-free for all $p\in \bbZ$.
\end{proof}

For the right relative Du Bois complex, due to its construction as taking cones ``towards the right'', torsion-freeness and depth do not behave well. Nonetheless, we still obtain flatness when the fiber is strict-$m$-Du Bois.

\begin{lemma}\label{lem_relDB_tf_right}
Suppose $X_b$ has strict-$m$-Du Bois singularities. Then $h^0(\uOmega^{p,+}_{X/B})$ is flat over $B$ at every point of $X_b$, for $p\le m$. Furthermore, $h^0(\uOmega^{p,+}_{X/B})$ is torsion-free in a neighborhood of $X_b$ for $p\le m-1$.
\end{lemma}

\begin{proof}
Fix $p\le m$. By Theorem \ref{thm_basechange_strict}, $X_b$ satisfies right-$m$-base change, that is, 
\[
\uOmega^{p,+}_{X/B}\otimes^{\tL} \scrO_{X_b}\qis \uOmega^p_{X_b}\qis \Omega^p_{X_b}.
\]
In particular, $h^i(\uOmega^{p,+}_{X/B}\otimes^{\tL} \scrO_{X_b})=0$ for $i\ne 0$. Therefore, by Lemma \ref{lem_Nakayama},
\[
h^i(\uOmega^{p,+}_{X/B}) = 0
\]
for $i\ne 0$ in a neighborhood of $X_b$. In other words, in a neighborhood of $X_b$, we have
\[
h^0(\uOmega^{p,+}_{X/B}) \qis \uOmega^{p,+}_{X/B}.
\]
Then, we see that $h^{-1}(h^0(\uOmega^{p,+}_{X/B}) \otimes^{\tL} \scrO_{X_b})=0$, which in turn implies that
\[
\Tor_1^X(h^0(\uOmega^{p,+}_{X/B}), \scrO_{X_b}) = 0.
\]
This is precisely saying that multiplication by a local parameter of $b\in B$ is injective on $h^0(\uOmega^{p,+}_{X/B})$, so it is a torsion-free $\scrO_{B,b}$-module. Therefore, $h^0(\uOmega^{p,+}_{X/B})$ is flat over the DVR $\scrO_{B,b}$, proving the first claim.

Now still fix $p\le m$. Taking cohomology sheaves of the fundamental triangle in Theorem \ref{thm_2relDB}(2), we see
\[
h^{-1}(\uOmega^{p,+}_{X/B}) = 0 \to h^0(\uOmega^{p-1,+}_{X/B}) \to h^0(\uOmega^p_X).
\]
Hence in a neighborhood of $X_b$, $h^0(\uOmega^{p-1,+}_{X/B})$ is a subsheaf of the torsion-free sheaf $h^0(\uOmega^p_X)$, proving the second claim.
\end{proof}

The next result can be interpreted as a condition for the family to be strict-$m$-Du Bois near a fiber.

\begin{proposition}\label{lem_relDB_strictDB}
Let $f:X\to B$ be a dominant morphism over a smooth curve $B$, $b\in B$. Suppose the fiber $X_b$ is strict-$m$-Du Bois. Then in a neighborhood of $X_b$ and for $p\le m$, we have
\begin{enumerate}
\item $\Omega^p_{X/B}\qis \uOmega^{p,+}_{X/B}$;

\item if $X_b$ also satisfies left-$m$-base change, then $\Omega^p_{X/B}\qis \uOmega^{p,-}_{X/B}$.
\end{enumerate}
\end{proposition}

\begin{proof}
The cases for $\uOmega^{p,+}_{X/B}$ and $\uOmega^{p,-}_{X/B}$ are largely similar, so we use $*\in \{+, -\}$ to denote either and point out the difference as needed. 

Fix $p\le m$. The strict-$m$-Du Bois and $m$-base change assumption (either directly or through Theorem \ref{thm_basechange_strict}) together imply that
\[
h^i(\uOmega^{p,*}_{X/B}\otimes^{\tL} \scrO_{X_b}) =0
\] 
for $i\ne 0$. Therefore using Lemma \ref{lem_Nakayama}, we only need to show that $\Omega^p_{X/B}\isom h^0(\uOmega^{p,*}_{X/B})$ near $X_b$.

Taking derived tensor with $\uOmega^{p,*}_{X/B}$, we have a triangle
\[
\uOmega^{p,*}_{X/B}\otimes \scrO_X(-X_b) \to \uOmega^{p,*}_{X/B} \to \uOmega^{p,*}_{X/B} \otimes^{\tL} \scrO_{X_b} \xrightarrow{+1}
\]
and take cohomology to obtain an exact sequence on $X$
\[
h^0(\uOmega^{p,*}_{X/B}(-X_b)) \xrightarrow{\cdot t} h^0(\uOmega^{p,*}_{X/B}) \to h^0(\uOmega^{p,*}_{X/B} \otimes^{\tL} \scrO_{X_b}).
\]
The first map is given by multiplication by a local equation $t$ of $X_b$, a Cartier divisor. Therefore the cokernel of $\cdot t$ injects into the third term, giving us
\[
h^0(\uOmega^{p,*}_{X/B})|_{X_b}\hookrightarrow h^0(\uOmega^{p,*}_{X/B} \otimes^{\tL} \scrO_{X_b}) \isom h^0(\uOmega^p_{X_b})
\]
where the last isomorphism follows by left- or right-$m$-base change. We consider the following diagram
\[\begin{tikzcd}
\Omega^p_{X/B}|_{X_b} \ar[r] \ar[d, "\isom"] &
h^0(\uOmega^{p,*}_{X/B})|_{X_b} \ar[d, hookrightarrow] \\
\Omega^p_{X_b} \ar[r, "\isom"] &
h^0(\uOmega^p_{X_b}).
\end{tikzcd}\]
This commutes by the naturality of the base change map. The left and bottom marked isomorphisms come from the fact that K\"ahler differentials satisfy base change and the strict-$m$-Du Bois hypothesis. Therefore, all four maps are isomorphisms of sheaves.

Now, multiplication by $t$ is injective on $h^0(\uOmega^{p,*}_{X/B})$. For the left complex, this is torsion-freeness from Lemma \ref{lem_relDB_tf_left} and $\Tor$-independence from \cite[Lemma 2.5]{Kov25}. For the right complex, this is Lemma \ref{lem_relDB_tf_right}. Hence we have a diagram with exact rows:
\[\begin{tikzcd}
&
\Omega^p_{X/B} \ar[r, "\cdot t"] \ar[d, "\phi_p"] &
\Omega^p_{X/B} \ar[r] \ar[d, "\phi_p"] &
\Omega^p_{X/B}|_{X_b} \ar[r] \ar[d, "\isom"] &
0 \\
0 \ar[r] &
h^0(\uOmega^{p,*}_{X/B}) \ar[r, "\cdot t"] &
h^0(\uOmega^{p,*}_{X/B}) \ar[r] &
h^0(\uOmega^{p,*}_{X/B})|_{X_b} \ar[r] &
0.
\end{tikzcd}\]
Despite the top sequence not being left exact a priori, a standard diagram chase similar to the snake lemma gives two surjections
\[
\ker \phi_p \xrightarrow{\cdot t} \ker \phi_p, \quad \text{and} \quad \coker\phi_p \xrightarrow{\cdot t}\coker \phi_p.
\]
We may now use Nakayama's lemma to conclude that $\phi_p$ is an isomorphism. This concludes the proof.
\end{proof}

\medskip

\section{Failure of base change}\label{section_example}

\cite[Example 4.7]{JK25} already exhibited an example of failure of base change for the top degree of the (left) relative Du Bois complex $\uOmega^n_{X/B}$. In this section we sharpen this result by illustrating that right-$m$-base change can fail for every degree beyond the strict-$m$-Du Bois hypothesis on the fiber. Additionally, the notion of left base change behaves substantially worse for any fiber that is not strict-$m$-Du Bois for all $m\in \bbN$.

First we show and recall some general facts. As before, $\f: X\to B$ is a dominant morphism from a variety $X$ to a smooth curve $B$. Fix $b\in B$ and consider the fiber $Z:= X_b$. Using $*\in \{+, -\}$ to denote either complex, the fundamental triangle in Theorem \ref{thm_2relDB}(2) reads
\[
\uOmega_{X/B}^{p-1, *} \otimes \f^*\omega_B \to \uOmega_X^p \to \uOmega_{X/B}^{p, *} \xrightarrow{+1}.
\]
Pulling this back to the fiber $X_b$ gives
\[
\uOmega_{X/B}^{p-1, *}|^{\tL}_{X_b} \to \uOmega_X^p|^{\tL}_{X_b} \to \uOmega_{X/B}^{p, *}|^{\tL}_{X_b} \xrightarrow{+1}.
\]
For simplicity in computations, we denote by $\scrF^\kdot|^{\tL}_{X_b}:= \scrF^\kdot\otimes^{\tL}\scrO_{X_b}$. Using induction either from $p=-1$ or $p=\dim X$, we immediately obtain:

\begin{lemma}\label{lem_rel_pullback}
Let $b\in B$, and denote $n:=\dim \f=\dim X-1$. For the left relative Du Bois complex, we have
\begin{align*}
\uOmega_{X/B}^{p,-}|^{\tL}_{X_b} & \qis \coCone(\uOmega^{p+1}_X \to \uOmega^{p+1, -}_{X/B})|^{\tL}_{X_b} \\
& \qis \coCone(\uOmega^{p+1}_X \to \coCone(\uOmega^{p+2}_X \to \cdots \coCone(\uOmega^n_X\to \uOmega^{n+1}_X) \cdots))|^{\tL}_{X_b}
\end{align*}
for all $p \le n$, where $\coCone(a) = \Cone(a)[-1]$. For the right relative Du Bois complex, we have
\begin{align*}
\uOmega_{X/B}^{p,+}|^{\tL}_{X_b} & \qis \Cone(\uOmega^{p-1, +}_{X/B}\to \uOmega^{p}_X)|^{\tL}_{X_b} \\
& \qis \Cone(\Cone(\cdots \Cone(\uOmega^0_X\to \uOmega^1_X) \cdots \to \uOmega^{p-1}_X ) \to \uOmega^p_X)|^{\tL}_{X_b}
\end{align*}
for all $p \ge 0$. If we assume that $X$ is smooth, then in a neighborhood of $b\in B$, we have
\[
\uOmega_{X/B}^{p,-}|^{\tL}_{X_b} \qis \left[ \Omega^{p+1}_X \xrightarrow{\wedge_{p+1}} \cdots \xrightarrow{\wedge_{n}} \Omega^{n+1}_X \right] |_{X_b}
\]
for all $p\le n$, with $\Omega^{p+1}_X$ in cohomological degree $0$. We also have
\[
\uOmega_{X/B}^{p,+}|^{\tL}_{X_b} \qis \left[ \scrO_X \xrightarrow{\wedge_0} \Omega^1_X \xrightarrow{\wedge_1} \cdots \xrightarrow{\wedge_{p-1}} \Omega^p_X \right] |_{X_b}
\]
for all $p\ge 0$, with $\Omega^p_X$ in cohomological degree $0$. 
\end{lemma}

\begin{proof}
The first two follow directly from the defining triangle. Pick a neighborhood of $b\in B$ to trivialize $\f^*\omega_B$, the last two claims then follow when $X$ is smooth since each $\Omega^p_X$ is locally free. We only require a neighborhood of $b\in B$ to suppress $\f^*\omega_B$ for simplicity of notation.
\end{proof}

We are interested in the situation where $Z=X_b$ is a (simple) normal crossing divisor in $X$. Recall the following well-known fact, which is a direct computation from the cubical hyperresolution of an snc variety given by the \v{C}ech resolution; see \cite[Example 7.23]{PS08}.

\begin{lemma}\label{lem_snc_cech}
Let $Z=\bigcup Z_i$ be an snc variety. Then $\uOmega_Z^p \qis \Omega_Z^p/{\rm torsion}$ is a sheaf and
\[
0\to \uOmega^p_Z \to \bigoplus_i \Omega^p_{Z_i} \to \bigoplus_{i<j} \Omega^p_{Z_i\cap Z_j} \to \dots \to \Omega^p_{\cap_i Z_i} \to 0
\]
is a resolution of $\uOmega_Z^p$. In particular, we have
\[
\uOmega_Z^n \qis \bigoplus_i \Omega_{Z_i}^n,
\]
where $n=\dim Z$.
\hfill\qedsymbol
\end{lemma}

The situation for right-$m$-base change is surprisingly clean. Indeed, the $m=0$ case is already shown in Lemma \ref{lem_basechange_0}.

\begin{proposition}\label{prop_right_bc_fails}
Let $\f: X\to B$ be a dominant morphism from a variety $X$ to a smooth curve $B$, $b\in B$. Suppose $X$ is smooth, and suppose $X_b$ is strict-$(m-1)$-Du Bois. Then 
\begin{enumerate}
\item $X_b$ satisfies right-$m$-base change if and only if $X_b$ is strict-$m$-Du Bois;

\item $X_b$ satisfies left-$m$-base change only if $X_b$ is strict-$m$-Du Bois.
\end{enumerate}
\end{proposition}

\begin{proof}
The first claim implies the second claim via Proposition \ref{prop_left_implies_right}. The if direction of the first claim also follows from the base change Theorem \ref{thm_basechange_strict}.

Fix $p\le m$. Since $Z:= X_b\subseteq X$ is a Cartier divisor that is strict-$(m-1)$-Du Bois, Lemma \ref{lem_conormal} gives a short exact sequence
\[
0\to \Omega^{p-1}_Z \xrightarrow{\wedge_{p-1}|_Z} \Omega^p_X|_Z \to \Omega^p_Z \to 0.
\]
We have used the fact that $\scrO_Z(-Z)\isom \f_b^*\omega_B \isom \scrO_Z$ to identify the conormal map as wedging with differentials from the base. Furthermore, again by Theorem \ref{thm_basechange_strict}, we have
\[
\uOmega_{X/B}^{p-1,+}|^{\tL}_{Z} \qis \uOmega^{p-1}_{Z}\qis \Omega^{p-1}_Z.
\]
Therefore, using Lemma \ref{lem_rel_pullback}, we see that
\[
\uOmega_{X/B}^{p,+}|^{\tL}_{Z} \qis \Cone(\uOmega_{X/B}^{p-1,+}|^{\tL}_{Z}\to \Omega^p_X|_Z) \qis \Omega^p_Z.
\]
We conclude that $Z$ satisfying right-$m$-base change must imply that $\uOmega^p_{Z}\qis \Omega^p_{Z}$.
\end{proof}

\begin{remark}\label{rem_bc_fails}
It is natural to ask whether base change can hold under the weaker assumption of (weakly-)$m$-Du Bois + normal to exclude snc examples (cf. \cite[Example 4.7]{JK25}), as a strengthening of Theorem \ref{thm_basechange_strict}. In light of Proposition \ref{prop_right_bc_fails}, we see that there is no hope for right-$m$-base change to hold for any notion weaker than strict-$m$-Du Bois without additional assumptions on the family. By Proposition \ref{prop_left_implies_right}, the stronger notion of left-$m$-base change therefore also cannot hold.
\end{remark}

For left base change, we fix the setup for the example we will use for the rest of this section to illustrate its failure.

\begin{example}\label{notation_snc_ex}
Let $B = \bbA^1_t$, $n\ge 1$, and define
\[
X := V(xy-tz^2) \subset \bbP^2_{[x,y,z]} \times \bbP^{n-1} \times B.
\]
Let $\f:X\to B$ be the projection. We easily see that the total space $X$ is smooth, and $\f$ is flat and projective of relative dimension $n$. Furthermore, let $b=0\in \bbA^1$ and
\[
Z := \f^{-1}(b) = V(xy) \subseteq \bbP^2_{[x,y,z]} \times \bbP^{n-1}.
\]
We see that $Z$ is a simple normal crossing divisor in $X$, and has two components isomorphic to $\bbP^1\times \bbP^{n-1}$. Recall that $Z$ is weakly-$m$-Du Bois for all $m$ but is strict-$m$-Du Bois if and only if $m=0$.
\end{example}

\begin{lemma}\label{lem_snc_ex_local}
In Example \ref{notation_snc_ex}, the map
\[
\wedge_{n}: \Omega^n_X \otimes \f^*\omega_B \to \Omega^{n+1}_X
\]
is not surjective after restriction to $Z$.
\end{lemma}

\begin{proof}
We show this using a local calculation. First we consider the case $n=1$, so $X=V(xy-tz^2)\subset \bbP^2_{[x,y,z]}\times B$. Fix chart $z\ne 0$ and continue using $x, y$ for the affine coordinates. The map $f$ is given by the equation $t=xy$. Therefore, for any $\alpha = a\, dx + b\, dy\in \Omega^1_X$, the wedge map is given by
\[
\wedge_1(\alpha \otimes dt) = \alpha \wedge dt = \alpha \wedge (y\, dx + x\, dy) = (ax-by)\, dx\wedge dy.
\]
We see that the image of $\wedge_1$ is contained in the ideal $(x, y)\cdot \Omega^2_X$. Since this fails to be surjective at the node in $Z$, the restriction to $Z$ remains not surjective.

Now, the same calculation works for arbitrary $n$. Since $X$ is a product of $X_2:=V(xy-tz^2)\subset \bbP^2_{[x,y,z]}\times B$ with $\bbP^{n-1}$, we have a K\"unneth product decomposition
\[
\Omega^n_X \isom \Omega^1_{X_2}\boxtimes \Omega^{n-1}_{\bbP^{n-1}} \oplus \Omega^2_{X_2}\boxtimes \Omega^{n-2}_{\bbP^{n-1}} \quad \text{and} \quad \Omega^{n+1}_X \isom \Omega^2_{X_2}\boxtimes \Omega^{n-1}_{\bbP^{n-1}}.
\]
For the component $\Omega^2_{X_2}\boxtimes \Omega^{n-2}_{\bbP^{n-1}}$, wedging with the pullback of $dt$ kills any $2$-form on $X_2$. Then, $\wedge_{n}$ exhibits the same failure of surjectivity on the remaining component $\Omega^1_{X_2}\boxtimes \Omega^{n-1}_{\bbP^{n-1}} \to \Omega^2_{X_2}\boxtimes \Omega^{n-1}_{\bbP^{n-1}}$.
\end{proof}

\begin{proposition}\label{prop_left_bc_fails}
Let $\f:X\to B$ be as in Example \ref{notation_snc_ex}. Then for all $0\le p< n$,
\[
h^{n-p}(\uOmega_{X/B}^{p,-}|^{\tL}_{X_b}) \not \isom h^{n-p}(\uOmega^p_{X_b}) = 0.
\]
Consequently, left-$m$-base change fails for $X_b$ for every $m\ge 0$, despite $X_b$ being (0-)Du Bois.
\end{proposition}

\begin{proof}
Using Lemma \ref{lem_rel_pullback}, we see that
\[
h^{n-p}(\uOmega_{X/B}^{p,-}|^{\tL}_{X_b}) \isom \coker (\wedge_n: \Omega^n_X \to \Omega^{n+1}_X)|_{X_b},
\]
but this is non-zero by Lemma \ref{lem_snc_ex_local}. On the other hand, for all $p< n$, Lemma \ref{lem_snc_cech} gives $\uOmega^p_{X_b}$ is a single sheaf. Therefore left-$m$-base change fails for every $m \ge 0$.
\end{proof}

We summarize the failure of left and right base change in the following proposition.

\begin{proposition}\label{prop_bc_fail_summary}
There exists a family $\f:X \to B$ and a fiber $X_b$ for $b\in B$ such that
\begin{enumerate}
\item $X_b$ is weakly-$m$-Du Bois for all $m\in \bbN$, but does not satisfy right-$1$-base change.

\item $X_b$ is weakly-$m$-Du Bois for all $m\in \bbN$, but does not satisfy left-$0$-base change.
\end{enumerate}
\end{proposition}

\begin{proof}
We may use Example \ref{notation_snc_ex}. Then, (1) follows from Proposition \ref{prop_right_bc_fails} and (2) is precisely Proposition \ref{prop_left_bc_fails}.
\end{proof}

\subsection*{Failure for a $1$-Du Bois divisor.} As base change and deformations are closely connected, these examples suggest that the corresponding deformation problem may also fail. In fact, here we produce many cone examples whose total space $X$ is not pre-$1$-Du Bois but has a $1$-Du Bois Cartier divisor. Through Proposition \ref{prop_deform_codim}, the corresponding family must not satisfy either left- or right-$1$-base change.

We first recall the computation of the Du Bois complex on affine cones by \cite{PS25}. Let $Y$ be a projective variety with an ample line bundle $\scrL$, and
\[
X=C(Y,\scrL) = \Spec \left(\bigoplus_{k\ge 0} H^0(Y, \scrL^k) \right)
\] 
the affine cone over $Y$ associated to $\scrL$.

\begin{theorem}[\protect{\cite[Theorem 3.1(1)]{PS25}}]\label{thm_PS_cone}
With the cone notation above, we have
\[
\Gamma(X, h^0(\uOmega^0_X)) \isom \bbC \oplus \bigoplus_{k\ge 1} \bbH^0(Y, \uOmega^0_Y \otimes \scrL^k),
\]
and for $i\ge 1$,
\[
\Gamma(X, h^i(\uOmega^0_X)) \isom \bigoplus_{k\ge 1} \bbH^i(Y, \uOmega^0_Y \otimes \scrL^k),
\]
and lastly for $p\ge 1$ and all $i\ge 0$,
\[
\Gamma(X, h^i(\uOmega^p_X)) \isom \bigoplus_{k\ge 1} \bbH^i(Y, \uOmega^p_Y \otimes \scrL^k) \oplus \bigoplus_{k\ge 1} \bbH^i(Y, \uOmega^{p-1}_Y \otimes \scrL^k).
\]
\end{theorem}

\begin{example}\label{ex_deform}
Let $S$ be a smooth projective surface, and $\scrA$ a very ample line bundle on $S$. Let $Y = S\times \bbP^1$, $\scrL=\scrA\boxtimes \scrO_{\bbP^1}(1)$. Let $L_1, L_2\in |\scrA|$ be two general members to form a pencil $S\dashrightarrow \bbP^1$. Then, let $D$ be its graph closure in $Y=S\times \bbP^1$, so that $D\in |\scrL|$. Indeed, if $L_i$ corresponds to section $s_i\in H^0(S, \scrA)$, then $D$ corresponds to the section $\sigma = us_1+ vs_2$, where $u,v \in H^0(\bbP^1,\scrO_{\bbP^1}(1))$ is a basis. Since $L_1, L_2$ are general, we may assume the common zero locus $L_1\cap L_2$ is reduced and $0$-dimensional. Therefore, $D\isom \Bl_{L_1\cap L_2}S$ is smooth and birational to $S$.

Finally, let $X= C(Y,\scrL)$ and $Z = C(D, \scrL|_D)$. We will study the Du Bois complex of $X$ and $Z$ from cohomological invariants on $S$. Observe that making $\scrA$ sufficiently positive will make both $X$ and $Z$ pre-$m$-Du Bois. The delicate balance is in controlling the positivity to make vanishing fail on $X$ but hold on $Z$.
\end{example}

We are interested in the cohomologies of the sheaf $\Omega^1_Y\otimes \scrL^k$ for $k\ge 0$. The differential $\Omega^1_Y$ decomposes as sums from the product $Y=S\times \bbP^1$ as
\[
\Omega^1_Y\isom (\Omega^1_S \boxtimes \scrO_{\bbP^1}) \oplus (\scrO_S\boxtimes \scrO_{\bbP^1}(-2)).
\]
Therefore, for $k\ge 0$,
\[
\Omega^1_Y\otimes \scrL^k\isom (\Omega^1_S \otimes \scrA^k) \boxtimes \scrO_{\bbP^1}(k) \oplus \scrA^k \boxtimes \scrO_{\bbP^1}(k-2).
\]
It is now easy to make $X$ not pre-$1$-Du Bois:

\begin{lemma}\label{lem_ex_deform_ambient}
In Example \ref{ex_deform}, suppose that
\[
H^1(S, \Omega^1_S\otimes \scrA) \ne 0.
\]
Then $X$ is not pre-$1$-Du Bois.
\end{lemma}

\begin{proof}
Taking cohomology using the K\"unneth formula, we see that
\[
H^1(Y, \Omega^1_Y\otimes \scrL) \isom H^1(S, \Omega^1_S \otimes \scrA) \otimes H^0(\bbP^1, \scrO_{\bbP^1}(1)) \isom H^1(S, \Omega^1_S \otimes \scrA)^{\oplus 2} \ne 0.
\]
Therefore by Theorem \ref{thm_PS_cone}, $h^1(\uOmega^1_X)\ne 0$, and $X$ is not pre-$1$-Du Bois.
\end{proof}

Making $Z$ pre-$1$-Du Bois requires more assumptions. The first two conditions below are expected in light of Theorem \ref{thm_PS_cone}, and the third technical condition exists to accommodate the required non-vanishing for $H^1(S, \Omega^1_S\otimes \scrA)$ in (2) for the previous Lemma.

\begin{lemma}\label{lem_ex_deform_divisor}
In Example \ref{ex_deform}, suppose the following holds:
\begin{enumerate}
\item $H^i(S, \scrA^k)= 0$ for all $i>0, k\ge 0$ (in particular $H^i(S, \scrO_S)=0$);
\item $H^i(S, \Omega^1_S\otimes \scrA^k)=0$ for all $i>0, k\ge 0$ except for possibly when $i=1$ and $k=0, 1$;
\item the map induced by multiplication of the two general sections $s_1, s_2\in H^0(S, \scrA)$ defining $L_1, L_2\in |\scrA|$ induces a surjection
\[
(\cdot s_1, \cdot s_2): H^1(S, \Omega^1_S)\to H^1(S, \Omega^1_S\otimes \scrA)^{\oplus 2}.
\]
\end{enumerate}
Then $Z$ is $1$-Du Bois.
\end{lemma}

\begin{proof}
Using the same K\"unneth formula computation in the previous Lemma and assumptions (1)-(2), we have
\[
H^i(Y, \Omega^1_Y\otimes \scrL^k)=0
\]
for every $k\ge 2$ and $i>0$. Moreover, for $i\ge2$
\[
H^i(Y,\Omega^1_Y)=H^i(Y,\Omega^1_Y\otimes \scrL)= 0,
\]
and
\begin{align*}
H^1(Y, \Omega^1_Y)\isom & H^1(S,\Omega^1_S)\oplus H^0(S, \scrO_S), \\
H^1(Y, \Omega^1_Y\otimes \scrL)\isom & H^1(S,\Omega^1_S\otimes \scrA)\otimes H^0(\bbP^1,\scrO_{\bbP^1}(1)).
\end{align*}
Note that $\Omega^1_Y$ is locally free. For any $k\ge 0$, we have the restriction sequence on $Y$ associated to $D$
\[
0\to \Omega^1_Y\otimes \scrL^{k-1} \to\Omega^1_Y\otimes \scrL^k \to (\Omega^1_Y\otimes \scrL^k )|_D\to 0,
\]
where the first map is given by multiplication by the defining section $\sigma \in H^0(Y, \scrO_Y(D))$. Using the associated long exact sequences on cohomology and the vanishing computed before, we have
\[
H^i(D, (\Omega^1_Y\otimes \scrL^k )|_D)= 0
\]
for all $i>0$ when $k\ge 2$, and for $i\ge2$ when $k=1$. When $k=i=1$, we consider the multiplication map $H^1(Y, \Omega^1_Y) \to H^1(Y, \Omega^1_Y\otimes \scrL)$ induced by the section $\sigma = us_1+vs_2 \in H^0(Y,\scrL)$. Using the K\"unneth decomposition before, on the non-trivial first component, this gives the map
\[
H^1(S,\Omega^1_S) \to H^1(S,\Omega^1_S\otimes \scrA)\otimes H^0(\bbP^1,\scrO_{\bbP^1}(1))\isom H^1(S,\Omega^1_S\otimes \scrA)^{\oplus 2}
\]
corresponding to multiplication by the two sections $s_1, s_2\in H^0(S,\scrA)$ used in constructing $\sigma$. Now by assumption (3), this map is surjective. Therefore, we obtain the vanishing
\[
H^1(D, (\Omega^1_Y\otimes \scrL )|_D)= 0.
\]
Since $D$ is smooth and birational to $S$, $H^i(D, \scrO_D)\isom H^i(S, \scrO_S)=0$ for $i>0$. Now, by K\"unneth and assumption (1), we also have
\[
H^i(Y, \scrL^k) = 0
\]
for $i>0, k\ge 0$. We perform the same calculation for the sequence
\[
0 \to \scrL^{k-1}\to \scrL^k \to \scrL^k|_D \to 0
\]
to see that
\[
H^i(D, \scrL^k|_D) =0.
\]
for $i>0, k\ge 1$. Finally, we consider the conormal sequence for the smooth divisor $D$ in $Y$:
\[
0 \to \scrL^{k-1}|_D \to (\Omega^1_Y\otimes \scrL^k )|_D \to \Omega^1_D \otimes \scrL^k|_D\to 0.
\]
Chasing the long exact sequence on cohomology once more, we conclude that
\[
H^i(D, \Omega^1_D \otimes \scrL^k|_D) = 0
\]
for $i>0, k\ge 1$. We remark that when $k\ge 2$, this directly follows, and for $k=1$ the only vanishing required is $H^1(D, \Omega^1_D \otimes \scrL|_D)\isom H^2(D, \scrO_D)=0$. Combining these vanishings with Theorem \ref{thm_PS_cone}, we see that $Z$ is pre-$1$-Du Bois. 

For cone $Z$ over smooth surface $D$, the only difference between pre-$1$-Du Bois and $1$-Du Bois is the vanishing of $H^1(D,\scrO_D)$, cf. \cite[Proposition 5.10]{SVV23}. Therefore, $Z$ is in fact $1$-Du Bois.
\end{proof}

\begin{lemma}\label{lem_ex_deform_cartier}
In Example \ref{ex_deform}, suppose $H^i(S, \scrA^k)=0$ for all $i>0, k\ge 0$. Then $Z$ is a Cartier divisor in $X$.
\end{lemma}

\begin{proof}
By the same proof as in Lemma \ref{lem_ex_deform_divisor}, the assumption implies that
\[
H^1(Y, \scrL^k)=0
\]
for $k\ge 0$. Now the restriction sequence is exact on global sections, which reads
\[
0\to H^0(Y, \scrL^{k-1}) \xrightarrow{\cdot \sigma} H^0(Y, \scrL^{k}) \to H^0(D, \scrL^{k}|_D)\to 0
\]
where as in Example \ref{ex_deform} the section $\sigma\in H^0(Y,\scrL)$ cuts out $D$. Summing over all $k\ge 0$,
\[
H^0(Z, \scrO_Z) = \bigoplus_{k\ge 0}H^0(D, \scrL^k|_D) \isom \bigoplus_{k\ge 0}H^0(Y, \scrL^k)/(\sigma) = H^0(X, \scrO_X)/(\sigma).
\]
This is precisely saying that $Z$ is cut out scheme-theoretically by the equation $\sigma$, now viewed as a section of $H^0(X, \scrO_X)$. Therefore $Z$ is Cartier, as desired.
\end{proof}

We are now ready to assemble our counterexample to deforming $1$-Du Bois singularities.

\begin{proposition}\label{prop_ex_deform}
Let $S\subset \bbP^3$ be a general smooth cubic surface and let $\scrA:= \scrO_S(2)$. Construct $Y=S\times \bbP^1$, $D$, $\scrL$, and cones $X, Z$ as in Example \ref{ex_deform}. Then $Z$ is a $1$-Du Bois Cartier divisor in $X$, but $X$ is not pre-$1$-Du Bois.
\end{proposition}

\begin{proof}
We need to show that the assumptions in Lemmas \ref{lem_ex_deform_ambient}, \ref{lem_ex_deform_divisor}, \ref{lem_ex_deform_cartier} are met for a general $S$.

As usual, we denote $\scrO_S(1)= \scrO_{\bbP^3}(1)|_S$. For any Fano hypersurface, Kodaira vanishing implies that 
\[
H^i(S, \scrO_S(k)) = 0
\]
for $i>0, k\ge 0$. We consider the Euler sequence on $\bbP^3$ restricted to $S$:
\[
0\to \Omega^1_{\bbP^3}(k)|_S \to \scrO_S(k-1)^{\oplus 4}\to \scrO_S(k) \to 0.
\]
Since the homogeneous coordinate ring of $S$ is generated in degree $1$, the second map above is surjective on global sections for $k\ge 1$. Using this and the vanishing of $H^i(S, \scrO_S(k))$ for $i>0$, we see that
\[
H^i(S, \Omega^1_{\bbP^3}(k)|_S) = 0
\]
for $i>0, k\ge 1$. Now we utilize the conormal sequence
\[
0\to \scrO_S(k-3) \to \Omega^1_{\bbP^3}(k)|_S \to \Omega^1_S(k) \to 0.
\]
For all $k\ge 1$, we see that 
\[
H^2(S, \Omega^1_S(k))=0,
\]
and by Serre duality
\[
H^1(S, \Omega^1_S(k))\isom H^2(S, \scrO_S(k-3))\isom H^0(S, \scrO_S(2-k))^\vee.
\]
Therefore, the higher cohomologies of $\Omega^1_S(k)$ vanish for all $k\ge 2$, except for the term
\[
H^1(S, \Omega^1_S(2))\isom H^0(S, \scrO_S)\isom \bbC.
\]
In addition, a cubic surface $S$ is rational, so by Hodge symmetry and duality we have
\[
h^2(S, \Omega^1_S)= h^{1,2}(S) = h^{1,0}(S) = 0.
\]
These address all the required vanishing and non-vanishing in Lemma \ref{lem_ex_deform_divisor}(1-2) and \ref{lem_ex_deform_ambient}. The only remaining piece is the surjectivity in Lemma \ref{lem_ex_deform_divisor}(3), which is established in Lemma \ref{lem_ex_deform_surject}.
\end{proof}

We can check the technical condition in Lemma \ref{lem_ex_deform_divisor}(3) directly in the following case.

\begin{lemma}\label{lem_ex_deform_surject}
Let $S\subset \bbP^3$ be a smooth cubic surface, and let $s_1, s_2\in H^0(S, \scrO_S(2))$ be two general sections. Then the multiplication map
\[
(\cdot s_1, \cdot s_2): H^1(S, \Omega^1_S)\to H^1(S, \Omega^1_S(2))^{\oplus 2}
\]
is surjective.
\end{lemma}

\begin{proof}
By naturality of multiplication, we have the following diagram of two conormal sequences
\[\begin{tikzcd}
0 \ar[r] &
\scrO_S(-3) \ar[r] \ar[d, "\cdot s_i"] &
\Omega^1_{\bbP^3}|_S \ar[r] \ar[d, "\cdot s_i"] &
\Omega^1_S \ar[r] \ar[d, "\cdot s_i"] &
0 \\
0 \ar[r] &
\scrO_S(-1) \ar[r] &
\Omega^1_{\bbP^3}|_S(2) \ar[r] &
\Omega^1_S(2) \ar[r] &
0.
\end{tikzcd}\]
Taking direct sum of the two induced the bottom rows for $i=1,2$, this yields a commutative diagram on cohomology groups
\[\begin{tikzcd}
H^1(S, \Omega^1_S) \ar[r, "\delta"] \ar[d, "\protect{(\cdot s_1, \cdot s_2)}"] &
H^2(S, \scrO_S(-3)) \ar[r, "\isom"] \ar[d, "\protect{(\cdot s_1, \cdot s_2)}"] &
H^0(S, \scrO_S(2))^\vee \ar[d, "\protect{(\cdot s_1, \cdot s_2)^\vee}"] \\
H^1(S, \Omega^1_S(2))^{\oplus 2} \ar[r, "\isom"] &
H^2(S, \scrO_S(-1))^{\oplus 2} \ar[r, "\isom"] &
(H^0(S, \scrO_S)^{\oplus 2})^\vee
\end{tikzcd}\]
where $(-)^\vee=\Hom_{\bbC}(-,\bbC)$ and the right square is obtained by Serre duality. As in Proposition \ref{prop_ex_deform}, we have $H^i(S, \Omega^1_{\bbP^3}(k)|_S)=0$ for $i>0, k\ge 1$, which justifies the first bottom map being an isomorphism. The right vertical map is dual to
\[
\bbC^2 \isom H^0(S, \scrO_S)^{\oplus 2} \to H^0(S, \scrO_S(2))
\]
sending each coordinate to $s_i$. It suffices to show that $\dim \im \delta \ge 2$. Indeed, viewing $\im \delta \subseteq H^0(S, \scrO_S(2))^\vee$, we may choose $s_i$ such that the image of the two generators under
\[
H^0(S, \scrO_S)^{\oplus 2} \to H^0(S, \scrO_S(2)) \to (\im \delta)^\vee
\]
is linearly independent. In this case, the composition is injective, and consequently the dual is surjective. Passing back through the isomorphism, this shows that $(\cdot s_1, \cdot s_2)\circ \delta$ maps onto $H^2(S, \scrO_S(-1))^{\oplus 2}$, which implies the desired surjectivity of the left vertical map.

We will now count the dimension of $\im \delta$. Using the restricted Euler sequence
\[
0\to \Omega^1_{\bbP^3}|_S \to \scrO_S(-1)^{\oplus 4}\to \scrO_S \to 0,
\]
and the vanishing from before, we see that 
\[
h^2(S, \Omega^1_{\bbP^3}|_S)=4\cdot h^2(S, \scrO_S(-1))=4\cdot h^0(S, \scrO_S)= 4.
\]
Now going back to the first conormal sequence, since $H^2(S, \Omega^1_S)=0$, we have 
\[
\dim \im \delta = h^2(S, \scrO_S(-3)) -4 = h^0(S, \scrO_S(2))-4 =10 - 4 = 6 \ge 2.
\]
Therefore we can choose $s_i$ to be independent. This concludes the proof.
\end{proof}

Finally, this example also implies the failure of base change for $1$-Du Bois fibers.

\begin{corollary}\label{cor_fail_bc_1DB}
There exists a family $\f:X \to B$ and a fiber $X_b$ for $b\in B$ such that $X_b$ is $1$-Du Bois, but does not satisfy either left- or right-$1$-base change.
\end{corollary}

\begin{proof}
Let $Z\subseteq X$ be as constructed from a general cubic surface as in Proposition \ref{prop_ex_deform}. Since $Z$ is a Cartier divisor, we may shrink to a neighborhood of $Z$ in $X$ to realize it as the fiber of a morphism $\f:X\to \bbA^1$, with $Z=X_b$. Therefore, the claim follows from Propositions \ref{prop_deform_codim} and \ref{prop_ex_deform}.
\end{proof}

\medskip

\bibliographystyle{alpha} 
\bibliography{ref}

@misc{SP,
  author       = {The {Stacks {P}roject {A}uthors}},
  title        = {The {S}tacks {P}roject},
  howpublished = {\url{https://stacks.math.columbia.edu}},
  year         = {2026},
}

@article {FL22,
    AUTHOR = {Friedman, Robert and Laza, Radu},
     TITLE = {Higher {D}u {B}ois and higher rational singularities},
      NOTE = {Appendix by Morihiko Saito},
   JOURNAL = {Duke Math. J.},
  FJOURNAL = {Duke Mathematical Journal},
    VOLUME = {173},
      YEAR = {2024},
    NUMBER = {10},
     PAGES = {1839--1881},
      ISSN = {0012-7094,1547-7398},
   MRCLASS = {14B05 (14B07 14F10)},
  MRNUMBER = {4776417},
       DOI = {10.1215/00127094-2023-0051},
       URL = {https://doi.org/10.1215/00127094-2023-0051},
}

@article {JKSY21,
    AUTHOR = {Jung, Seung-Jo and Kim, In-Kyun and Saito, Morihiko and Yoon,
              Youngho},
     TITLE = {Higher {D}u {B}ois singularities of hypersurfaces},
   JOURNAL = {Proc. Lond. Math. Soc. (3)},
  FJOURNAL = {Proceedings of the London Mathematical Society. Third Series},
    VOLUME = {125},
      YEAR = {2022},
    NUMBER = {3},
     PAGES = {543--567},
      ISSN = {0024-6115,1460-244X},
   MRCLASS = {14B05 (14F10 14F17 32S25)},
  MRNUMBER = {4480883},
MRREVIEWER = {Jihao\ Liu},
       DOI = {10.1112/plms.12464},
       URL = {https://doi.org/10.1112/plms.12464},
}

@book {Kol13,
    AUTHOR = {Koll\'{a}r, J\'{a}nos},
     TITLE = {Singularities of the minimal model program},
    SERIES = {Cambridge Tracts in Mathematics},
    VOLUME = {200},
      NOTE = {With a collaboration of S\'{a}ndor Kov\'{a}cs},
 PUBLISHER = {Cambridge University Press, Cambridge},
      YEAR = {2013},
     PAGES = {x+370},
      ISBN = {978-1-107-03534-8},
   MRCLASS = {14E30 (14B05)},
  MRNUMBER = {3057950},
MRREVIEWER = {Tommaso\ De Fernex},
       DOI = {10.1017/CBO9781139547895},
       URL = {https://doi.org/10.1017/CBO9781139547895},
}

@article {Kov00,
    AUTHOR = {Kov\'{a}cs, S\'{a}ndor},
     TITLE = {Rational, log canonical, {D}u {B}ois singularities. {II}.
              {K}odaira vanishing and small deformations},
   JOURNAL = {Compositio Math.},
  FJOURNAL = {Compositio Mathematica},
    VOLUME = {121},
      YEAR = {2000},
    NUMBER = {3},
     PAGES = {297--304},
      ISSN = {0010-437X,1570-5846},
   MRCLASS = {14F17 (14B05)},
  MRNUMBER = {1761628},
MRREVIEWER = {Gerhard\ Pfister},
       DOI = {10.1023/A:1001830707422},
       URL = {https://doi.org/10.1023/A:1001830707422},
}

@misc{Kov25,
      title={Complexes of differential forms and singularities: The injectivity theorem}, 
      author={Kov\'{a}cs, S\'{a}ndor J.},
      year={2026},
      note={arXiv:2505.09912v5},
      eprint={2505.09912},
      archivePrefix={arXiv},
      primaryClass={math.AG},
      url={https://arxiv.org/abs/2505.09912}, 
}

@misc{NN25,
      title={Higher {D}u {B}ois and Higher Rational Pairs}, 
      author={Haoming Ning and Brian Nugent},
      year={2025},
      note={arXiv:2510.19813},
      eprint={2510.19813},
      archivePrefix={arXiv},
      primaryClass={math.AG},
      url={https://arxiv.org/abs/2510.19813}, 
}

@incollection {KS16a,
    AUTHOR = {Kov\'acs, S\'andor J. and Schwede, Karl},
     TITLE = {Du {B}ois singularities deform},
 BOOKTITLE = {Minimal models and extremal rays ({K}yoto, 2011)},
    SERIES = {Adv. Stud. Pure Math.},
    VOLUME = {70},
     PAGES = {49--65},
 PUBLISHER = {Math. Soc. Japan, [Tokyo]},
      YEAR = {2016},
      ISBN = {978-4-86497-036-5},
   MRCLASS = {14B07 (14B05 14F17 14F18)},
  MRNUMBER = {3617778},
MRREVIEWER = {Jan\ Stevens},
       DOI = {10.2969/aspm/07010049},
       URL = {https://doi.org/10.2969/aspm/07010049},
}

@article {KS16b,
    AUTHOR = {Kov\'acs, S\'andor J. and Schwede, Karl},
     TITLE = {Inversion of adjunction for rational and {D}u {B}ois pairs},
   JOURNAL = {Algebra Number Theory},
  FJOURNAL = {Algebra \& Number Theory},
    VOLUME = {10},
      YEAR = {2016},
    NUMBER = {5},
     PAGES = {969--1000},
      ISSN = {1937-0652,1944-7833},
   MRCLASS = {14J17 (14B05 14B25 14D06 14E30)},
  MRNUMBER = {3531359},
MRREVIEWER = {I.\ Dolgachev},
       DOI = {10.2140/ant.2016.10.969},
       URL = {https://doi.org/10.2140/ant.2016.10.969},
}

@article {MOPW21,
    AUTHOR = {Musta\c{t}\u{a}, Mircea and Olano, Sebasti\'{a}n and Popa,
              Mihnea and Witaszek, Jakub},
     TITLE = {The {D}u {B}ois complex of a hypersurface and the minimal
              exponent},
   JOURNAL = {Duke Math. J.},
  FJOURNAL = {Duke Mathematical Journal},
    VOLUME = {172},
      YEAR = {2023},
    NUMBER = {7},
     PAGES = {1411--1436},
      ISSN = {0012-7094,1547-7398},
   MRCLASS = {14D07 (14B05 14F10 14F17 32S35)},
  MRNUMBER = {4583654},
       DOI = {10.1215/00127094-2022-0074},
       URL = {https://doi.org/10.1215/00127094-2022-0074},
}

@article {MP25,
    AUTHOR = {Musta\c{t}\u{a}, Mircea and Popa, Mihnea},
     TITLE = {On {$k$}-rational and {$k$}--{D}u {B}ois local complete
              intersections},
   JOURNAL = {Algebr. Geom.},
  FJOURNAL = {Algebraic Geometry},
    VOLUME = {12},
      YEAR = {2025},
    NUMBER = {2},
     PAGES = {237--261},
      ISSN = {2313-1691,2214-2584},
   MRCLASS = {14B05 (14F10 32S35)},
  MRNUMBER = {4869966},
MRREVIEWER = {Jan\ Stevens},
       DOI = {10.14231/ag-2025-007},
       URL = {https://doi.org/10.14231/ag-2025-007},
}

@misc{SVV23,
      title={On $k$-{D}u {B}ois and $k$-rational singularities}, 
      author={Wanchun Shen and Sridhar Venkatesh and Anh Duc Vo},
      year={2023},
      note={arXiv:2306.03977},
      eprint={2306.03977},
      archivePrefix={arXiv},
      primaryClass={math.AG}
}

@article {DB81,
    AUTHOR = {Du Bois, Philippe},
     TITLE = {Complexe de de {R}ham filtr\'e{} d'une vari\'et\'e{}
              singuli\`ere},
   JOURNAL = {Bull. Soc. Math. France},
  FJOURNAL = {Bulletin de la Soci\'et\'e{} Math\'ematique de France},
    VOLUME = {109},
      YEAR = {1981},
    NUMBER = {1},
     PAGES = {41--81},
      ISSN = {0037-9484},
   MRCLASS = {14C30},
  MRNUMBER = {613848},
MRREVIEWER = {J.\ H. M. Steenbrink},
       URL = {http://www.numdam.org/item?id=BSMF_1981__109__41_0},
}

@book {GNPP88,
    AUTHOR = {Guill\'en, F. and {Navarro Aznar}, V. and {Pascual Gainza}, P. and
              Puerta, F.},
     TITLE = {Hyperr\'esolutions cubiques et descente cohomologique},
    SERIES = {Lecture Notes in Mathematics},
    VOLUME = {1335},
      NOTE = {Papers from the Seminar on Hodge-Deligne Theory held in
              Barcelona, 1982},
 PUBLISHER = {Springer-Verlag, Berlin},
      YEAR = {1988},
     PAGES = {xii+192},
      ISBN = {3-540-50023-5},
   MRCLASS = {14F20 (14F40 32G20 32L20)},
  MRNUMBER = {972983},
MRREVIEWER = {Gerhard\ Pfister},
       DOI = {10.1007/BFb0085054},
       URL = {https://doi.org/10.1007/BFb0085054},
    shorthand = {GNPP88}
}

@article {Elk81,
    AUTHOR = {Elkik, Ren\'ee},
     TITLE = {Rationalit\'e{} des singularit\'es canoniques},
   JOURNAL = {Invent. Math.},
  FJOURNAL = {Inventiones Mathematicae},
    VOLUME = {64},
      YEAR = {1981},
    NUMBER = {1},
     PAGES = {1--6},
      ISSN = {0020-9910,1432-1297},
   MRCLASS = {14B05 (14J30)},
  MRNUMBER = {621766},
MRREVIEWER = {G.\ Horrocks},
       DOI = {10.1007/BF01393930},
       URL = {https://doi.org/10.1007/BF01393930},
}

@article {DJ74,
    AUTHOR = {Du {B}ois, Philippe and Jarraud, Pierre},
     TITLE = {Une propri\'et\'e{} de commutation au changement de base des
              images directes sup\'erieures du faisceau structural},
   JOURNAL = {C. R. Acad. Sci. Paris S\'er. A},
  FJOURNAL = {Comptes Rendus Hebdomadaires des S\'eances de l'Acad\'emie des
              Sciences. S\'erie A. Sciences Math\'ematiques},
    VOLUME = {279},
      YEAR = {1974},
     PAGES = {745--747},
      ISSN = {0302-8429},
   MRCLASS = {14F05},
  MRNUMBER = {376678},
MRREVIEWER = {J.\ S.\ Joel},
}

@misc{PSV24,
      title={Injectivity and Vanishing for the {D}u {B}ois Complexes of Isolated Singularities}, 
      author={Mihnea Popa and Wanchun Shen and Anh Duc Vo},
      year={2024},
      note={arXiv:2409.18019},
      eprint={2409.18019},
      archivePrefix={arXiv},
      primaryClass={math.AG},
      url={https://arxiv.org/abs/2409.18019}, 
}

@book {PS08,
    AUTHOR = {Peters, Chris A. M. and Steenbrink, Joseph H. M.},
     TITLE = {Mixed {H}odge structures},
    SERIES = {Ergebnisse der Mathematik und ihrer Grenzgebiete. 3. Folge. A
              Series of Modern Surveys in Mathematics [Results in
              Mathematics and Related Areas. 3rd Series. A Series of Modern
              Surveys in Mathematics]},
    VOLUME = {52},
 PUBLISHER = {Springer-Verlag, Berlin},
      YEAR = {2008},
     PAGES = {xiv+470},
      ISBN = {978-3-540-77015-2},
   MRCLASS = {14C30 (14D07 32G20 32J25 32S35)},
  MRNUMBER = {2393625},
MRREVIEWER = {Matt\ Kerr},
}

@book {EV92,
    AUTHOR = {Esnault, H\'el\`ene and Viehweg, Eckart},
     TITLE = {Lectures on vanishing theorems},
    SERIES = {DMV Seminar},
    VOLUME = {20},
 PUBLISHER = {Birkh\"auser Verlag, Basel},
      YEAR = {1992},
     PAGES = {vi+164},
      ISBN = {3-7643-2822-3},
   MRCLASS = {14F17 (14F40 32L10 32L20)},
  MRNUMBER = {1193913},
MRREVIEWER = {Marko\ Roczen},
       DOI = {10.1007/978-3-0348-8600-0},
       URL = {https://doi.org/10.1007/978-3-0348-8600-0},
}

@book{Kol23, 
    place={Cambridge}, 
    series={Cambridge Tracts in Mathematics}, 
    title={Families of Varieties of General Type}, 
    publisher={Cambridge University Press}, 
    author={Kollár, János}, 
    year={2023}, 
    collection={Cambridge Tracts in Mathematics}
}

@incollection {KT25,
    AUTHOR = {Kov\'acs, S\'andor J. and Taji, Behrouz},
     TITLE = {The relative {D}u {B}ois complex---on a question of {S}.
              {Z}ucker},
 BOOKTITLE = {Higher dimensional algebraic geometry---a volume in honor of
              {V}. {V}. {S}hokurov},
    SERIES = {London Math. Soc. Lecture Note Ser.},
    VOLUME = {489},
     PAGES = {151--162},
 PUBLISHER = {Cambridge Univ. Press, Cambridge},
      YEAR = {2025},
      ISBN = {978-1-009-39624-0},
   MRCLASS = {14B05 (14E30)},
  MRNUMBER = {4844630},
}

@article {Kov96,
    AUTHOR = {Kov\'acs, S\'andor J.},
     TITLE = {Smooth families over rational and elliptic curves},
   JOURNAL = {J. Algebraic Geom.},
  FJOURNAL = {Journal of Algebraic Geometry},
    VOLUME = {5},
      YEAR = {1996},
    NUMBER = {2},
     PAGES = {369--385},
      ISSN = {1056-3911,1534-7486},
   MRCLASS = {14J10 (14J15)},
  MRNUMBER = {1374712},
MRREVIEWER = {Marco\ Andreatta},
}

@incollection {Kov97,
    AUTHOR = {Kov\'acs, S\'andor J.},
     TITLE = {Relative de {R}ham complex for non-smooth morphisms},
 BOOKTITLE = {Birational algebraic geometry ({B}altimore, {MD}, 1996)},
    SERIES = {Contemp. Math.},
    VOLUME = {207},
     PAGES = {89--100},
 PUBLISHER = {Amer. Math. Soc., Providence, RI},
      YEAR = {1997},
      ISBN = {0-8218-0769-2},
   MRCLASS = {14F40 (14D07 14F05)},
  MRNUMBER = {1462927},
MRREVIEWER = {I.\ Dolgachev},
       DOI = {10.1090/conm/207/02722},
       URL = {https://doi-org.offcampus.lib.washington.edu/10.1090/conm/207/02722},
}

@misc{JK25,
      title={General base change for relative {D}u {B}ois complexes}, 
      author={Caleb Ji and Sándor Kovács},
      year={2025},
      note={arXiv:2508.02848},
      eprint={2508.02848},
      archivePrefix={arXiv},
      primaryClass={math.AG},
      url={https://arxiv.org/abs/2508.02848}, 
}

@Unpublished{CDO26,
    author       = {Chen, Qianyu and Dirks, Bradley and Olano, Sebasti\'an},
    title        = {FILTRATIONS ON LOCAL COHOMOLOGY, INJECTIVITY THEOREM, AND HIGHER SINGULARITIES},
    year         = {2026},
    note         = {Preprint},
}

@article {Hei08,
    AUTHOR = {Heitmann, Raymond C.},
     TITLE = {Lifting seminormality},
      NOTE = {Special volume in honor of Melvin Hochster},
   JOURNAL = {Michigan Math. J.},
  FJOURNAL = {Michigan Mathematical Journal},
    VOLUME = {57},
      YEAR = {2008},
     PAGES = {439--445},
      ISSN = {0026-2285,1945-2365},
   MRCLASS = {13F45},
  MRNUMBER = {2492461},
MRREVIEWER = {Hiroshi\ Tanimoto},
       DOI = {10.1307/mmj/1220879417},
       URL = {https://doi-org.offcampus.lib.washington.edu/10.1307/mmj/1220879417},
}

@article {Kov02,
    AUTHOR = {Kov\'acs, S\'andor J.},
     TITLE = {Logarithmic vanishing theorems and {A}rakelov-{P}arshin
              boundedness for singular varieties},
   JOURNAL = {Compositio Math.},
  FJOURNAL = {Compositio Mathematica},
    VOLUME = {131},
      YEAR = {2002},
    NUMBER = {3},
     PAGES = {291--317},
      ISSN = {0010-437X,1570-5846},
   MRCLASS = {14F17 (14J10)},
  MRNUMBER = {1905025},
MRREVIEWER = {Scott\ R.\ Nollet},
       DOI = {10.1023/A:1015592420937},
       URL = {https://doi-org.offcampus.lib.washington.edu/10.1023/A:1015592420937},
}

@article {KT23,
    AUTHOR = {Kov\'acs, S\'andor J. and Taji, Behrouz},
     TITLE = {Hodge sheaves underlying flat projective families},
   JOURNAL = {Math. Z.},
  FJOURNAL = {Mathematische Zeitschrift},
    VOLUME = {303},
      YEAR = {2023},
    NUMBER = {3},
     PAGES = {Paper No. 75, 34},
      ISSN = {0025-5874,1432-1823},
   MRCLASS = {14D07 (14D06 14E05 14E30 14F06)},
  MRNUMBER = {4552140},
MRREVIEWER = {Yajnaseni\ Dutta},
       DOI = {10.1007/s00209-023-03219-4},
       URL = {https://doi-org.offcampus.lib.washington.edu/10.1007/s00209-023-03219-4},
}

@article {KT24,
    AUTHOR = {Kov\'acs, S\'andor J. and Taji, Behrouz},
     TITLE = {Arakelov inequalities in higher dimensions},
   JOURNAL = {J. Reine Angew. Math.},
  FJOURNAL = {Journal f\"ur die Reine und Angewandte Mathematik. [Crelle's
              Journal]},
    VOLUME = {806},
      YEAR = {2024},
     PAGES = {115--145},
      ISSN = {0075-4102,1435-5345},
   MRCLASS = {14D05 (14D07 14E30)},
  MRNUMBER = {4685085},
MRREVIEWER = {Lo\"is\ Faisant},
       DOI = {10.1515/crelle-2023-0075},
       URL = {https://doi-org.offcampus.lib.washington.edu/10.1515/crelle-2023-0075},
}

@article {PS25,
    AUTHOR = {Popa, Mihnea and Shen, Wanchun},
     TITLE = {Du {B}ois complexes of cones over singular varieties, local
              cohomological dimension, and {$K$}-groups},
   JOURNAL = {Rev. Roumaine Math. Pures Appl.},
  FJOURNAL = {Revue Roumaine de Math\'ematiques Pures et Appliqu\'ees.
              Romanian Journal of Pure and Applied Mathematics},
    VOLUME = {70},
      YEAR = {2025},
    NUMBER = {1-2},
     PAGES = {133--155},
      ISSN = {0035-3965},
   MRCLASS = {14B05 (14C30 19E08)},
  MRNUMBER = {4876698},
       DOI = {10.59277/rrmpa.2025.133.155},
       URL = {https://doi-org.offcampus.lib.washington.edu/10.59277/rrmpa.2025.133.155},
}

@incollection {Kat89,
    AUTHOR = {Kato, Kazuya},
     TITLE = {Logarithmic structures of {F}ontaine-{I}llusie},
 BOOKTITLE = {Algebraic analysis, geometry, and number theory ({B}altimore,
              {MD}, 1988)},
     PAGES = {191--224},
 PUBLISHER = {Johns Hopkins Univ. Press, Baltimore, MD},
      YEAR = {1989},
      ISBN = {0-8018-3841-X},
   MRCLASS = {14F30 (14G20)},
  MRNUMBER = {1463703},
MRREVIEWER = {Adolfo\ Quir\'os},
}

@book {Ogu18,
    AUTHOR = {Ogus, Arthur},
     TITLE = {Lectures on logarithmic algebraic geometry},
    SERIES = {Cambridge Studies in Advanced Mathematics},
    VOLUME = {178},
 PUBLISHER = {Cambridge University Press, Cambridge},
      YEAR = {2018},
     PAGES = {xviii+539},
      ISBN = {978-1-107-18773-3},
   MRCLASS = {14D06 (14A20 14M25)},
  MRNUMBER = {3838359},
MRREVIEWER = {Howard\ M.\ Thompson},
       DOI = {10.1017/9781316941614},
       URL = {https://doi-org.offcampus.lib.washington.edu/10.1017/9781316941614},
}

\end{document}